\documentclass{amsart}
\usepackage{amscd,amssymb,amsmath,amsfonts,mathrsfs}

\usepackage{color}

\definecolor{darkgreen}{cmyk}{1,0,1,.2}
\definecolor{m}{rgb}{1,0.1,1}

\newcommand\ip{{\imath\pi}}

\newcommand\rank{\operatorname{rank}}
\newcommand\intsi{\stackrel{\scriptscriptstyle{o}}{\sigma}}
\newcommand\res{\operatorname{Res}}

\newcommand\Kp{{\mathscr{K}}}

\renewcommand\lq{\leqslant}
\newcommand\gq{\geqslant}

\newcommand\CC{\mathcal{C}}

\newcommand\A{\mathcal{A}}

\newcommand\diag{\operatorname{diag}}

\newcommand\I{\operatorname{I}_F^\Ga}
\newcommand\ds{\displaystyle}
\renewcommand\S{\mathbb{S}}

\newcommand\C{\mathbb C}
\newcommand\To{{\mathbb T}}
\newcommand\K{\mathcal{K}}
\newcommand\F{\mathcal{F}}
\newcommand\Si{\Sigma}
\newcommand\De{\Delta}

\newcommand\JR{\mathcal{J}_{\Ga}^{red}}

\newcommand\JJ{\mathcal{J}}
\renewcommand\H{\mathscr{H}}
\newcommand\de{\delta}
\newcommand\N{\mathbb N}

\newcommand\Id{\mathcal{I}d}

\newcommand\R{\mathbb R}
\newcommand\st{\text{ such that }}
\newcommand\Z{\mathbb Z}
\newcommand\T{\mathcal{T}}
\newcommand\TT{\mathcal{T}}
\renewcommand\O{\mathcal{O}}
\newcommand\MM{\mathcal{M}}

\newcommand\D{\mathcal D}
\newcommand\DD{\mathcal D}
\newcommand\G{\mathcal{G}}
\newcommand\ts{{\otimes}}
\newcommand\rt{{\rtimes}}
\newcommand\rtr{{\rtimes_{red}}}

\newcommand\si{\sigma}
\newcommand\la{\lambda}
\newcommand{\aeq}{\stackrel{(\lambda,h)}{\sim}}

\newcommand\Ga{{\Gamma}}

\newcommand\ga{{\gamma}}
\newcommand\al{{\alpha}}
\newcommand\lto{{\longrightarrow}}
\newcommand\defi{{\stackrel{\text{def}}{=\!=}}}

\newcommand\AG{{A\rtimes_{red}\Ga}}
\newcommand\M{{M}}
\newcommand\U{\operatorname{U}}
\renewcommand\P{\operatorname{P}}
\newcommand\ka{\kappa}
\newcommand\ue{\operatorname{U}_n^{\varepsilon,r}}
\newcommand\pe{\operatorname{P}_n^{\varepsilon,r}}
\newcommand\eps{\varepsilon}
\newcommand\erp{$\eps$-$r$-projection }
\newcommand\eru{$\eps$-$r$-unitary }

\theoremstyle{plain}
\newtheorem{theorem}{Theorem}[section]
\newtheorem{proposition}[theorem]{Proposition}
\newtheorem{corollary}[theorem]{Corollary}

\newtheorem{lemma}[theorem]{Lemma}
\newtheorem{definition}[theorem]{Definition}
\theoremstyle{definition}
\newtheorem{remark}[theorem]{Remark}
\newtheorem{example}[theorem]{Example}
\newtheorem{notation}[theorem]{Notation}
\begin{document}
\title[Quantitative  $K$-theory and the K\"unneth formula]{Quantitative  $K$-theory and the K\"unneth formula for operator algebras }

 \author[H. Oyono-Oyono]{Herv\'e Oyono-Oyono}
 \address{Universit\'e de Lorraine, Metz , France}
 \email{herve.oyono-oyono@math.cnrs.fr}
\author[G. Yu]{Guoliang Yu }
 \address{Texas A\&M University, USA and Shanghai Center for Mathematical Sciences, Fudan University, China}
 \email{guoliangyu@math.tamu.edu}
\thanks{Oyono-Oyono is partially supported by the ANR ``SingStar'' and Yu is partially supported by  a grant from the
US National Science Foundation and by the NSFC 11420101001}
\begin{abstract} In this paper, we apply quantitative operator  $K$-theory to develop an algorithm for computing $K$-theory for the class of filtered $C^\ast$-algebras with asymptotic finite nuclear decomposition. As a consequence, we  prove the K\"unneth formula for $C^\ast$-algebras in this class.
Our main technical tool is a quantitative Mayer-Vietoris sequence for $K$-theory of filtered $C^\ast$-algebras.
\end{abstract}
\maketitle

\begin{flushleft}{\it Keywords:
    quantitative operator $K$-theory,  K\"unneth formula, filtered $C^\ast$-algebras}

\medskip

{\it 2000 Mathematics Subject Classification: 19K35,46L80,58J22}
\end{flushleft}

\tableofcontents
\section{Introduction}
The concept of quantitative operator $K$-theory was first introduced in \cite{y2} and was set-up  in full generality for     filtered $C^\ast$-algebras in \cite{oy2}.
The aim of this paper is to develop techniques of quantitative operator $K$-theory to compute $K$-theory of $C^\ast$-algebras. In particular, we introduce a concept of finite asymptotic  nuclear decomposition    for filtered $C^\ast$-algebras. This $C^*$-algebraic concept  can be viewed as the noncommutative analogue of metric spaces with finite asymptotic dimension. We  establish  an algorithm for computing $K$-theory of  $C^\ast$-algebras with finite asymptotic  nuclear decomposition.  As a consequence, we prove the K\"unneth formula   for  $C^\ast$-algebras in this class. 
 The key idea is that the  finite asymptotic  nuclear decomposition  for a  $C^\ast$-algebra would allow us to compute quantitative $K$-theory for the $C^\ast$-algebra at each scale and the $K$-theory of the $C^\ast$-algebra can then be computed by letting the scale go to infinity. The main technical tool to compute quantitative $K$-theory is   controlled Mayer-Vietoris pairs   and attached  controlled six term exact sequences. 
 \medskip
 
 The paper is organized as follows. In Section 1, we give from \cite{oy2,oy3} an overview of quantitative $K$-theory.
 In Section 2 we introduce the concept of a controlled Mayer-Vietoris pair. This is the key ingredient to define later on  the class of $C^\ast$-algebras with finite asymptotic  nuclear decomposition. Typical example of these objects arise from Roe algebras and more generally from $C^*$-algebras of \'etale groupo\"\i ds.  In Section $3$ is stated for a controlled Mayer-Vietoris pair the controlled six term exact sequence. We apply this sequence to $K$-contractibility of  $C^*$-algebra. Section 4 is devoted to the quantitative K\"unneth formula, which implies the K\"unneth formula in $K$-theory. We show that examples of filtered $C^*$-algebras for which  the quantitative K\"unneth formula holds  are provided by crossed product of $C^*$-algebras by finitely generated groups  satisfying the  Baum-Connes conjecture with coefficients. The main result of these section in that the quantitative K\"unneth formula is stable under decomposition by controlled Mayer-Vietoris pair.
 In Section 5 we introduce the class of $C^\ast$-algebras with finite asymptotic  nuclear decomposition. This class is the class   of filtered $C^*$-algebras $A$ for which there exists a integer $n$ such that at every order  $r$, the $C^*$-algebra  $A$ decomposes  in $n$ steps   under controlled Mayer-Vietoris into $C^*$-algebras in the Bootstrap category.
 We then provide an  algorithm for computing $K$-theory for $C^*$-algebras in this class. In particular we show that these $C^*$-algebras satisfy  the quantitative K\"unneth formula. As a consequence, we show that the uniform Roe algebra of a discrete metric space with bounded geometry and with finite asymptotic dimension satisfies the Kunneth formula.
\section{Overview of quantitative $K$-theory}
In this section, we recall the basic concepts of quantitative $K$-theory for filtered $C^\ast$-algebras and collect  the main results of \cite{oy2} concerning quantitative $K$-theory that we shall use throughout  this paper. Roughly speaking,  quantitative $K$-theory is the abelian groups of $K$-theory elements at each given scale and $K$-theory can be obtained as an inductive  limit of quantitative $K$-groups (see Corollary \ref{cor-limit-quant-K-groups}).   The key point is that quantitative $K$-theory is in numerous geometric situations  more computable that usual $K$-theory. The structure of filtered $C^*$-algebras allows us to talk about the scale of elements in the $C^\ast$-algebras.
\begin{definition}
A filtered $C^*$-algebra $A$ is a $C^*$-algebra equipped with a family
$(A_r)_{r>0}$ of  closed linear subspaces   indexed by positive numbers such that:
\begin{itemize}
\item $A_r\subset A_{r'}$ if $r\lq r'$;
\item $A_r$ is stable by involution;
\item $A_r\cdot A_{r'}\subset A_{r+r'}$;
\item the subalgebra $\ds\bigcup_{r>0}A_r$ is dense in $A$.
\end{itemize}
If $A$ is unital, we also require that the identity  $1$ is an element of $ A_r$ for every positive
number $r$. The elements of $A_r$ are said to have {\bf propagation $r$}.
\end{definition}

Many examples of filtered  $C^*$-algebras arise from geometry.
Typical examples are provided  by Roe algebras, group and crossed-product algebra \cite{oy2}, groupoid algebras (see Section \ref{subsection-groupoid}) and finitely generated $C^*$-algebras.
Indeed all   filtered $C^*$-algebras are associated with a  {\bf length function}:
let $A$ be a $C^*$-algebra and assume that there exists  a function $\ell:A\to \R^+\cup\{\infty\}$ such that
\begin{itemize}
\item $\ell(0)=0$;
\item $\ell(x+y)\lq \max\{\ell(x),\ell(y)\}$ for all $x$ and $y$ in $A$;
\item $\ell(x^*)=\ell(x)$ for all $x$ in $A$;
\item $\ell(\la x)=\ell(x)$ for all $x$ in $A$ and $\la$ in $\C\setminus\{0\}$;
\item $\ell(xy)\lq\ell(x)+\ell(y)$ for all $x$ and $y$ in $A$;
\item  $\{x\in A \text{ such that }\ell(x)\lq r\}$ is closed in $A$ for all positive number $r$;
\item  $\bigcup_{r>0}\{x\in A \text{ such that }\ell(x)\lq r\}$  is dense in $A$.
\end{itemize}
If we set $A_r=\{x\in A\text{ such that }\ell(x)\lq r\}$, then $A$ is filtered by $(A_r)_{r>0}$.
It is straightforward to show  that the category of filtered $C^*$-algebras is equivalent to the category of $C^*$-algebras equipped with a length function. In essence, we study geometric $C^*$-algebras just as group theorists study geometric  group theory.

\medskip
 Let $A$ and $A'$ be respectively  $C^*$-algebras filtered by
$(A_r)_{r>0}$ and  $(A'_r)_{r>0}$. A homomorphism of $C^*$
-algebras $\phi:A\lto A'$
is a {\bf filtered homomorphism} (or a {\bf homomorphism of  filtered $C^*$-algebras}) if  $\phi(A_r)\subset A'_{r}$ for any
positive number $r$.
If $A$ is not unital, let us denote by $\tilde{A}$ its unitarization, i.e.,
$$\tilde{A}=\{(x,\lambda);\,x\in A\,,\lambda\in \C\}$$  with the product $$(x,\lambda)(x',\lambda')=(xx'+\lambda x'+\lambda' x)$$ for all $(x,\lambda)$ and $(x',\lambda')$ in $\tilde{A}$. Then ${\tilde{A}}$ is filtered with
$${\tilde{A}_r}=\{(x,\lambda);\,x\in {A}_{r}\,,\lambda\in \C\}.$$
We also define $\rho_A:\tilde{A}\to\C;\, (x,\lambda)\mapsto \lambda$.

\subsection{Definition of quantitative $K$-theory}

Let $A$ be a unital filtered $C^*$-algebra. For any  positive
numbers $r$ and $\eps$ with $\eps<1/4$, we call
\begin{itemize}
\item an element $u$ in $A$  an $\eps$-$r$-unitary if $u$
  belongs to $A_r$,  $\|u^*\cdot
  u-1\|<\eps$
and  $\|u\cdot u^*-1\|<\eps$. The set of $\eps$-$r$-unitaries on $A$ will be denoted by $\operatorname{U}^{\varepsilon,r}(A)$.
\item an element $p$ in $A$   an $\eps$-$r$-projection    if $p$
  belongs to $A_r$,
  $p=p^*$ and  $\|p^2-p\|<\eps$. The set of $\eps$-$r$-projections on $A$ will be denoted by $\operatorname{P}^{\varepsilon,r}(A)$.
\end{itemize} Then $\eps$ is the called the control and $r$ is called the propagation of the $\eps$-$r$-projection or of the $\eps$-$r$-unitary.
Notice that an $\eps$-$r$-unitary is invertible, and that if $p$ is an \erp in $A$, then it has a spectral gap around $1/2$ and then gives rise by functional calculus to a  projection $\ka_{0}(p)$  in  $A$ such that
 $\|p-\ka_{0}(p)\|< 2\eps$.

Recall the following from \cite[Lemma 1.7]{oy2}
the following result that  will be  used quite  extensively throughout  the paper.
\begin{lemma}\label{lemma-almost-closed}
 Let $A$ be a  $C^*$-algebra filtered by $(A_r)_{r>0}$.
\begin{enumerate}

\item  If $p$  is an \erp in $A$ and  $q$ is a self-adjoint element of $A_r$ such
  that
   $\|p-q\|<\frac{\eps-\|p^2-p\|}{4}$, then $q$ is an  $\eps$-$r$-projection. In
   particular, if $p$ is an \erp in $A$ and if $q$ is a self-adjoint  element in
   $A_r$ such that $\|p-q\|< \eps$, then $q$ is a
   $5\eps$-$r$-projection in $A$ and $p$ and $q$ are connected by a
   homotopy of $5\eps$-$r$-projections.
\item  If $A$ is unital and if $u$ is an \eru and  $v$ is an element of $A_r$ such that
  $\|u-v\|<\frac{\eps-\|u^*u-1\|}{3}$, then $v$ is an $\eps$-$r$-unitary. In
  particular, if $u$ is an \eru and $v$ is an element of $A_r$ such that
  $\|u-v\|<\eps$, then $v$ is an $4\eps$-$r$-unitary  in $A$ and $u$
  and $v$ are connected by a homotopy of $4\eps$-$r$-unitaries.
 \item  If $p$  is a projection  in $A$ and  $q$ is a self-adjoint element of $A_r$ such
  that
   $\|p-q\|<\frac{\eps}{4}$, then $q$ is an  $\eps$-$r$-projection. 
   \item  If $A$ is unital and if $u$ is a unitary in $A$  and  $v$ is an element of $A_r$ such that
  $\|u-v\|<\frac{\eps}{3}$, then $v$ is an $\eps$-$r$-unitary.
\end{enumerate}
\end{lemma}
Let us also mention the following result concerning homotopy up to stabilization of products of $\eps$-$r$-unitaries \cite[Corollary 1.8]{oy2}.
 \begin{lemma}\label{cor-example-homotopy} Let $\eps$ and $r$ be positive numbers with $\eps<1/12$ and let $A$ be a  unital filtered $C^*$-algebra.
 \begin{enumerate}
 \item Let $u$ and $v$ be  $\eps$-$r$-unitaries in $A$, then $\diag(u,v)$ and
 $\diag(uv,1)$ are homotopic as   $3\eps$-$2r$-unitaries  in   $M_2(A)$; \item Let $u$ be an $\eps$-$r$-unitary in  $A$, then $\diag(u,u^*)$ and
 $I_2$ are homotopic as   $3\eps$-$2r$-unitaries  in   $M_2(A)$.
 \end{enumerate}

 \end{lemma}

 For purpose of rescalling the control and the propagation of an $\eps$-$r$-projection or of an $\eps$-$r$-unitary, we introduce the  following   concept of control pair.

 \begin{definition} A control pair  is a pair $(\lambda,h)$, where
\begin{itemize}
 \item $\lambda$ is a positive number with $\lambda >1$;
\item  $h:(0,\frac{1}{4\lambda})\to (1,+\infty);\, \eps\mapsto h_\eps$  is a map such that there exists a non-increasing map
$g:(0,\frac{1}{4\lambda})\to (1,+\infty)$, with $h\lq g$.
\end{itemize}\end{definition}
 The set of control pairs is equipped with a partial order:
 $(\lambda,h)\lq (\lambda',h')$ if $\lambda\lq\lambda'$ and $h_\eps\lq h'_\eps$
for all $\eps$ in $(0,\frac{1}{4\lambda'})$.

Recall the following from \cite[Corollary 1.31]{oy2}.
\begin{proposition}\label{prop-conjugate} There exists a control pair
  $(\alpha,k)$    such  that the following holds:

For any    unital
   filtered $C^*$-algebra $A$, any  positive numbers $\eps$ and $r$ with $\eps<\frac{1}{4\alpha}$ and any
  homotopic $\eps$-$r$-projections $q_0$ and $q_1$
   in $\pe(A)$, then there is for some integers $k$ and $l$  an $\alpha\eps$-$k_{\eps}r$-unitary
$W$ in  $\U^{\alpha\eps,k_{\eps}r}_{n+k+l}(A)$ such that
$$\|\diag(q_0,I_k,0_l)-W\diag(q_1,I_k,0_l)W^*\|<\alpha\eps.$$
\end{proposition}


 For any  $n$ integer, we set  $\ue(A)=\operatorname{U}^{\varepsilon,r}(M_n(A))$ and
$\pe(A)=\operatorname{P}^{\varepsilon,r}(M_n(A))$.
For any  unital filtered $C^*$-algebra $A$, any
 positive  numbers $\eps$ and $r$ and  any positive integer $n$, we consider inclusions
$$\P_n^{\eps,r}(A)\hookrightarrow \P_{n+1}^{\eps,r}(A);\,p\mapsto
\begin{pmatrix}p&0\\0&0\end{pmatrix}$$ and
$$\U_n^{\eps,r}(A)\hookrightarrow \U_{n+1}^{\eps,r}(A);\,u\mapsto
\begin{pmatrix}u&0\\0&1\end{pmatrix}.$$ This allows us to  define
 $$\U_{\infty}^{\eps,r}(A)=\bigcup_{n\in\N}\ue(A)$$ and
$$\P_{\infty}^{\eps,r}(A)=\bigcup_{n\in\N}\pe(A).$$

For a unital filtered $C^*$-algebra $A$, we define the
following
equivalence relations on $\P_\infty^{\eps,r}(A)\times\N$ and on  $\U_\infty^{\eps,r}(A)$:
\begin{itemize}
\item if $p$ and $q$ are elements of $\P_\infty^{\eps,r}(A)$, $l$ and
  $l'$ are positive integers, $(p,l)\sim(q,l')$ if there exists a
  positive integer $k$ and an element $h$ of
  $\P_\infty^{\eps,r}(A[0,1])$ such that $h(0)=\diag(p,I_{k+l'})$
and $h(1)=\diag(q,I_{k+l})$.
\item if $u$ and $v$ are elements of $\U_\infty^{\eps,r}(A)$, $u\sim v$ if
  there exists an element $h$ of
  $\U_\infty^{3\eps,2r}(A[0,1])$ such that $h(0)=u$
and $h(1)=v$.
\end{itemize}

If $p$ is an  element of $\P_\infty^{\eps,r}(A)$ and  $l$ is an integer, we
denote by $[p,l]_{\eps,r}$ the equivalence class of $(p,l)$ modulo  $\sim$
and if $u$ is an element of $\U_\infty^{\eps,r}(A)$ we denote by
$[u]_{\eps,r}$ its  equivalence class  modulo  $\sim$.
\begin{definition} Let $r$ and $\eps$ be positive numbers with
  $\eps<1/4$.
We define:
\begin{enumerate}
\item $K_0^{\eps,r}(A)=\P_\infty^{\eps,r}(A)\times\N/\sim$ for $A$ unital and
$$K_0^{\eps,r}(A)=\{[p,l]_{\eps,r}\in \P^{\eps,r}({\tilde{A}})\times\N/\sim \st
\rank \kappa_0(\rho_{A}(p))=l\}$$ for $A$ non unital ($\kappa_0(\rho_{A}(p))$ being the spectral projection associated to $\rho_A(p)$);
\item $K_1^{\eps,r}(A)=\U_\infty^{\eps,r}({\tilde{A}})/\sim$, with $\tilde{A}=A$ if $A$ is already unital.
\end{enumerate}
\end{definition}

 Then $K_0^{\eps,r}(A)$ turns to be an abelian group \cite[Lemma 1.15]{oy2}, where
 $$[p,l]_{\eps,r}+[p',l']_{\eps,r}=[\diag(p,p'),l+l']_{\eps,r}$$  for any  $[p,l]_{\eps,r}$ and $[p',l']_{\eps,r}$ in $K_0^{\eps,r}(A)$. According to Corollary \ref{cor-example-homotopy},  $K_1^{\eps,r}(A)$ is
 equipped with a structure of abelian group such that
$$[u]_{\eps,r}+[u']_{\eps,r}=[\diag(u,v)]_{\eps,r},$$ for
any  $[u]_{\eps,r}$ and $[u']_{\eps,r}$ in $K_1^{\eps,r}(A)$.

  \medskip

  Recall from \cite[Corollaries 1.19 and 1.21]{oy2} that
for any positive numbers  $r$ and $\eps$ with $\eps<1/4$, then
$$K_0^{\eps,r}(\C)\to\Z;\,[p,l]_{\eps,r}\mapsto \rank\kappa_0(p)-l$$
is an isomorphism and
$K_1^{\eps,r}(\C)=\{0\}$.

 \medskip

 We have for any filtered $C^*$-algebra $A$ and any  positive numbers
$r$, $r'$, $\eps$ and $\eps'$  with
  $\eps\lq\eps'<1/4$ and $r\lq r'$  natural group homomorphisms called the structure maps:
\begin{itemize}
\item $\iota_0^{\eps,r}:K_0^{\eps,r}(A)\lto K_0(A);\,
[p,l]_{\eps,r}\mapsto [\kappa_0(p)]-[I_l]$ (where  $\kappa_0(p)$ is the spectral projection associated to $p$);
\item $\iota_1^{\eps,r}:K_1^{\eps,r}(A)\lto K_1(A);\,
  [u]_{\eps,r}\mapsto [u]$  ;
\item $\iota_*^{\eps,r}=\iota_0^{\eps,r}\oplus \iota_1^{\eps,r}$;
\item $\iota_0^{\eps,\eps',r,r'}:K_0^{\eps,r}(A)\lto K_0^{\eps',r'}(A);\,
[p,l]_{\eps,r}\mapsto [p,l]_{\eps',r'};$
\item $\iota_1^{\eps,\eps',r,r'}:K_1^{\eps,r}(A)\lto K_1^{\eps',r'}(A);\,
  [u]_{\eps,r}\mapsto [u]_{\eps',r'}$.
\item $\iota_*^{\eps,\eps',r,r'}=\iota_0^{\eps,\eps',r,r'}\oplus\iota_1^{\eps,\eps',r,r'}$
\end{itemize}
If some of the indices $r,r'$ or $\eps,\eps'$ are equal, we shall not
repeat it in $\iota_*^{\eps,\eps',r,r'}$.
The structures maps satisfy  the obvious compatibilitity rules with respect to compositions. We  have in the formalism of quantitative $K$-theory the analogue of the standard form for a $K$-theory class.
 \begin{lemma}\label{lemma-almost-canonical-form}
Let $A$  be a non unital filtered $C^*$-algebra. Let $\eps$ and $s$ be positive numbers with $\eps<\frac{1}{36}$. Then for any $x$ in
$K_0^{\eps,s}(A)$, there exist
\begin{itemize}
\item two integers $k$ and $n$ with $k\lq n$;
\item $q$ a  $9\eps$-$s$-unitary in $M_n(\widetilde{A})$
\end{itemize}
such that $\rho_A(q)=\diag(I_k,0)$ and $x=[q,k]_{9 \eps, s}$ in $K_0^{9\eps,s}(A)$.\end{lemma}
\begin{proof}
Let $x$ be an element in $K^{\eps,s}_0({A})$, let $p$ be an $\eps$-$s$-projection in some $M_n(\tilde{A})$ and let $k$  be an integer with $\rank \kappa_0(\rho_A(q))=k$ and such that
$x=[p,k]_{\eps,s}$. We can assume without loss of generality that $k\leq n$.  Let  $U$ be a unitary in $M_n(\C)$ such that  
$U\cdot \kappa_0(\rho_A(q))\cdot U^*=\diag(I_k,0)$. Since $U$ is homotopic to $I_n$ as a unitary of $M_n(\C)$, we see that $U\cdot p\cdot U^*$ and $p$ are homotopic as
$\eps$-$s$-projections  in $M_n(\tilde{A})$. Set then $$q'=U\cdot q\cdot U^*+  \diag(I_k,0) -U\cdot  \rho_A(q)\cdot U^*.$$      Since 
\begin{eqnarray*}
\|q'-U\cdot q\cdot U^*\|&=&\| U\cdot (\kappa_0( \rho_A(p))- \rho_A(p))\cdot U^*\|\\
&=&\|\kappa_0( \rho_A(p))- \rho_A(p) \|\\&<&2 \eps,
\end{eqnarray*} we get according to Lemma \ref{lemma-almost-closed} that $q$ and $q'$ are homotopic $9\eps$-$s$-projections in  $M_n(\tilde{A})$. Since $\rho_A(q')=\diag(I_k,0)$, we get the result.\end{proof}

We have a similar result in the odd case.
\begin{lemma}\label{lemma-almost-canonical-form-odd}
Let $A$  be a non unital  $C^*$-algebra filtered by $(A_s)_{s>0}$. Let $\eps$ and $s$ be positive numbers with $\eps<\frac{1}{84}$.
\begin{enumerate}
\item for any $x$ in $K^{\eps,s}_1(A)$, there exists an  $21\eps$-$s$-unitary $u$ in $M_n(\tilde{A})$,
such that $\rho_A(u)=I_n$ and $\iota^{\eps,21\eps,s}_1(x)=[u]_{21\eps,s}$  in $K_1^{21\eps,s}(A)$;
\item if $u$ and $v$ are two of $\eps$-$r$-unitaries in  $M_n(\tilde{A})$ such that $\rho_A(u)=\rho_A(v)=I_n$ and 
$[u]_{\eps,s}=[v]_{\eps,s}$  in $K_1^{\eps,r}(A)$, then there exists an integer $k$ and  a homotopy $(w_t)_{t\in[0,1]}$ of
$21\eps$-$s$-unitaries of $M_{n+k}(\tilde{A})$ between $\diag(u,I_k)$ and $\diag(v,I_k)$ such that $\rho_A(w_t)=I_{n+k}$ for every $t$ in $[0,1]$.
\end{enumerate}
\end{lemma}
\begin{proof}
Let $v$ be an $\eps$-$r$-unitary in some $M_n(\tilde{A})$ such that $x=[v]_{\eps,r}$.
According to \cite[Remark 1.4]{oy2}, we have that $\|\rho_A(v)^{-1}-\rho_A(v^*)\|<2\eps$. In particular, $\rho_A(v)^{-1}$ is a $7\eps$-$r$unitary and $\rho_A(v)^{-1}$ is homotopic to $I_n$ as a $7\eps$-$s$-unitary of $M_n(\C)$, where $\C$ is provided with the trivial filtration \cite[Lemma 1.20]{oy2}.
Hence, if we set $u=\rho_A(v)^{-1} v$, then $u$ is a $21\eps$-$s$ unitary of $M_n(\tilde{A})$ such that $\rho_A(u)=I_n$ and $u$ and $v$ are homotopic as $21\eps$-$s$ unitaries of $M_n(\tilde{A})$. Hence  we have the equality 
$$\iota^{\eps,21\eps,s}_1(x)=[v]_{21\eps,s}=[u]_{21\eps,s}.$$
\end{proof}

%
%
%

 If  $\phi:A\to B$ is a  homomorphism  filtered $C^*$-algebras, then since $\phi$ preserves $\eps$-$r$-projections
 and $\eps$-$r$-unitaries, it obviously induces  for any positive number $r$ and any
$\eps\in(0,1/4)$ a
group homomorphism $$\phi_*^{\eps,r}:K_*^{\eps,r}(A)\longrightarrow
K_*^{\eps,r}(B).$$  Moreover quantitative $K$-theory is homotopy invariant with respect to homotopies that preserves propagation
\cite[Lemma 1.26]{oy2}.
There is also a quantitative version of Morita equivalence \cite[Proposition 1.28]{oy2}.
\begin{proposition}\label{prop-morita}
If $A$ is a filtered algebra and $\H$ is a separable Hilbert space, then the homomorphism
$$A\to \Kp(\H)\otimes A;\,a\mapsto \begin{pmatrix}a&&\\&0&\\&&\ddots\end{pmatrix}$$
induces a ($\Z_2$-graded) group isomorphism (the Morita equivalence)
$$\MM_A^{\eps,r}:K_*^{\eps,r}(A)\to K_*^{\eps,r}(\Kp(\H)\otimes A)$$
for any positive number $r$ and any
$\eps\in(0,1/4)$.
\end{proposition}

The following observation establishes a connection between quantitative $K$-theory and classical $K$-theory (see \cite[Remark 1.17]{oy2}).
\begin{proposition}\label{proposition-approximation-K-th}\
\begin{enumerate}
\item Let $A$ be a filtered $C^\ast$-algebra. For any positive number $\eps$ with $\eps< \frac{1}{4}$ and  any element $y$ of $K_*(A)$,  there exists  a positive number $r$ and an element  $x$ of  $K_*^{\eps,r}(A)$ such that $\iota_{*}^{\eps,r}(x)=y$;
\item There exists a positive number $\lambda_0$ such that for any $C^*$-algebra $A$, any positive numbers $\eps$  and $r$ with $\eps< \frac{1}{4\lambda_0}$ and any element $x$ of  $K_*^{\eps,r}(A)$ for which  $\iota_*{\eps,r}(x)=0$ in $K_*(A)$, then there exists a positive number $r'$ with $r'\gq r$ such that $\iota_*^{\eps,\lambda_0\eps,r,r'}(x)=0$ in $K_*^{\lambda\eps,r'}(A)$.
\end{enumerate}
\end{proposition}
As a consequence , we get the following approximation property.
\begin{corollary}\label{cor-limit-quant-K-groups}
Let $\lambda_0$ be as in Proposition \ref{proposition-approximation-K-th}. Then for any positive number $\eps$ with $\eps<\frac{1}{4\la_0}$ and for any filtered $C^*$-algebra $A$, then
$$K_*(A)=\lim_r \iota^{\eps,\la_0\eps,r}_*(K_*^{\eps,r}(A)).$$
\end{corollary}

\subsection{Quantitative objects}
In order to study the functorial properties of quantitative $K$-theory, we introduced in \cite{oy3} the concept of quantitative object.

\begin{definition}
 A quantitative object is a family   $\O=(O^{\eps,r})_{0<\eps<1/4,r>0}$ of abelian groups,
together with a family of
group homomorphisms  $$\iota_{\O}^{\eps,\eps',r,r'}:O^{\eps,r}\to O^{\eps',r'}$$ for $0<\eps\lq \eps'<1/4$ and $0<r\leq r'$ called the structure maps such that
\begin{itemize}
 \item $\iota_{\O}^{\eps,\eps,r,r}=Id_{O^{\eps,r}}$ for any  $0<\eps<1/4$ and $r>0$;
 \item $\iota_{\O}^{\eps',\eps'',r',r''}\circ\iota_{\O}^{\eps,\eps',r,r'}=\iota_{\O}^{\eps,\eps'',r,r''}$ for any $0<\eps\lq \eps'\lq \eps''<1/4$ and $0<r\lq r'\lq r''$.
\end{itemize}
\end{definition}
\begin{example}\
Our prominent example will be of course quantitative $K$-theory $\K_*(A)=(K^{\eps,r}_*(A))_{0<\eps<1/4,r>0}$ of a filtered $C^*$-algebra  $A$  with structure maps $\iota^{\eps,\eps',r,r'}_*:K^{\eps,r}_*(A)\lto K^{\eps',r'}_*(A)$  for $0<\eps\lq \eps'<1/4$ and $0<r\lq r'$.
%
\end{example}

\subsection{Controlled morphisms}

In this subsection, we recall from  \cite[Section 2]{oy2}  the relevant notion of morphisms in the framework of quantitative objects.
\begin{definition}
 Let $(\lambda,h)$ be a control pair and let  $\O=(O^{\eps,r})_{0<\eps<1/4,r>0}$ and $\O'=(O'^{\eps,r})_{0<\eps<1/4,r>0}$ be
quantitative objects. A $(\lambda,h)$-controlled morphism
$$\F:\O\to\O'$$ is a family $\F=(F^{\eps,r})_{0<\eps< \frac{1}{4\lambda},r>0}$ of  groups homomorphisms
 $$F^{\eps,r}:\O^{\eps,r} \to \O'^{\lambda\eps,h_\eps r}$$ such that for any positive
numbers $\eps,\,\eps',\,r$ and $r'$ with
$0<\eps\lq\eps'< \frac{1}{4\lambda},\, r\lq r'$ and $h_\eps r\lq h_{\eps'}r'$, we have
$$F^{\eps',r'}\circ \iota_\O^{\eps,\eps',r,r'}=\iota_{\O'}^{\lambda\eps,\lambda\eps', h_\eps r,h_{\eps'}r'}\circ F^{\eps,r}.$$
\end{definition}
When  it is not necessary  to specify the control pair, we will just say that $\F$ is a controlled morphism.
If $\O=(O^{\eps,r})_{0<\eps<1/4,r>0}$ is a quantitative object, let us define the identity $(1,1)$-controlled morphism
$$\Id_\O=(Id_{O^{\eps,r}})_{0<\eps< 1/4,r>0}:\O\to\O.$$ Recall that if $A$ and $B$ are filtered
$C^*$-algebras and if $\F:\K_*(A)\to\K_*(B)$ is a $(\lambda,h)$-controlled morphism, then $\F$ induces
 a  morphism $F:K_*(A)\to K_*(B)$ uniquely defined
by $\iota^{\eps,r}_*\circ F^{\eps,r}=F\circ\iota^{\eps,r}_*$.

In some situation,  as  for instance  control boundary maps of controlled Mayer-Vietoris pair (see Section \ref{subs-section-boundary}), we deal with  family 
$F^{\eps,r}:\O^{\eps,r} \to \O'^{\lambda\eps,h_\eps r}$ of group morphism  defined indeed only up to  a certain order.
\begin{definition}
 Let $(\lambda,h)$ be a control pair, let  $\O=(O^{\eps,s})_{0<\eps<1/4,s>0}$ and $\O'=(O'^{\eps,s})_{0<\eps<1/4,s>0}$ be
quantitative objects and let $r$ be a positive number.
A $(\lambda,h)$-controlled morphism of order $r$ 
$$\F:\O\to\O'$$ is a family $\F=(F^{\eps,s})_{0<\eps< \frac{1}{4\lambda},0<s<\frac{r}{h_\eps}}$ of  groups homomorphisms
 $$F^{\eps,s}:\O^{\eps,s} \to \O'^{\lambda\eps,h_\eps s}$$ such that for any positive
numbers $\eps,\,\eps',\,s$ and $s'$ with
$0<\eps\lq\eps'< \frac{1}{4\lambda},\,s\lq s'<r$ and $h_\eps s\lq h_{\eps'}s'$, we have
$$F^{\eps',s'}\circ \iota_\O^{\eps,\eps',s,s'}=\iota_{\O'}^{\lambda\eps,\lambda\eps', h_\eps s,h_{\eps'}s'}\circ F^{\eps,s}.$$
\end{definition}

 $(\lambda,h)$ and $(\lambda',h')$ are two control pairs, define
$$h*h': (0,\frac{1}{4\lambda\lambda'})\to (1,+\infty);\, \eps\mapsto h_{\lambda'\eps}h'_\eps.$$
Then $(\lambda\lambda',h*h')$  is again  a control pair.
 Let   $\O=(O^{\eps,r})_{0<\eps<1/4,r}$, $\O'=(O'^{\eps,r})_{0<\eps<1/4,r}$ and $\O''=(O''^{\eps,r})_{0<\eps<1/4,r}$
be quantitative objects, let $$\F=(F^{\eps,r})_{0<\eps<\frac{1}{4\alpha_\F},r>0}:\O\to\O'$$
 be
a $(\alpha_\F,k_\F)$-controlled morphism, and let
 $$\G=(G^{\eps,r})_{0<\eps<\frac{1}{4\alpha_\G},r>0}:\O'\to\O''$$
be a $(\alpha_\G,k_\G)$-controlled morphism. Then $\G\circ\F:\O\to \O''$ is the $(\alpha_\G\alpha_F,k_\G*k_\F)$-controlled
morphism defined by the family
\begin{equation}\label{equ-composition}(G^{\alpha_\F\eps,k_{\F,\eps}r}\circ F^{\eps,r}:O^{\eps,r}\to
{O''}^{\alpha_\G\alpha_\F\eps,k_{\F,\eps}k_{\G,\alpha_\F,\eps}r})_{0<\eps<\frac{1}{4\alpha_\F\alpha_\G},r>0}.\end{equation}Notice that if  
let $\F:\O\to\O'$ and     $\G:\O'\to\O''$  are respectively
a $(\alpha_\F,k_\F)$-controlled morphism  and  $(\alpha_\G,k_\G)$-controlled morphism of order $r$,  
then equation (\ref{equ-composition}) defines a  $(\alpha_\G\alpha_F,k_\G*k_\F)$-controlled
morphism  $\G\circ\F:\O\to \O''$ of order $r$.

\begin{notation}
 Let $(\la,h)$ be a control pair and  let   $\O=(O^{\eps,r})_{0<\eps<1/4,r>0}$ and $\O'=(O'^{\eps,r>0})_{0<\eps<1/4,r>0}$be  quantitative objects and
let $\F=(F^{\eps,r})_{0<\eps<1/4,r>0}:\O\to\O'$ (resp. $\G=(G^{\eps,r})_{0<\eps<1/4,r>0}:\O\to\O'$) be a
$(\alpha_\F,k_\F)$-controlled morphism (resp. a $(\alpha_\G,k_\G)$-controlled morphism). Then we write
$\F\aeq\G$ if
\begin{itemize}
 \item $(\alpha_\F,k_\F)\lq (\lambda,h)$ and $(\alpha_\G,k_\G)\lq (\lambda,h)$.
\item for every $\eps$ in $(0,\frac{1}{4\lambda})$ and $r>0$, then
$$\iota^{\alpha_\F\eps,\lambda\eps,k_{\F,\eps}r,h_\eps r}_j\circ F^{\eps,r}=\iota^{\alpha_\G\eps,\lambda\eps,k_{\G,\eps}r,h_\eps r}_j\circ G^{\eps,r}.$$
\end{itemize}
\end{notation}


\begin{definition}
Let $\F:\O_1\to \O'_1,\, \F:\O_2\to \O'_2,\,\G:\O_1\to \O_2$ and $\G':\O'_1\to \O'_2$ be controlled
morphisms and let $(\lambda,h)$ be a control pair. Then the diagram (or the square)
$$\begin{CD}
\O_1@>\G>> \O_2 \\
         @V\F VV
         @VV\F'V\\
 \O'_1@>\G'>> \O'_2
\end{CD}$$
is called ${\mathbf{(\lambda,h)}}${\bf -commutative} (or  $\mathbf{(\lambda,h)}${\bf-commutes})  if $\G' \circ \F\aeq \F'\circ \G$.
The definition of $(\lambda,h)$-commutative diagram can be obviously extended to the setting  of controlled morphism of order $r$.
\end{definition}

Recall from \cite{oy3} the definition of  a controlled isomorphisms.

 \begin{definition}
  Let $(\lambda,h)$ be a control pair, and let $\F:\O\to\O'$  be a $(\alpha_\F,k_\F)$-controlled morphism with
$(\alpha_\F,k_\F)\lq (\lambda,h)$.
 $ \F$ is called  $(\lambda,h)$-invertible or a $(\lambda,h)$-isomorphism if there exists a controlled morphism
$\G:\O'\to\O$ such that
$\G\circ\F\aeq \Id_{\O}$ and
 $\F\circ\G\aeq \Id_{\O'}$. The controlled morphism $\G$ is called a $(\lambda,h)$-inverse for $\G$.
 \end{definition}

In particular, if $A$ and $B$ are filtered
$C^*$-algebras and if $\G:\K_*(A)\to\K_*(B)$ is a $(\lambda,h)$-isomorphism, then the induced  morphism $G:K_*(A)\to K_*(B)$  is an isomorphism
and its inverse is induced by  a controlled morphism (indeed induced by any $(\lambda,h)$-inverse for $\F$).

 
 \medskip

In order to state  in Section \ref{sec-kunneth} the quantitative K\"unneth formula, we will need the more general of quantitative isomorphism. Let  $\O_1=(O^{\eps,s}_1)_{0<\eps<1/4,s>0}$ and $\O_2=(O^{\eps,s}_2)_{0<\eps<1/4,s>0}$ be
quantitative objects. For a $(\alpha,h)$-controlled morphism  
 $$\F=(F^{\eps,r})_{0<\eps< \frac{1}{4\al},s>0}:\O_1\to\O_2,$$ consider the following statements:
 \begin{description}
\item[$QI_{\F}(\eps,\eps',s,s');\, 0<\eps\lq \eps' < \frac{1}{4\al},\,0<s\lq s'$] 
for any element  $x$  in $O^{\eps,s}_1$ such that   $F^{\eps,s}(x)=0$ in $O_2^{\al\eps,{h_\eps}s}$, then $\iota_{\O_1}^{\eps',s'}(x)=0$ in $O_1^{\eps', s'}$;
\item[$QS_{\F}(\eps,\eps',s,s');\, 0<\eps\lq \eps' < \frac{1}{4\al},\,0<s\lq h_{\eps'}s'$] for any element $y$ in $O_2^{\eps,s}$, there exists  an element  $x$ in $O_1^{\eps',s'}$ such that $F^{\eps',s'}(x)=\iota_{\O_2}^{\eps,\al\eps',s,h_{\eps'} s'}(y)$ in $O_2^{\al\eps',h_{\eps'}s'}$.
 \end{description} 
\begin{definition}\label{def-quantitative-morphism}\
\begin{itemize}
\item Let  $\O_1=(O^{\eps,s}_1)_{0<\eps<1/4,s>0}$ and $\O_2=(O^{\eps,s}_2)_{0<\eps<1/4,s>0}$ be
quantitative objects.  Then a   quantitative isomorphism 
$$\F:\O_1\to\O_2$$ is  an $(\alpha,h)$-controlled morphism  
 $\F=(F^{\eps,r})_{0<\eps< \frac{1}{4\al},s>0}$   for some control pair $(\al,h)$
   that satisfies the following: there exists a positive number $\la_0$, with $\la_0\gq 1$ such that
 \begin{itemize}
 \item for any  positive numbers $\eps$ and $s$  with   $\eps<\frac{1}{4\al\la_0}$  there exists a positive number   $s'$ with $s\lq s'$
    such that $QI_{\F}(\eps,\la_0\eps,s,s')$ is satisfied;
 \item for any  positive numbers $\eps$ and $s$  with   $\eps<\frac{1}{4\al}$,  there exists a  positive number  $s'$ with  $s \lq s'h_{\la_0\eps}$  such that $QS_{\F}(\eps,\la_0\eps,s,s')$ is satisfied. The positive number $\la_0$ is called the {\bf rescaling}  of the quantitative isomorphism $\F$.
 \end{itemize} 
 \item A uniform family of quantitative isomorphisms if a family $(\F_i)_{i\in I}$ where,
 $\F_i:\O_i\to\O'_i$ is for any  $i$ in $I$ an $(\al,h)$-controlled morphism for a given control pair $(\al,h)$ such that there exists a positive number $\la_0$, with $\la_0\gq 1$ for which the following holds
 \begin{itemize}
 \item for any  positive numbers $\eps$ and $s$  with   $\eps<\frac{1}{4\al\la_0}$  there exists a positive number   $s'$ with $s\lq s'$
    such that $QI_{\F_i}(\eps,\la_0\eps,s,s')$ is satisfied for any $i$ in $I$;
\item for any  positive numbers $\eps$ and $s$  with   $\eps<\frac{1}{4\al}$,  there exists a  positive number  $s'$ with  $s \lq s'h_{\al_0\eps}$  such that  $QS_{\F_i}(\eps,\la_0\eps,s,s')$ is satisfied for any $i$ in $I$.
 \end{itemize} 
 \end{itemize}
 \end{definition}
 In particular, if $A$ and $B$ are filtered
$C^*$-algebras and if $\G:\K_*(A)\to\K_*(B)$ is a  quantitative  isomorphism, then the induced  morphism $G:K_*(A)\to K_*(B)$  is an isomorphism (but its inverse is no more given by a quantitative isomorphism).

\subsection{Controlled exact sequences}\label{subsection-control-exact-sequence}

In this subsection, we recall the controlled exact sequence for quantitative objects. This controlled exact sequence is
an important tool in computing quantitative $K$-theory for filtered $C^\ast$-algebras.

\begin{definition} Let $(\lambda,h)$ be a control pair,
\begin{itemize}
 \item
Let $\O=(O_{\eps,s})_{0<\eps< \frac{1}{4},s>0},\, \O'=(O'_{\eps,s})_{0<\eps< \frac{1}{4},s>0}$  and
$\O''=(O''_{\eps,s})_{0<\eps< \frac{1}{4},s>0}$ be quantitative objects and let
$$\F=(F^{\eps,s})_{0<\eps< \frac{1}{4\lambda},s>0}:\O\to\O'$$ be a  $(\alpha_\F,k_\F)$-controlled morphism and
  let   $$\G=(G^{\eps,s})_{0<\eps<\frac{1}{4\alpha_\G},r>0}:\O'\to \O''$$
be a $(\alpha_\G,k_\G)$-controlled morphism.
Then the composition
$$\O\stackrel{\F}{\to}O'\stackrel{\G}{\to}\O''$$ is said to be $(\lambda,h)$-exact at $\O'$ if
 $\G\circ\F=0$ and if
 for any $0<\eps<\frac{1}{4\max\{\lambda\alpha_\F,\alpha_\G\}}$, any $s>0$  and  any $y$ in
$O'^{\eps,s}$  such that  $G^{\eps,s}(y)=0$ in $O''_{\eps,r}$, there exists an
  element $x$ in $O^{\lambda\eps,h_\eps s}$
such that $$F^{\lambda\eps,h_{\lambda\eps}s}(x)=\iota_{\O'}^{\eps,\alpha_\F\lambda\eps,r,k_{\F,\lambda\eps}h_\eps s}(y)$$ in
$O'^{\alpha_\F\lambda\eps,k_{\F,\lambda\eps}h_\eps s}$.
\item  A sequence of controlled morphisms
$$\cdots\O_{k-1}\stackrel{\F_{k-1}}{\lto}\O_{{k}}\stackrel{\F_{k}}{\lto}
\O_{{k+1}}\stackrel{\F_{k+1}}{\lto}\O_{{k+2}}\cdots$$ is called $(\lambda,h)$-exact if for every $k$,
the composition   $$\O_{{k-1}}\stackrel{\F_{k-1}}{\lto}\O_{{k}} \stackrel{\F_{k}}{\lto}
\O_{{k+1}}$$ is $(\lambda,h)$-exact at $\O_{{k}}$.
\item the notion of $(\lambda,h)$-exactness of a composition and of a sequence can obviously be extended to the setting of  controlled morphism of order $r$. \end{itemize}\end{definition}

\subsection{Controlled exact sequence in quantitative $K$-theory}

Examples of controlled exact sequences in quantitative $K$-theory are provided by controlled six term exact sequences associated to a completely filtered extensions of $C^*$-algebras \cite[Section 3]{oy2}.

\begin{definition}\label{def-completely-filtered-extension}
Let $A$ be a $C^*$-algebra filtered by $(A_r)_{r>0}$, let $J$ be an
ideal of $A$ and set $J_r=J\cap A_r$. The  extension of $C^*$-algebras $$0\to J\to A\to A/J\to 0$$ is called a completely filtered extension of $C^*$-algebras  if 
 the bijective continuous linear map $$A_r/J_r\lto (A_r+J)/J$$ induced by the inclusion $A_r\hookrightarrow\ A$  is  a complete  isometry i.e for any integer $n$, any positive number $r$ and any $x$ in $M_n(A_r)$,   then  $$\inf_{y\in M_n(J_r)}\|x+y\|=\inf_{y\in M_n(J)}\|x+y\|.$$\end{definition}
 Notice that  in this case, the ideal $J$ is  filtered by $(A_r\cap J)_{r>0}$  and $A/J$ is filtered by $(A_r+J)_{r>0}$.
 A particular case of completely   filtered extension of $C^*$-algebra is the case of filtered and semi-split extension of $C^*$-algebras \cite[Lemma 3.3]{oy2} (or a semi-split extension of filtered  algebras) i.e extension 
$$0\to J\to A\to A/J {\to} 0,$$ where
\begin{itemize}
\item $A$ is filtered by $(A_r)_{r>0}$;
\item   there
exists a completely positive (complete) norm decreasing cross-section  $$s:A/J\to A$$ such that
$$s(A_r+J)\subseteq A_r$$  for any
number $r>0$. 
\end{itemize}

For any extension of $C^*$-algebras $$0\to J\to A\to A/J\to 0$$  we denote by $\partial_{J,A}:K_{*}(A/J)\to K_{*}(J)$ the associated
(odd degree) boundary map in $K$-theory.
\begin{proposition}\label{prop-bound}
There exists a control pair $(\alpha_\DD,k_\DD)$  such that for any      completely filtered extension of   $C^*$-algebras
 $$0\longrightarrow J \longrightarrow A \stackrel{q}{\longrightarrow}
A/J\longrightarrow 0,$$ there exists a $(\alpha_\DD,k_\DD)$-controlled  morphism of odd degree
 $$\DD_{J,A}=(\partial_ {J,A}^{\eps,r})_{0<\eps<\frac{1}{4\alpha_\DD},r>0}: \K_{*+1}(A/J)\to
  \K_*(J)$$
 which induces in $K$-theory   $\partial_{J,A}:K_{*}(A/J)\to K_{*+1}(J)$.
\end{proposition}
Moreover  the controlled boundary map enjoys the usual naturally properties with respect to extensions (see \cite[Remark 3.8]{oy2}).
\begin{theorem}\label{thm-six term}  There exists a control  pair $(\lambda,h)$  such that for any  completely  filtered extension of    $C^*$-algebras $$0  \longrightarrow J
\stackrel{\jmath}{\longrightarrow} A \stackrel{q}{\longrightarrow}
A/J\longrightarrow 0,$$ then the following six-term sequence is $(\lambda,h)$-exact
$$\begin{CD}
\K_0(J) @>\jmath_*>> \K_0(A)  @>q_*>>\K_0(A/J)\\
    @A\DD_{J,A} AA @.     @V\DD_{J,A} VV\\
\K_1(A/J) @<q_*<< \K_1(A)@<\jmath_*<< \K_1(J)
\end{CD}
$$
\end{theorem}

 \subsection{$KK$-theory and controlled morphisms}

 In this subsection, we discuss compatibility of Kasparov's $KK$-theory with quantitative $K$-theory of filtered $C^\ast$-algebras.

 Let $A$ be a $C^*$-algebra and let $B$ be a filtered $C^*$-algebra filtered by $(B_r)_{r>0}$.  Let us define $A\ts B_r$ as the closure in the spatial tensor
product $A\ts B$ of the algebraic tensor product of $A$ and $B_r$. Then the $C^*$-algebra $A\ts B$ is filtered by $(A\ts B_r)_{r>0}$.
 If $f:A_1\to A_2$ is a homomorphism of $C^*$-algebras, let us set
 $$f_B:A_1\ts B\to A_2\ts B;\, a\ts b\mapsto f(a)\ts b.$$ Recall from \cite{kas} that for $C^*$-algebras  $A_1$, $A_2$ and $B$,  Kasparov defined a tensorization map
$$\tau_B:KK_*(A_1,A_2)\to KK_*(A_1\ts B,A_2\ts B).$$ If $B$ is a filtered $C^*$-algebra, then for any $z$ in $KK_*(A_1,A_2)$  the morphism
$$K_*(A_1\ts B)\lto K_*(A_2\ts B);\, x\mapsto x\ts_{A_1\ts B}\tau_B(z)$$  is induced by a controlled morphism  which enjoys compatibity properties with Kasparov product \cite[Theorem 4.4]{oy2}.
  \begin{theorem}\label{thm-tensor}
  There exists a control pair $(\alpha_\TT,k_\TT)$ such that
  \begin{itemize}
  \item for any filtered   $C^*$-algebra $B$;
  \item for any $C^*$-algebras $A_1$ and $A_2$;
  \item for any element $z$ in $KK_*(A_1,A_2)$,
  \end{itemize}
  There exists a $(\alpha_\TT,k_\TT)$-controlled morphism  $\TT_B(z):\K_*(A_1\ts B)\to \K_*(A_2\ts B)$  with same degree as $z$ that satisfies the following:
 \begin{enumerate}
 \item $\TT_B(z):\K_*(A_1\ts B)\to \K_*(A_2\ts B)$ induces in $K$-theory the right multiplication by $\tau_B(z)$;
\item  For any  elements $z$ and $z'$ in $KK_*(A_1,A_2)$ then
$$\TT_B(z+z')=\TT_B(z)+\TT_B(z').$$
\item Let $A'_1$ be a filtered $C^*$-algebras and  let $f:A_1\to A'_1$ be a homomorphism of  $C^*$-algebras, then
$\T_B(f^*(z))=\T_B(z)\circ f_{B,*}$  for all $z$ in $KK_* (A'_1,A_2)$.
\item Let $
A'_2$ be a $C^*$-algebra and  let $g:A'_2\to A_2$ be a homomorphism of $C^*$-algebras then
$\TT_B(g_{*}(z))=g_{B,*}\circ \TT_B(z)$
for any $z$ in $KK_*(A_1,A'_2)$.
\item $\TT_B([Id_{A_1}])\stackrel{(\alpha_\TT,k_\TT)}{\sim} \Id _{\K_*({A_1}\ts B)}$.
\item For any $C^*$-algebra $D$ and any element $z$ in $KK_*(A_1,A_2)$, we have  $\TT_B(\tau_D(z))=\TT_{B\ts D}(z)$.
\item For any  semi-split extension of  $C^*$-algebras $0\to J\to A\to A/J\to 0$    with corresponding   element   $[\partial_{J,A}]$  of
$KK_1(A/J,J)$
 that implements the boundary map, then we have
$$\TT_B([\partial_{J,A}])\stackrel{(\al_\TT,k_\TT)}{\sim}\DD_{J\ts B,A\ts B}.$$\end{enumerate}
\end{theorem}
Moreover, the controlled tensorization morphism $\TT_B$ is compatible with Kasparov products.
\begin{theorem}\label{thm-product-tensor} There exists a control pair $(\lambda,h)$ such that the following holds :

  let  $A_1,\,A_2$ and $A_3$  be separable $C^*$-algebras and let $B$ be a  filtered $C^*$-algebra. Then for any
$z$ in $KK_*(A_1,A_2)$ and  any  $z'$ in
$KK_*(A_2,A_3)$, we have
$$\TT_B(z\ts_{A_2} z')\aeq  \TT_B(z') \circ\TT_B(z).$$
\end{theorem}
We also have in the case of finitely generated group a controlled version of the Kasparov transformation.
Let $\Ga$ be a finitely generated group. Recall that a length on $\Gamma$ is a map $\ell:\Gamma\to\R^+$ such
that
\begin{itemize}
\item $\ell(\gamma)=0$ if and only if $\gamma$ is the identity element  $e$  of $\Gamma$;
\item  $\ell(\gamma\gamma')\lq\ell(\gamma)+\ell(\gamma')$ for all
  element $\gamma$ and $\gamma'$ of $\Gamma$.
\item $\ell(\gamma)=\ell(\gamma^{-1})$.
\end{itemize}
In what follows, we will assume that $\ell$ is a word length arising from a finite generating symmetric  set $S$, i.e.,
$\ell(\gamma)=\inf\{d\text{ such that }\ga=\ga_1\cdots\ga_d\text{ with } \ga_1,\ldots,\ga_d\text{ in }S\}$.
Let us denote by $B(e,r)$ the
ball centered at the neutral  element of $\Ga$ with  radius $r$, i.e.,
$B(e,r)=\{\gamma\in\Gamma\text{ such that }\ell(\gamma)\lq r\}$. Let $A$ be a separable $\Ga$-$C^*$-algebra, i.e., a separable $C^*$-algebra provided with an action of $\Ga$ by automorphisms.  For any positive number $r$, we set
$$(A\rtr \Gamma)_r\defi\{f\in C_c(\Ga,A)\text{ with support in
}B(e,r)\}.$$ Then the $C^*$-algebra $A\rtr \Gamma$ is filtered by
$((A\rtr \Gamma)_r)_{r>0}$.
Moreover if $f:A\to B$ is a $\Ga$-equivariant morphism of $C^*$-algebras, then the induced homomorphism
$f_\Ga:A\rtr\Ga\to B\rtr\Ga$ is a filtered homomorphism. In \cite{kas},
for any $\Ga$-$C^*$-algebras $A$ and $B$, was  constructed  a natural transformation
$J_\Ga:KK_*^\Ga(A, B)\to KK_*(A\rtr\Ga,B\rtr\Ga)$ that preserves Kasparov products, called the Kasparov transformation.
The Kasparov transformation admits a quantitative version \cite[Section 5]{oy2}.
\begin{theorem}\label{thm-kas}
There exists a control pair $(\alpha_\JJ,k_\JJ)$ such that
\begin{itemize}
\item for any  separable  $\Gamma$-$C^*$-algebras  $A$ and $B$;
\item For any $z$ in $KK^\Ga_*(A,B)$,
\end{itemize}
there exists a $(\alpha_\JJ,k_\JJ)$-controlled morphism
$$\JR(z): \K_*(\AG)\to
\K_*(B\rtr\Gamma)$$
 of same degree as $z$ that satisfies the following:
\begin{enumerate}
\item For any element  $z$  of
$KK^\Ga_*(A,B)$, then
$\JR(z): \K_*(\AG)\to
\K_*(B\rtr\Gamma)$ induces in $K$-theory right multiplication by
$J_\Ga^{red}(z)$.
\item  For any $z$ and
  $z'$ in $KK_*^\Gamma(A,B)$, then
  $$\JR(z+z')=\JR(z)+\JR(z').$$

\item For any $\Gamma$-$C^*$-algebra $A'$, any homomorphism  $f:A\to A'$  of $\Gamma$-$C^*$-algebras and  any $z$ in $KK_*^\Gamma(A',B)$, then
$\JR(f^*(z))=
\JR(z)\circ f_{\Ga,*}$.
\item For any $\Gamma$-$C^*$-algebra $B'$, any homomorphism $g:B\to B'$ of
  $\Gamma$-$C^*$-algebras and  any
  $z$ in $KK_*^\Gamma(A,B)$, then $\JR(g_*(z))=g_{\Ga,*}\circ
\JR(z)$.
\item $\JR(Id_A)\stackrel{(\al_\JJ,k_\JJ)}{\sim}\Id_{\K_*(A\rtr\Ga)}$.
\item If $$0\to J\to A\to A/J\to 0$$ is a semi-split exact sequence of
  $\Ga$-$C^*$-algebras, let $[\partial_{J,A}]$ be the element of
  $KK^\Ga_1(A/J,J)$ that implements the boundary map
  $\partial_{J,A}$. Then we have
$$\JR([\partial_{J,A}])\stackrel{(\al_\JJ,k_\JJ)}{\sim}\D_{J\rtr\Ga,A\rtr\Ga}.$$\end{enumerate}
\end{theorem}
The controlled Kasparov transformation is compatible with Kasparov products.
\begin{theorem}\label{thm-product} There exists a control pair $(\lambda,h)$ such that the following holds:
for every  separable $\Ga$-$C^*$-algebras $A,\,B$ and $D$, any elements $z$ in
$KK_*^\Gamma(A,B)$ and   $z'$ in
$KK_*^\Gamma(B,D)$, then
$$\JR(z\otimes_B z')\aeq \JR(z')\circ \JR(z).$$
\end{theorem}
\begin{remark}\label{remark-kasp-tens}
We can choose $(\al_\JJ,k_\JJ)$ such that $(\al_\JJ,k_\JJ)=(\al_\T,k_\T)$. In this case, for any $\Ga$-$C^*$-algebra $A$, any 
$C^*$-algebras $D_1$ and $D_2$ equipped with the  trivial action of $\Ga$ and any $z$ in $KK_*(D_1,D_2)$, then 
$$\T_{A\rtr \Ga,*}(z)=\JR(\tau_{A,*}(z)).$$
\end{remark}
We have similar result for maximal crossed products.
\subsection{Quantitative  assembly maps}\label{subsection-quantitative-assembly-map}

In this subsection, we discuss a quantitative version of the Baum-Connes assembly map.

Let $\Ga$ be a finitely generated group and let $B$ be a $\Ga$-$C^*$-algebra. We equip $\Ga$ with any word metric. Recall that if $d$ is  a positive number, then the Rips complex of degree $d$ is the set $P_d(\Ga)$ of probability measures with support of diameter less than $d$. Then $P_d(\Ga)$ is a locally finite simplicial complex and  provided  with the simplicial topology, $P_d(\Ga)$ is endowed  with a proper and cocompact action of $\Ga$ by left translation.
In \cite{oy2},  for any $\Ga$-$C^*$-algebra $B$, we construct quantitative assembly maps
$$\mu_{\Gamma,B,*}^{\eps,r,d}: KK_*^\Gamma(C_0(P_d(\Ga)),B)\to
K_*^{\eps,r}(B\rtr\Ga),$$ with $d>0,\,\eps\in(0,1/4)$ and $r\gq r_{d,\eps}$,  where
$$[0,+\infty)\times (0,1/4)\lto (0,+\infty):\, (d,\eps)\mapsto r_{d,\eps}$$ is a function  independent on
$B$, non decreasing in $d$ and
non increasing  in $\eps$.  Moreover, the maps $\mu_{\Gamma,B,*}^{\eps,r,d}$ induce the usual assembly maps
$$\mu_{\Gamma,B,*}^{d}: KK_*^\Gamma(C_0(P_s(\Ga)),B)\to
K_*(B\rtr\Ga),$$ i.e.,  $\mu_{\Gamma,B,*}^{d}=\iota_*^{\eps,r}\circ\mu_{\Gamma,B,*}^{\eps,r,d}$.
Let us recall now the definition of the quantitative assembly maps.
Observe first that  any $x$ in $P_d(\Ga)$ can be written down  in a unique way as a finite convex combination
$$x=\sum_{\ga\in\Ga}\lambda_\ga(x)\de_\ga,$$ where $\de_\ga$ is the Dirac probability measure at $\ga$ in $\Ga$.
The functions $$\lambda_\ga:P_d(\Ga)\to [0,1]$$ are continuous and
$\ga(\lambda_{\ga'})=\lambda_{\ga\ga'}$ for all $\ga$ and $\ga'$ in $\Ga$.
 The
function
$$p_{\Ga,d}:\Gamma\to C_0(P_d(\Ga));\, \gamma\mapsto\sum_{\gamma\in\Gamma}\lambda_e^{1/2}\lambda_\ga^{1/2}$$
 is a projection of
 $C_0(P_d(\Ga))\rtr\Gamma$ with  propagation  less than $d$. Let us set then $r_{d,\eps}=k_{\JJ,\eps/\alpha_\JJ}d$, where
 the control pair $(\al_\JJ,k_\JJ)$ is as  inTheorem \ref{thm-kas}. Recall that $k_\JJ$ can be chosen non increasing and in this case,  $r_{d,\eps}$ is non decreasing in
 $d$ and non increasing in $\eps$.
\begin{definition}\label{def-quantitative-assembly-map}
For any $\Gamma$-$C^*$-algebra $A$ and any positive numbers $\eps$, $r$ and
$d$ with $\eps<1/4$ and
$r\gq r_{d,\eps}$, we define the quantitative assembly map
\begin{eqnarray*}
\mu_{\Gamma,A,*}^{\eps,r,d}: KK_*^\Gamma(C_0(P_d(\Ga)),A)&\to&
K_*^{\eps,r}(\AG)\\
z&\mapsto&\big(J_\Gamma^{red,\frac{\eps}{\alpha_\JJ},\frac{r}{k_{\JJ,{\eps}/{\alpha_\JJ}}}}(z)\big)\left([p_{\Ga,d},0]_{\frac{\eps}{\alpha_\JJ},\frac{r}{k_{\JJ,{\eps}/{\alpha_\JJ}}}}\right),
\end{eqnarray*} where the notation $  [p_{\Ga,d},0]_{\frac{\eps}{\alpha_\JJ},\frac{r}{k_{\JJ,{\eps}/{\alpha_\JJ}}}}$ is as in Definition 1.5.
\end{definition}
Then according to  theorem \ref{thm-kas}, the map  $\mu_{\Gamma,A}^{\eps,r,d}$
is a  group homomorphism. For any positive numbers
 $d$ and $d'$ such that $d\lq d'$, we denote by
$q_{d,d'}:C_0(P_{d'}(\Ga))\to C_0(P_d(\Ga))$ the homomorphism induced by the restriction
  from $P_{d'}(\Ga)$ to $P_d(\Ga)$. It is straightforward to check that if
  $d$, $d'$ and  $r$  are positive numbers such that $d\lq d'$ and
  $r\gq r_{d',\eps}$, then
$\mu_{\Gamma,A}^{\eps,r,d}=\mu_{\Gamma,A}^{\eps,r,d'}\circ
q_{d,d',*}$. Moreover, for every positive numbers
$\eps,\,\eps',\,d,\,r$ and $r'$ such that $\eps\lq\eps'\lq
1/4$, $r_{d,\eps}\lq r$, $r_{d,\eps'}\lq r'$, and
$r<r'$, we get by definition of a controlled morphism that
\begin{equation*}
\iota^{\eps,\eps',r,r'}_*\circ \mu_{\Ga,A,*}^{\eps,r,d}=\mu_{\Ga,A,*}^{\eps',r',d}.
\end{equation*}
In \cite{oy2} we introduced quantitative statements for the quantitative assembly maps.
For a  $\Ga$-$C^*$-algebra $A$ and positive numbers
$d,d',r,r',\eps$ and $\eps'$ with $d\lq d'$, $\eps\lq\eps'< 1/4$,  $r_{d,\eps}\lq
r'$ and  $ r\lq
r'$ we set:
\begin{description}
\item[$QI_{\Ga,A,*}(d,d',r,\eps)$]  for any element $x$ in
  $KK_*^\Gamma(C_0(P_d(\Ga)),A)$, then
  $\mu_{\Gamma,A,*}^{\eps,r,d}(x)=0$
  in $K_*^{\eps,r}(\AG)$ implies that
  $q_{d,d'}^*(x)=0$ in $KK_*^\Gamma(C_0(P_{d'}(\Ga)),A)$.
\item[$QS_{\Ga,A,*}(d,r,r',\eps,\eps')$] for every $y$
  in  $K_*^{\eps,r}(\AG)$, there exists an element $x$ in
  $KK_*^\Gamma(C_0(P_d(\Ga)),A)$ such that
$\mu_{\Gamma,A,*}^{\eps',r',d}(x)=\iota_*^{\eps,\eps',r,r'}(y).$
\end{description}
The following results were then proved.
\begin{theorem}\label{thm-quant-surj}Let $\Ga$ be a discrete group.
\begin{enumerate}
\item Assume that   for any  $\Ga$-$C^*$-algebra
  $A$, the assembly map $\mu_{\Ga,A,*}$ is one-to-one.
    Then for any positive numbers $d$,
  $\eps$ and $r\gq r_{d,\eps}$ with $\eps<1/4$ and  $r\gq r_{d}$, there
exists   a positive number $d'$ with $d'\gq d$ such that
$QI_{\Ga,A}(d,d',r,\eps)$ is satisfied  for every $\Ga$-$C^*$-algebra
  $A$;
\item Assume that   for any  $\Ga$-$C^*$-algebra
  $A$, the assembly map $\mu_{\Ga,A,*}$ is onto.  Then for some  positive number $\al_0$ which not depends  on $\Ga$ or on $A$ and such that with $\al_0>1$ and for any  positive numbers $\eps$
and $r$ with  $\eps<\frac{1}{4\al_0}$, there exist
positive numbers $d$ and  $r'$ with $
r_{d,\eps}\lq r'$ and $r\lq r'$ such that  $QS_{\Ga,A}(d,r,r',\eps,\al_0\eps)$
is satisfied for every $\Ga$-$C^*$-algebra
  $A$.
\end{enumerate}
In particular, if  $\Ga$ satisfies the Baum-Connes conjecture with
coefficients, then $\Ga$ satisfies points  {\rm (i)} and {\rm
  (ii)} above.
\end{theorem}

In \cite{oy3} we developed  a geometric version of the controlled assembly maps and of the quantitative statements in the following setting.
Let $\Si$ be a proper discrete metric space and let $A$ be  a $C^*$-algebra.
Then the distance $d$ on $\Si$ induces  a filtration on $A\ts \K(\ell^2(\Si))$ in the following way: let $r$ be a positive number and
 $T=(T_{\si,\si'})_{(\si,\si')\in\Si^2}$ be an element in $A\ts \K(\ell^2(\Si))$, with $T_{\si,\si'}$ in $A$ for any $\si$ and $\si'$ in $\Si^2$. Then 
 $T$ has propagation  less that $r$ if $T_{\si,\si'}=0$ for $\si$ and $\si'$ in $\Si$ such that $d(\si,\si')>r$. As for finitely generated group,  we define 
 the  Rips complex of degree $d$ of $\Si$ as  the set $P_d(\Si)$ of probability measure with support of diameter less than $d$. Then $P_d(\Si)$ is a locally finite simplicial complex and  is locally compact when endowed   with the simplicial topology.
 Let us define then $$K_*(P_d(\Si),A)\defi \lim_{Z\subseteq P_d(\Si)} KK_*(C(Z),A),$$ where $Z$ runs through compact subset of $P_d(\Si)$.
In \cite{oy3},  for any  $C^*$-algebra $A$, was  constructed  local quantitative coarse assembly maps
$$\nu_{\Si,A,*}^{\eps,r,d}: K_*(P_d(\Si)),A)\to
K_*^{\eps,r}(A\ts \K(\ell^2(\Si))),$$ with $d>0,\,\eps\in(0,1/4)$ and $r\gq r_{d,\eps}$.   The map  $\nu_{\Si,\bullet,*}^{\eps,r,d}$ is natural in the $C^*$-algebra and induces in $K$-theory the index map, i.e the maps $\iota_*^{\eps,r}\circ\nu_{\Si,A,*}^{\eps,r,d}$  is up   to Morita equivalence  given for any compact subset $Z$ of $P_d(\Si)$ by the morphism in the inductive limit $KK_*(C(Z),A)\to K_*(A)$ induced by the map  $Z\to \{pt\}$.
Moreover, the maps  $\nu_{\Si,A,*}^{\bullet,\bullet,\bullet}$ are compatible with  structure morphisms  and with inclusion of Rips complexes:
\begin{itemize}
 \item $\iota^*_{\eps,\eps',r,r'}\circ \nu_{\Si,A,*}^{\eps,r,d}=\nu_{\Si,A,*}^{\eps',r',d}$ for any
positive numbers $\eps,\,\eps',\,r,\,,r'$ and $s$ such that $\eps\lq\eps'<1/4,\,r_{d,\eps}\lq r,\,r_{d,\eps'}\lq r'$ and
$r\lq r'$;
\item $\nu_{\Si,A,*}^{\eps,r,d'}\circ q^*_{d,d'}=\nu_{\Si,A,*}^{\eps,r,d}$ for any
positive numbers $\eps,\,r,\,d$ and $d'$ such that $\eps<1/4,\,d\lq d'$ and $r_{d',\eps}\lq r$, where $$q^*_{d,d'}:KK_*(C_0(P_{d'}(\Si)),A)\to KK_*(C_0(P_d(\Si)),A)$$  is the homomorphism induced by the restriction
  from $P_{d'}(\Si)$ to $P_d(\Si)$
\end{itemize}
 Let us consider then 
for 
$d,d',r,r',\eps$ and $\eps'$ positive numbers with $d\lq d'$, $\eps'\lq\eps< 1/4$,  
$r_{d,\eps}\lq
r$ and  $ r'\lq
r$ the following statements:
\begin{description}
\item[$QI_{\Si,A,*}(d,d',r,\eps)$]  for any element $x$ in
  $K_*(P_d(\Si),A)$, then 
  $\nu_{\Si,A,*}^{\eps,r,d}(x)=0$
  in $K_*^{\eps,r}(A\ts \K(\ell^2(\Si)))$ implies that
  $q_{d,d'}^*(x)=0$ in    $K_*(P_{d'}(\Si),A)$.
\item[$QS_{\Si,A,*}(d,r,r',\eps,\eps')$] for every $y$
  in  $K_*^{\eps',r'}(A\ts \K(\ell^2(\Si)))$, there exists an element $x$ in $K_*(P_d(\Si),A)$
   such that
$$\nu_{\Si,A,*}^{\eps,r,d}(x)=\iota_*^{\eps',\eps,r',r}(y).$$
\end{description}

Recall that a   a proper discrete metric space $\Si$ with bounded geometry coarsely embeds  in a Hilbert space if there exist 
\begin{itemize}
\item a map $\phi:\Si\to\H$ where $\H$ is a Hilbert space;
\item $\rho_\pm:\R^+\to\R^+$   two maps with $\lim_{+\infty}\rho_\pm=+\infty$,
\end{itemize}
such that $$\rho_-(d(x,y))\lq \|\phi(x)-\phi(y)\|\lq \rho_+(d(x,y))$$ for any $x$ and $y$ in $\Si$.
Proper discrete metric spaces with bounded geometry that coarsely embed into a Hilbert space provide numerous examples that  satisfy the above geometric quantitative statement \cite[Theorems 4.9 and 4.10]{oy3}. 
\begin{theorem}\label{thm-geo-statement}  Let $\Si$ be a 
 discrete metric space with bounded geometry that coarsely embeds into a Hilbert space.
\begin{enumerate}
\item  For any positive numbers $d$,
  $\eps$ and $r$ with  $\eps<1/4$ and  $r\gq r_{d,\eps}$, there
exists   a positive number $d'$ with $d'\gq d$ for which
$QI_{\Si,A,*}(d,d',r,\eps)$ is satisfied for any $C^*$-algebra $A$.
\item There exists a positive number $\lambda>1$ such that for any positive numbers $\eps$ and $r'$  with $\eps<\frac{1}{4\lambda}$, there exist
positive numbers $d$ and  $r$ with  $
r_{d,\eps}\lq r$ and $r'\lq r$ for which  $QS_{\Si,A,*}(d,r,r',\lambda\eps,\eps)$
is satisfied for any $C^*$-algebra $A$.
\end{enumerate}
\end{theorem}

\section{Controlled Mayer-Vietoris pairs}
In this section, we introduce the concepts of neighborhood $C^\ast$-algebras and controlled Mayer-Vietoris pair.
This is the cornerstone for the  controlled Mayer-Vietoris sequence for quantitative $K$-theory which will be stated in Section \ref{section-six-terms}. We also discuss a few technical lemmas useful for establishing  the latter.
\subsection{$\eps$-$r$-$N$-invertible elements of  a filtered $C^*$-algebra}
In \cite[Section $7$]{oy2} was introduced the notion of $\eps$-$r$-$N$-invertible element of a unital Banach algebra. In this subsection, we study  $\eps$-$r$-$N$-invertible elements for $C^*$-algebras. In particular, we state   an analogue of the polar decomposition  in the setting of $\eps$-$r$-unitaries.
\begin{definition}Let $A$  be a  unital $C^*$-algebra filtered by $(A_s)_{s>0}$  and
let $\eps,\,r$ and $N$ be positive numbers with $\eps<1$.
An element $x$ in $A_r$ is called $\eps$-$r$-$N$-invertible if $\|x\|\lq N$ and  there exists $y$ in $A_r$ such that
$\|y\|\lq N,\,\|xy-1\|<\eps$ and $\|yx-1\|<\eps$. Such an element $y$ is called an $\eps$-$r$-$N$-inverse for $x$.
\end{definition}
\begin{remark}\label{rem-norm-inv}If  $x$ is $\eps$-$r$-$N$-invertible, then $x$ is invertible and for any $\eps$-$r$-$N$-inverse  $y$ for $x$, we have  $\|x^{-1}-y\|\leq \frac{\eps N}{1-\eps} $.\end{remark}
\begin{definition}Let $A$  be a unital  $C^*$-algebra filtered by  $(A_s)_{s>0}$ and
let $\eps,\,r$ and $N$ be positive numbers with $\eps<1$. Two $\eps$-$r$-$N$-invertibles in $A$ are called homotopic if there
exists $Z:[0,1]\to A$ an  $\eps$-$r$-$N$-invertible  in $A[0,1]$ such that $Z(0)=x$ and $Z(1)=y$.
\end{definition}
In the setting of $\eps$-$r$-$N$-invertibles and of $\eps$-$r$-unitaries, there is the analogue of the polar decomposition.

\begin{lemma}\label{lemma-polar-decomp}
For any  positive number $N$ there exists a control pair $(\alpha,l)$ and a positive number $N'$ with $N'\gq N$ such that the following holds.

\medskip

For any filtered unital $C^*$-algebra  $A$   filtered by $(A_s)_{s>0}$, any positive numbers $\eps$ and $r$ with $\eps<\frac{1}{4\al}$ and every $\eps$-$r$-$N$-invertible element $x$ of $A$, there exist $h$ positive and $\alpha\eps$-$l_\eps r$-$N'$-invertible in $A$  and $u$ an $\alpha\eps$-$l_\eps r$-unitary in $A$ such that $\||x|-h\|<\alpha\eps$ and $\|x-uh\|<\alpha\eps$. Moreover we can choose $u$ and $h$ such that
\begin{itemize}
\item there exists a real  polynomial function $Q$  with $Q(1)=1$ such that  $u=xQ(x^*x)$ and $h=x^*xQ(x^*x)$;
\item $h$ admits a positive $\alpha\eps$-$l_\eps r$-$N'$-inverse;
\item If $x$ is homotopic to $1$ as an $\eps$-$r$-$N$-invertible, then $u$ is homotopic to $1$ as an $\alpha\eps$-$l_\eps r$-unitary.
\end{itemize}
\end{lemma}
\begin{proof}
According to Remark \ref{rem-norm-inv} and since $\eps<1/4$, if $x$ is an $\eps$-$r$-$N$-invertible, then $x$ is invertible
and $\|x^{-1}\|<2N$ and hence $\|(x^*x)^{-1}\|<4N^2$. This implies that the spectrum of $x^*x$ is included in $[\frac{1}{4N^2},N^2]$. Let $t_0$ and $t_1$ be positive numbers such that $t_0<\min(\frac{1}{4N^2},1)$ and $\max(N^2,1)<t_1$.
Let us consider the power serie $\sum a_nt^n$ of $t\mapsto \frac{1}{\sqrt{1+t}}$ for $t$ in $[0,1]$ and let $l_\eps$ be the smallest integer such that
$$\sum_{k=l_\eps+1}^{+\infty}|a_k|\left(\frac{1-t_1}{t_1}\right)^k<\frac{\min(\sqrt{t_1},1)\eps}{2}$$ and
$$\sum_{k=0}^{l_\eps}a_kt^k\gq 1/2$$ for all $t$ in $[0,1-\frac{t_0}{t_1}]$.
 Since $\sum a_n\left(\frac{x^*x-t_1}{t_1}\right)^n$ converges to
$\sqrt{t_1}(x^*x)^{-1/2}=\sqrt{t_1}|x|^{-1}$,  if we set
$$Q(t)=\frac{1}{\sqrt{t_1}}\sum_{k=0}^{l_\eps}a_k\left(\frac{t-t_1}{t_1}\right)^k+\frac{1}{\sqrt{t_1}}\sum_{k=l_\eps+1}^{+\infty}a_k\left(\frac{1-t_1}{t_1}\right)^k$$ then $Q$ is a polynomial of degree $l_\eps$ such that $Q(1)=1,\,$ $Q(t)\gq 0$ for every $t$ in $[0,1-\frac{t_0}{t_1}]$ and $\|Q(x^*x)-(x^*x)^{-1/2}\|<\eps$.
If we set $u=xQ(x^*x)$, then $u$ is a $\alpha$-$2l_\eps+1$-unitary for some $\alpha>1$ depending only on $N$.
Set now $h=u^*x=x^*xQ(x^*x)$, then up to taking a large $\alpha$, there exists a control pair $(\alpha,k)$ and a positive number $N'$  depending only on $N$, with $N'\gq N$ and such that $u$ and $h$ satisfy the required properties and $Q(x^*x)$ is a positive $\alpha\eps$-$k_\eps r$-$N'$-inverse for $h$. Moreover, if $(x_t)_{t\in[0,1]}$ is a homotopy of $\eps$-$r$-invertibles between $1$ and $x$, then
$(x_tQ(x_t^*x_t))_{t\in[0,1]}$ is a homotopy of $\alpha\eps$-$k_\eps r$-unitaries between $1$ and $u$.
\end{proof}
Let $A$  be a  unital $C^*$-algebra filtered by $(A_s)_{s>0}$. For $x$ and $y$ in $A$, set
$$X(x)=\begin{pmatrix}1&x\\0&1\end{pmatrix}$$ and $$Y(y)=\begin{pmatrix}1&0\\y&1\end{pmatrix}$$ and consider the commutators $$Z(x,y)=X(x)Y(y)X(x)^{-1}Y(y)^{-1}=\begin{pmatrix}1+xy+xyxy&-xyx\\yxy&1-yx\end{pmatrix}$$ and
$$Z'(x,y)=Y(y)^{-1}X(x)^{-1}Y(y)X(x)=\begin{pmatrix}1-xy&-xyx\\yxy&1+yx+yxyx\end{pmatrix}.$$

\begin{lemma}\label{lem-commutator}
Let  $A$  be a  unital  $C^*$-algebra filtered by $(A_s)_{s>0}$  and
let $\eps$ and $r$  be positive numbers with $\eps<1/4$. Let $x_1$ and $x_2$ in $A_r$ such that $x_1+x_2$ is an $\eps$-$r$-unitary.Then we have the inequality
\begin{equation*}\begin{split}
\left\|X(x_1)Z(x_2,-x_1^*)Y(-x_1^*)X(x_1)X(x_2)Y(-x_2^*)\right.Z'(&x_1,-x_2^*)X(x_2)\begin{pmatrix}0&-1\\1&0\end{pmatrix} \\
 &\left.-\begin{pmatrix}x_1+x_2&0\\0&x_1^*+x_2^*\end{pmatrix}\right\|< 3\eps.\end{split}
\end{equation*}
\end{lemma}
\begin{proof}
Let us set $u=x_1+x_2$.
Consider the matrix
$$W(u)=X(u)Y(-u^*)X(u)\begin{pmatrix}0&-1\\1&0\end{pmatrix}=\begin{pmatrix}2u-uu^*u&uu^*-1\\1-u^*u&u^*\end{pmatrix}.$$
Since $u$ is an $\eps$-$r$-unitary, then
$$\left\|\begin{pmatrix}u& 0\\0&u^*\end{pmatrix}-W(u)\right\|<3\eps.$$
We have
$$W(u)=X(x_1)X(x_2)Y(-x_1^*)Y(-x_2^*)X(x_1)X(x_2)\begin{pmatrix}0&-1\\1&0\end{pmatrix}.$$
 This, together with the definition of $Z$ and $Z'$, implies that $$W(u)=X(x_1)Z(x_2,-x_1^*)Y(-x_1^*)X(x_1)X(x_2)Y(-x_2^*)Z'(x_1,-x_2^*)X(x_2)\begin{pmatrix}0&-1\\1&0\end{pmatrix}.$$
\end{proof}

\subsection{Coercive decomposition pair and $R$-neighborhood $C^*$-algebra}
We introduce in this subsection the basic ingredient to define  controlled Mayer-Vietoris pairs.

\medskip

If $\De$ and $\De'$ are two closed linear subspace of a $C^*$-algebra $A$ such that $\De\subseteq \De'$,
we equipped   $M_n(\De/\De')\cong M_n(\De)/M_n(\De')$ with the quotient $C^*$-algebra norm, i.e
if $x$ is a element of  $M_n(\De)$, then $\|x+M_n(\De')\|=\inf\{\|x+y\|;\,y\in M_n(\De') \}$. Then this family of norms  is a matrix norm on $\De/\De'$.
\begin{definition}
Let $A$ be a $C^*$-algebra filtered by $(A_s)_{s>0}$ and let $r$ be a positive number. 
\begin{itemize}
\item a    coercive decomposition pair of degree $r$ for
$A$ (or a  coercive decomposition $r$-pair) is a pair   $(\Delta_1,\Delta_2)$  of  closed linear subspaces of $A_r$  such that
there exists  a positive number  $C$ satisfying the following: for any positive number $s$ with $s\lq r$ the  inclusion $\De_1\cap A_s \hookrightarrow A_s$
induces an  isomorphism $$\frac{\De_1\cap A_s}{\De_1\cap \De_2\cap A_s}\stackrel{\cong}{\lto} \frac{A_s}{\De_2\cap A_s}$$
whose inverse is  bounded in norm by $C$.
\item a completely coercive decomposition  pair of degree $r$ for
$A$ (or a completely coercive decomposition $r$-pair) is a pair   $(\Delta_1,\Delta_2)$  of  closed linear subspaces of $A_r$ such that
there exists  a positive number  $C$ satisfying the following: for any positive number $s$ with $s\lq r$ the  inclusion $\De_1\cap A_s \hookrightarrow A_s$
induces a complete isomorphism $$\frac{\De_1\cap A_s}{\De_1\cap \De_2\cap A_s}\stackrel{\cong}{\lto} \frac{A_s}{\De_2\cap A_s}$$
whose inverse has complete norm bounded by $C$.
\end{itemize}
\end{definition}

\begin{remark}
Let $A$ be a $C^*$-algebra filtered by $(A_s)_{s>0}$, let $r$ be a positive number and let $(\Delta_1,\Delta_2)$  be a pair of  closed linear subspaces of $A_r$. 
Then $(\Delta_1,\Delta_2)$
is a  coercive decomposition pair of degree $r$ for
$A$ if and only if there exists a positive number  $c$ such that for every positive number $s$ with $s\lq r$  and any $x$ in 
$A_s$, there exists $x_1$ in $\De_1\cap A_s$ and $x_2$ in $\De_2\cap A_s$, both with norm at most $c\|x\|$ and such that 
$x=x_1+x_2$. In the same way, 
 $(\Delta_1,\Delta_2)$
is a  completely coercive decomposition pair of degree $r$ for
$A$ if and only if there exists a positive number  $c$ such that for every positive number $s$ with $s\lq r$, any integer $n$ and any $x$ in 
$M_n(A_s)$, there exists $x_1$ in $M_n(\De_1\cap A_s)$ and $x_2$ in $M_n(\De_2\cap A_s)$, both with norm at most $c\|x\|$ and such that 
$x=x_1+x_2$. 
The (completely) coercive decomposition  $r$-pair   $(\De_1,\De_2)$ is said to have  {\bf coercity $\mathbf{c}$}. 
\end{remark}
The aim of this subsection is to show that   for any coercive decomposition  $r$-pair   $(\De_1,\De_2)$, there exists a control pair $(\al,h)$ depending indeed only on the coercitivity, such that up to stabilisation, any $\eps$-$s$-unitary of $A$ with $0<\eps\lq \frac{1}{4\al}$ and $0<s\lq\frac{r}{h_\eps}$ can be decomposed as a  product of 
two $\al\eps$-$h_\eps s$-unitaries  lying respectively in  the unitarization of  some suitable neighborhood algebras of $\De_1$ and $\De_2$.

\begin{definition} Let $A$  be a  $C^*$-algebra filtered by $(A_s)_{s>0}$. Let $r$ and $R$ be  positive numbers  and let $\Delta$ be a closed  linear subspace  of  $A_r$. We define $C^*N_\Delta^{(r,R)}$, the $R$-neighborhood $C^*$-algebra of $\Delta$,  as the $C^\ast$-subalgebra of $A$ generated by its $R$-neighborhood $N_\Delta^{(r,R)}=\Delta+A_R\cdot \Delta+\Delta\cdot A_R+A_R \cdot \Delta \cdot A_R$.
\end{definition}


Notice that  $C^*N_\Delta^{(r,R)}$ inherits from $A$ a structure of filtered $C^*$-algebra with
$C^*N^{(r,R)}_{\Delta,s}=C^*N_\Delta^{(r,R)}\cap A_s$ for every positive number $s$. For a positive number $s$ satisfying $s\lq r$, we also denote by $C^*N^{(s,R)}_{\Delta}$ for the $R$-neighbohood $C^*$-algebra of $\Delta\cap A_s$.

\begin{lemma}\label{lemma-decomp-invertible} For any positive number $c$,   there exist   positive numbers $\lambda,\,C$ and $N$, with $\lambda>1$ and $C>1$ such that the following holds.

Let $A$  be a  unital $C^*$-algebra filtered by $(A_s)_{s>0}$, let $r$  and $\eps$ be  positive numbers such that $\eps<\frac{1}{4\lambda}$ and let
$(\Delta_1,\Delta_2)$ be a  coercive decomposition pair for $A$ of degree $r$ with coercity $c$. Then
for any $\eps$-$r$-unitary $u$ in $A$ homotopic to $1$, there exist
an integer $k$ and $P_1$ and $P_2$ in $M_k(A_{C r})$ such that
\begin{itemize}
\item $P_1$ and $P_2$ are invertible;
\item $P_i-I_k$ and $P^{-1}_i-I_k$ are elements in the matrix algebra   $M_n(C^*N^{(r,4r)}_{\Delta_i,C r})$ for $i=1,2$;
\item $\|P_i\|<N$ and  $\|P_i^{-1}\|<N$ for $i=1,2$;
\item for $i=1,2$, there exists a homotopy $(P_{i,t})_{t\in[0,1]}$ of invertible elements  in   $M_k(A_{C r})$ between $1$ and $P_i$ such that  $\|P_{i,t}\|<N,\, \|P_{i,t}^{-1}\|<N$ and $P_{i,t}-I_k$ and $P^{-1}_{i,t}-I_k$ are elements in the matrix algebra  $M_n(C^*N^{(r,4r)}_{\Delta_i,C r})$ for every $t$ in $[0,1]$.
\item $\|\diag(u,I_{k-1})-P_1P_2\|<\lambda\eps$.
\end{itemize}
\end{lemma}
%
\begin{proof}
Let $(u_t)_{t\in[0,1]}$ be a homotopy of $\eps$-$r$-unitaries of $A$ between $u=u_0$ and $1=u_1$ and let $t_0=0<t_1<\cdots<t_k=1$ be a partition of $[0,1]$ such that $\|u_{t_{i}}-u_{t_{i-1}}\|<\eps$ for $i=1,\cdots,k$.
Set $$V=\diag(u_{t_0},\ldots,u_{t_k},u_{t_0}^*,\ldots,u_{t_k}^*)$$ and $$W=\diag(1,u_{t_0}^*,\ldots,u_{t_{k-1}}^*,u_{t_0},\ldots,u_{t_{k-1}},1).$$
Then we have
\begin{equation*}\begin{split}
&\|\diag(u,I_{2k+1})-VW\|\lq\|\diag(u,I_{2k+1})-\diag(u_0,u_{t_1}u_{t_1}^*,\ldots,u_{t_k}u_{t_k}^*,I_{k+1})\|\\
&+\|\diag(u_{t_0},u_{t_1}u_{t_1}^*,\ldots,u_{t_k}u_{t_k}^*,I_{k+1})-\diag(u_{t_0},u_{t_1}u_{t_0}^*,\ldots,u_{t_k}u_{t_{k-1}}^*,I_{k+1})\| \\&+\|\diag(u_{t_0},u_{t_1}u_{t_0}^*,\ldots,u_{t_k}u_{t_{k-1}}^*,I_{k+1})-\diag(u_{t_0},u_{t_1}u_{t_0}^*,\ldots,u_{t_k}u_{t_{k-1}}^*,u_{t_0}u_{t_0}^*,u_{t_1}u_{t_1}^*,\ldots,u_{t_k}u_{t_k}^*)\|\\
&<4\eps.
\end{split}
\end{equation*}  For any matrix $X$ in $M_{2k}(A)$, let us set $\widetilde{X}=\diag(1,X,1)$ in  $M_{2k+2}(A)$. For every integer $i=-1,0,\ldots,k$, pick $v_i$ in $\De_1$ and $w_i$ in $\De_2$ such that $u_{t_i}=v_i+w_i$ such  that $\|v_i\|\lq c\|u_{t_i}\|$ and $\|w_i\|\lq c \|u_{t_i}\|$.
Set $x_1=\diag(v_0,\ldots,v_k)\,, x_2=\diag(w_0,\ldots,w_k)\,,y_1=\diag(v_0^*,\ldots,v_{k-1}^*)$
and  $y_2=\diag(w^*_0,\ldots,w^*_{k-1})$.  
Since we have
$$V=\diag(x_1+x_2,x_1^*+x_2^*)$$ and $$W=\widetilde{\diag}(y_1+y_2,y_1^*+y_2^*),$$ then if  we set
$$T(x,y)=X(x)Z(y,-x^*)Y(-x^*)X(x),$$
we deduce from Lemma \ref{lem-commutator} that
$$\left\|VW-T(x_1,x_2) T^{-1}(-x_2,-x_1)U_{k+1}\widetilde{T}(y_1,y_2) \widetilde{T}^{-1}(-y_2,-y_1)\widetilde{U_k}\right\|<9\eps$$ with $U_k=\begin{pmatrix}0&-I_{k}\\I_{k}&0\end{pmatrix}$ in $M_{2k}(\C)$ and hence
\begin{equation}\label{equation-P1P2}\left\|\diag(u,I_{2k+1})- {T}(x_1,x_2)T^{-1}(-x_2,-x_1)U_{k+1}\widetilde{T}(y_1,y_2)\widetilde{T}^{-1}(-y_2,-y_1)\widetilde{U_k}\right\|<13\eps\end{equation}
Let us show that \begin{equation}\label{equation-S}S(x_1,x_2,y_1,y_2)\defi T(x_1,x_2) T^{-1}(-x_2,-x_1)U_{k+1}\widetilde{T}(y_1,y_2) \widetilde{T}^{-1}(-y_2,-y_1)\widetilde{U_k}\end{equation} can  be decomposed  as  a   product  $P_1P_2$ with $P_1$ and $P_2$ satisfying the required properties. Notice  that as a product of elementary matrices, then $T(x_1,x_2)$ is invertible.
 By the definition of the neighborhood $C^\ast$-algebra, $T(x_1,x_2)-I_{2k+2}$ and $T^{-1}(x_1,x_2)-I_{2k+2}$
are elements in the matrix algebra   $M_{2k+2}(C^*N^{(r,r)}_{\Delta_1,7r})$. The same holds for $\widetilde{T}(y_1,y_2)$, and we have similar properties for $T(-x_2,-x_1) $ and
$\widetilde{T}(-y_2,-y_1)$ with respect to $C^*N^{(r,r)}_{\Delta_2,7r}$. In order to bring out some commutators, let us study the term  $$\widetilde{T}(y_1,y_2)^{-1}{U_{k+1}^*}T^{-1}(-x_2,-x_1){U_{k+1}}\widetilde{T}(y_1,y_2)$$ which is the product of the four first terms in equation (\ref{equation-S}).
Let us make the following observations:
\begin{itemize}
\item $\widetilde{X}(-y_1){U_{k+1}^*}{T}^{-1}(-x_2,-x_1){U_{k+1}}\widetilde{X}(y_1)-I_{2k+2}$ is an element in the matrix algebra $M_{2k+2}(C^*N^{(r,2r)}_{\Delta_2,9r})$;
\item $\widetilde{Z}^{-1}(y_1,y_2)\widetilde{X}(-y_1){U_{k+1}^*}{T}^{-1}(-x_2,-x_1){U_{k+1}}\widetilde{X}(y_1)\widetilde{Z}(y_1,y_2)-I_{2k+2}$ is an element  in the matrix algebra 
 $M_{2k+2}(C^*N^{(r,2r)}_{\Delta_2,17r})$ (because  $Z^{-1}(y_1,y_2)-I_{2k}=Z'(y_1,y_2)-I_{2k}$ and $Z(y_1,y_2)-I_{2k}$  are  elements   of  the matrix algebra   $M_{2k}(C^*N^{(r,r)}_{\Delta_2,4r}$));
\item $\widetilde{Y}(y_1^*)\widetilde{Z}^{-1}(y_1,y_2)\widetilde{X}(-y_1){U_{k+1}^*}{T}^{-1}(-x_2,-x_1){U_{k+1}}\widetilde{X}(y_1)\widetilde{Z}(y_1,y_2)\widetilde{Y}(y_1)-I_{2k+2}$ is an  element  in the matrix algebra  $M_{2k+2}(C^*N^{(3r,r)}_{\Delta_2,19r})$;
\item $\widetilde{X}(-x_1)\widetilde{Y}(y_1^*)\widetilde{Z}^{-1}(y_1,y_2)\widetilde{X}(-y_1){U_{k+1}^*}{T}^{-1}(-x_2,-x_1){U_{k+1}}\widetilde{X}(y_1)\widetilde{Z}(y_1,y_2)\widetilde{Y}(y_1)X(x_1)-I_{2k+2}$ is an element  in the matrix algebra  $M_{2k+2}(C^*N^{(r,4r)}_{\Delta_2,21r})$.\end{itemize}
Hence $\widetilde{T}(y_1,y_2)^{-1}{U_{k+1}^*}T^{-1}(-x_2,-x_1){U_{k+1}}\widetilde{T}(y_1,y_2)-I_{2k+1}$ is an element  in the matrix algebra 
  $M_{2k+2}(C^*N^{(r,r)}_{\Delta_2,21r})$. 
  Since $1=v_k+w_k$, then we have $$\begin{pmatrix}0&1\\-1&0\end{pmatrix}= T(v_k,w_k) T^{-1}(-w_k,-v_k).$$
 Therefore there exists for $i=1,2$ an invertible  matrix $Q_i(v_k,w_k)$ in $M_{2k+2}(A)$ such that
  $Q_i(v_k,w_k)-I_{2k+2}$ and  $Q^{-1}_i(v_k,w_k)-I_{2k+2}$ lie in $M_{2k+2}(C^*N^{(r,r)}_{\Delta_i,7r})$and
  $$Q_1(v_k,w_k)Q_2(v_k,w_k)=U_{k+1}\widetilde{U_{k}}.$$

  Therefore if we write 
  $S(x_1,x_2,y_1,y_2)=P_1 P_2$ with  $$P_1=T(x_1,x_2){U_{k+1}}\widetilde{T}(y_1,y_2)U_{k+1}^*Q_1(v_k,w_k)$$ and $$P_2=Q_2(v_k,w_k)\widetilde{U_k^*}\widetilde{T}(y_1,y_2)^{-1}{U_{k+1}^*}T^{-1}(-x_2,-x_1){U_{k+1}}\widetilde{T}(y_1,y_2)\widetilde{T}^{-1}(-y_2,-y_1)\widetilde{U_{k}}$$ are invertible matrices of $M_{2k+2}(A)$ such that
\begin{itemize}
\item $P_1-I_{2k+2}$ and $P^{-1}_1-I_{2k+2}$ are elements  in the matrix algebra    $M_{2k+2}(C^*N^{(r,r)}_{\Delta_1,21r})$.
\item $P_2-I_{2k+2}$ and $P^{-1}_2-I_{2k+2}$ are elements  in the matrix algebra    $M_{2k+2}(C^*N^{(r,4r)}_{\Delta_2,35r})$.
\end{itemize}
Since $P_1$ and $P_2$ can be written as a product of a fixed number, say $p$,  of matrices $X(x)$ or $Y(x)$ with $\|x\|< 2c$,
we see that $P_1$ and $P_2$ have norm less than $(2c+1)^p$.
According to  equation (\ref{equation-P1P2}), we have $$\left\|\diag(u,I_{2k+1})-P_1P_2\right\|<13\eps.$$
The required homotopies are then $$(T(tx_1,tx_2)U_{k+1}\widetilde{T}(ty_1,ty_2)U_{k+1}^*Q_1(tv_k,tw_k))_{t\in[0,1]}$$ and $$(Q_2(tv_k,tw_k)\widetilde{U_k^*}\widetilde{T}(ty_1,ty_2)^{-1}U_{k+1}T^{-1}(-tx_2,-tx_1)U_{k+1}^*\widetilde{T}(ty_1,ty_2)\widetilde{T}^{-1}(-ty_2,-ty_1)\widetilde{U_k})_{t\in[0,1]}.$$
\end{proof}

Let us briefly explain how we deal with  the non unital case. Let $A$ be a non unital filtered $C^*$-algebra and   let
$u$ be an $\eps$-$r$ unitary in $\tilde{A}$ such that $u-1$ is in $A$.
Assume that $u=x_1+x_2$  with $1-x_1$ and $x_2$ lie in  $A$. Proceeding as in the proof of Lemma  \ref{lem-commutator}, we see that 
$\diag(u,u^*)$ is $3\eps$-close to the product of
\begin{equation*}\label{equ-factor1}
P_1=X(x_2)X(x_1)Y(-x_1^*)X(x_1)\begin{pmatrix}0&-1\\1&0\end{pmatrix}X(-x_2)
\end{equation*}

and 
\begin{equation*}
P_2=X(x_2)\begin{pmatrix}0&1\\-1&0\end{pmatrix}X(-x_1)Y(-x_2^*)X(x_1)X(x_2)\begin{pmatrix}0&-1\\1&0\end{pmatrix}.
\end{equation*}
Now if $(\Delta_1,\Delta_2)$ be a  coercive decomposition pair for $A$ of degree $r$ with coercity $c$ and assume that
in the above decomposition of $u=x_1+x_2$ we have $1-x_1$ in $\De_1$ and $x_2$ in $\De_2$, we get  then
by  a straightforward computation that 
$P_1-I_2$ has coefficient in $C^*N^{(r,2r)}_{\Delta_1,5r}$ and 
$P_2-I_2$ has coefficient in $C^*N^{(r,2r)}_{\Delta_2,5r}$.
Notice that in  view of  the proof of Lemma \ref{lemma-almost-canonical-form-odd}, if under above assumption, $u$ is connected to $1$ as a $\eps$-$r$-unitary of $\tilde{A}$, then $u$ is connected to $1$ by a homotopy of $21\eps$-$r$-unitaries $(u_t)_{t\in[0,1]}$ of $\tilde{A}$ such that $u_t-1$ lies in $A$ for all $t$ in $[0,1]$. Hence proceeding as in the proof of Lemma \ref{lemma-decomp-invertible} we get:

\begin{lemma}\label{lemma-decomp-invertible-non-unital} For any positive number $c$,   there exist   positive numbers $\lambda,\,C$ and $N$, with $\lambda>1$ and $C>1$ such that the following holds.

Let $A$  be a non  unital $C^*$-algebra filtered by $(A_s)_{s>0}$, let $r$  and $\eps$ be  positive numbers such that $\eps<\frac{1}{4\lambda}$ and let
$(\Delta_1,\Delta_2)$ be a  coercive decomposition pair for $A$ of degree $r$ with coercity $c$. Then
for any $\eps$-$r$-unitary $u$ in $\tilde{A}$ homotopic to $1$ and such that $u-1$ lies in $A$, there exist
an integer $k$ and $P_1$ and $P_2$ in $M_k(\tilde{A}_{C r})$ such that
\begin{itemize}
\item $P_1$ and $P_2$ are invertible;
\item $P_i-I_k$ and $P^{-1}_i-I_k$ are elements in the matrix algebra   $M_n(C^*N^{(r,5r)}_{\Delta_i,C r})$ for $i=1,2$;
\item $\|P_i\|<N$ and  $\|P_i^{-1}\|<N$ for $i=1,2$;
\item for $i=1,2$, there exists a homotopy $(P_{i,t})_{t\in[0,1]}$ of invertible elements  in   $M_k(\tilde{A}_{C r})$ between $1$ and $P_i$ such that  $\|P_{i,t}\|<N,\, \|P_{i,t}^{-1}\|<N$ and $P_{i,t}-I_k$ and $P^{-1}_{i,t}-I_k$ are elements in the matrix algebra  $M_n(C^*N^{(r,5r)}_{\Delta_i,C r})$ for every $t$ in $[0,1]$.
\item $\|\diag(u,I_{k-1})-P_1P_2\|<\lambda\eps$.
\end{itemize}
\end{lemma}

\begin{proposition}\label{proposition-decom-unitaries} For every positive number $c$, 
there exists a control pair $(\alpha,l)$ such that the following holds.

\medskip

Let $A$  be a  unital  $C^*$-algebra filtered by $(A_s)_{s>0}$, let $r$  and $\eps$ be  positive numbers such that $\eps<\frac{1}{4\alpha}$ and let
$(\Delta_1,\Delta_2)$ be a  coercive decomposition pair for $A$ of degree $r$ with coercity $c$. Then  for any $\eps$-$r$-unitary $u$ in $A$ homotopic to $1$, there exist
a positive  integer $k$ and $w_1$ and $w_2$  two $\alpha\eps$-$l_\eps r$-unitaries in $M_k(A)$ such that
\begin{itemize}
\item $w_i-I_k$ is an  element  in the matrix algebra   $M_k(C^*N^{(r,4r)}_{\Delta_i,l_\eps r})$ for $i=1,2$;
\item for $i=1,2$, there exists  a homotopy $(w_{i,t})_{t\in[0,1]}$ of  $\alpha\eps$-$l_\eps r$-unitaries between  $1$ and $w_i$ such that
$w_{i,t}-I_k\in M_k(C^*N^{(r,4r)}_{\Delta_i,l_\eps r})$ for all $t$ in $[0,1]$.
\item $\|\diag(u,I_{k-1})-w_1w_2\|<\alpha\eps$.
\end{itemize}
\end{proposition}
\begin{proof}
As in Lemma \ref{lemma-decomp-invertible}, let $\lambda,\,C$ and $N$ be positive numbers, $k$ be an integer and  $P_1$ and $P_2$  be matrices of $M_k(A_{C r})$ such that $\|\diag(u,I_{k-1})-P_1P_2\|<\lambda\eps$. Since $P_1$ and $P_2$ are $\eps$-$Cr$-$N$-invertible for every $\eps$ in $(0,1)$, then according to
 Lemma
\ref{lemma-polar-decomp}, there exists
\begin{itemize}
\item  a control pair $(\alpha,l)$;
\item $w_1$ an $\eps$-$l_{\eps/\alpha}r$-unitary and $h_1$ an $\eps$-$l_{\eps/\alpha}r$-$N$-invertible both  in
$$M_{2k+2}(C^*N^{(r,4r)}_{\Delta_1}+\C),$$ with $h_1$ positive and admitting a positive  $\eps$-$l_{\eps/\alpha}r$-$N$-inverse;
\item $w_2$ an $\eps$-$l_{\eps/\alpha}r$-unitary and $h_2$ an $\eps$-$k_{\eps/\alpha}r$-$N$-invertible both  in $$M_{2k+2}(C^*N^{(r,4r)}_{\Delta_2}+\C),$$ with $h_2$ positive and admitting a positive  $\eps$-$l_{\eps/\alpha}r$-$N$-inverse,
\end{itemize}
such that $\|P_1-w_1h_1\|< \eps\,,\|P_2^*-w_2h_2\|< \eps,\,\||P_1|-h_1\|<\eps$ and $\||P_2^*|-h_2\|<\eps$.
Then  \begin{equation}\label{equ-lemma-polar}\|w_1h_1h_2w_2^*-\diag(u,I_{2k+1})\|<(2N+\lambda+1)\eps\end{equation}  and hence, up to replacing $\lambda$ by $4(2N+\lambda+1)$, we know according to Lemma \ref{lemma-almost-closed} that $w=w_1h_1h_2w_2^*$ is a
$\lambda\eps$-$4l_{\eps/\alpha}r$-unitary. Let us prove that $h_1h_2$ is closed to $I_{2k+2}$.

\medskip

Let $h_1'$ be a positive  $\eps$-$l_{\eps/\alpha}r$-$N$-inverse for $h_1$.  Then we have  $\|w_1^*w-h_1h_2w_2^*\|<2\lambda\eps$ and then $\|h_1'w_1^*w-h_2w_2^*\|<4\lambda N\eps$.  This implies that
 $$\|h_1'w_1^*v(h_1'w_1^*w)^*-h_2w_2^*(h_2w_2^*)^*\|<16\lambda N\eps.$$
 Since $\|h_1'w_1^*w(h_1'w_1^*w)^*-{h'_1}^2\|<3\lambda N^2\eps$ and $\|h_2w_2^*(h_2w_2^*)^*-h_2^2\|<3\lambda N^2\eps$, we
 deduce that there exists $\lambda'\gq\lambda$ depending only on $\lambda$ and $N$ such that
 $\|{h'_1}^2-h_2^2\|<\lambda'\eps$. But,  since $h'_1$ and $h_2$ are  $\eps$-$l_{\eps/\alpha}r$-$N$-invertible with $\eps<1/2$, their  spectrum is bounded below by $\frac{1}{2N}$. The square root is   Lipschitz on the set of positive elements of $A$ with
 spectrum   bounded below by $\frac{1}{2N}$ (this can be checked easily by holomorphic functional calculus), thus there exists a positive number $M$, depending only on $N$ such that $\|{h'_1}-h_2\|<M\lambda'\eps$. Since $h'_1$ is
  an $\eps$-$l_{\eps/\alpha}r$-$N$-inverse for $h_1$,
 we finally obtain  that $\|h_1h_2-I_{2k+2}\|<(1+\lambda'MN)\eps$.

 \medskip
 Combining this inequality with equation (\ref{equ-lemma-polar}), we know that there exist a positive number $\lambda''>1$, depending only on
 $N$ and $\lambda'$ such that $$\|w_1w_2^*-\diag(u,I_{2k+1})\|<\lambda''\eps.$$
 According to Lemma \ref{lemma-polar-decomp}, $w_1=P_1Q(P_1^*P_1)$ where $Q$ is polynomial and such that $Q(1)=1$.
 Since $P_1-I_{2k+2}$ lies in $C^*N^{(r,4r)}_{\Delta_1}$, then the same holds for $w_1-I_{2k+2}$ and similarly, $w_2-I_{2k+2}$ lies in $C^*N^{(r,4r)}_{\Delta_2}$.

\end{proof}

Proceeding similarly, Lemma \ref{lemma-decomp-invertible-non-unital} allows to deal with the non unital case.
\begin{proposition}\label{proposition-decom-unitaries-non-unital} For every positive number $c$, 
there exists a control pair $(\alpha,l)$ such that the following holds.

\medskip

Let $A$  be a  non unital  $C^*$-algebra filtered by $(A_r)_{r>0}$, let $r$  and $\eps$ be  positive numbers such that $\eps<\frac{1}{4\alpha}$ and let
$(\Delta_1,\Delta_2)$ be a  coercive decomposition pair for $A$ of degree $r$ with coercity $c$. Then  for any $\eps$-$r$-unitary $u$ in $\tilde{A}$ homotopic to $1$ and such that $u-1$ lies in $A$, there exist
a positive  integer $k$ and $w_1$ and $w_2$  two $\alpha\eps$-$l_\eps r$-unitaries in $M_k(\tilde{A})$ such that
\begin{itemize}
\item $w_i-I_k$ is an  element  in the matrix algebra   $M_k(C^*N^{(r,5r)}_{\Delta_i,l_\eps r})$ for $i=1,2$;
\item for $i=1,2$, there exists  a homotopy $(w_{i,t})_{t\in[0,1]}$ of  $\alpha\eps$-$l_\eps r$-unitaries between  $1$ and $w_i$ such that
$w_{i,t}-I_k\in M_k(C^*N^{(r,5r)}_{\Delta_i,l_\eps r})$ for all $t$ in $[0,1]$.
\item $\|\diag(u,I_{k-1})-w_1w_2\|<\alpha\eps$.
\end{itemize}
\end{proposition}

\subsection{Neighborhood $C^*$-algebras and controlled Mayer-Vietoris pair.}
In this subsection, we give the definition of neighborhood $C^\ast$-algebras and controlled Mayer-Vietoris pair. Our prominent examples will be given by Roe algebras.
\begin{definition}
Let  $A$ be a $C^*$-algebra   filtered by $(A_s)_{s>0}$, let $r$ be a positive number and let $\Delta$ be a closed  linear subspace of $A_r$. Then a sub-$C^*$-algebra $B$ of $A$ is called an $r$-controlled  $\Delta$-neighborhood-$C^*$-algebra  if
\begin{itemize}
\item $B$ is filtered by $(B\cap A_r)_{r>0}$;
\item $C^*N^{(r,5r)}_{\Delta}\subseteq B$.
\end{itemize}

\end{definition}
\begin{example}\label{example-Roe-algebras}
Let $\Si$ be a discrete metric space with bounded geometry and consider $C^*(\Si)$ the Roe Algebra of $\Si$. Recall that $C^*(\Si)$ is the closure of the algebra of locally compact and finite propagation operators on $\ell^2(\Si)\ts\H$, where $\H$ is a fixed separable Hilbert space. Then  $C^*(\Si)$ is filtered by the propagation.  For $r$ a positive number, let $(X_i)_{i\in\N}$ be a family of finite  subsets of $\Si$  whith uniformly bounded diameter which is $R$-disjoint
(i.e., $d(X_i,X_j)\gq R$ if $i\neq j$) for some positive number $R\gq 12r$. Let us consider the set $\Delta\subseteq C^*(\Si)_r$  of locally compact operators on $\ell^2(\Si)\ts\H$ with support in
$$\big\{(x,y)\in \Si\times\Si;\, x\in\bigcup_{i\in\N}X_i,\,  d(x,y)\lq r\big\}.$$ For a positive number $s$, let us set
$X_{i,s}=\{x\in X_i\text{ such that } d(x,X_i)<s\}$. If $s<R/2$,  then $(X_{i,s})_{i\in\N}$ is a family of $(R-2s)$-disjoint subsets of $\Si$ with  uniformly bounded diameter. Consider then the subalgebra $A_\De$ of   $C^*(\Si)$ of  operators
with support in $\bigcup_{i\in \N}X_{i,s}\times X_{i,s}$. Then
$$A_\De\cong \prod_{i\in\N}\Kp(\ell^2(X_{i,s})\ts\H)\cong \left(\prod_{i\in\N}\Kp(\ell^2(X_{i,s})\ts\Kp(\H)\right)$$
and $A_\De$  is for every $s$
with $5r<s<R/2$  an $r$-controlled  $\Delta$-neighborhood-$C^*$-algebra.
\end{example}

\begin{definition}\label{def-CIA} Let $S_1$ and $S_2$ be two subsets of a $C^*$-algebra $A$.
The pair $(S_1, S_2)$ is said to have  \emph{complete intersection approximation}  property (CIA) if
there exists $c>0$ such that for any positive number $\eps$, any   $x\in M_n(S_1)$ and $ y\in M_n(S_2)$ for some $n$ and $||x-y||<\eps $,
then there exists $z\in M_n(S_1\cap S_2)$ satisfying
$$||z-x||<c \eps, ~~~~||z-y||<c\eps.$$ The positive number $c$ is called the {\bf coercity} of the pair 
$(S_1, S_2)$.
\end{definition}

In the above definition, we note that the inequalities $||x-y||<\eps $ and  $||z-x||<c \eps$ implies $||z-y||<(c+1)\epsilon.$
Hence we can remove the condition $||z-y||<c\eps$ up to replacing the constant $c$ by $c+1$.

\begin{definition}\label{definition-weak-coarse-MV-pair}
Let  $A$ be a $C^*$-algebra filtered by $(A_s)_{s>0}$   and let $r$  be a  positive number. An $r$-controlled weak  Mayer-Vietoris pair for $A$ is a quadruple
 $(\De_1,\De_2,A_{\Delta_1},A_{\Delta_2})$ such that for some positive number $c$.
\begin{enumerate}
\item  $(\De_1,\De_2)$ is a completely coercive  decomposition pair for $A$ of order $r$ with coercitivity $c$.
\item  $A_{\Delta_i}$ is an $r$-controlled   $\Delta_i$-neighborhood-$C^*$-algebra  for $i=1,2$;
\item the pair $ (A_{ \Delta_1, s}, A_{\Delta_2, s})$ has the CIA property    with coercitivity $c$ as defined  above for any {positive number $s$} with $s\lq r$.
\end{enumerate}
The positive number $c$ is called the {\bf coercity} of the $r$-controlled weak  Mayer-Vietoris pair $(\De_1,\De_2,A_{\Delta_1},A_{\Delta_2})$.
\end{definition}

\begin{remark}
In the above definition,
\begin{enumerate}
\item $(\Delta_1\cap A_s,\Delta_2\cap A_s,A_{\Delta_1},A_{\Delta_2})$ is an  $s$-controlled Mayer-Vietoris pair for any $0<s\lq r$ with same coercitivity as 
$(\De_1,\De_2,A_{\Delta_1},A_{\Delta_2})$.
\item {$A_{\De_1}\cap A_{\De_2}$ is filtered by  $(A_{\De_1,r}\cap A_{\De_2,r})_{r>0}$.}
\end{enumerate}
\end{remark}

In order to ensure some persistence properties for  the controlled Mayer-Vietoris exact sequence (see Corollary \ref{cor-bound1}), we need to strenghen condition (iii) of Definition \ref{definition-weak-coarse-MV-pair}.
\begin{definition}\label{definition-coarse-MV-pair}
Let  $A$ be a $C^*$-algebra filtered by $(A_s)_{s>0}$   and let $r$  be a  positive number. An $r$-controlled  Mayer-Vietoris pair for $A$  is a quadruple
 $(\De_1,\De_2,A_{\Delta_1},A_{\Delta_2})$ such that for some positive number $c$.
\begin{enumerate}
\item  $(\De_1,\De_2)$ is a completely coercive  decomposition pair for $A$ of order $r$ with coercitivity $c$.
\item  $A_{\Delta_i}$ is an $r$-controlled   $\Delta_i$-neighborhood-$C^*$-algebra  for $i=1,2$;
\item the pair $ (A_{ \Delta_1, s}, A_{\Delta_2, s})$ has the CIA property for any {positive number $s$}   with coercitivity $c$ as defined  above.
\end{enumerate}
The positive number $c$ is called the {\bf coercity} of the $r$-controlled Mayer-Vietoris pair $(\De_1,\De_2,A_{\Delta_1},A_{\Delta_2})$.
\end{definition}

If $A$  is a unital  $C^*$-algebra filtered by $(A_s)_{s>0}$ and if $(\De_1,\De_2,A_{\Delta_1},A_{\Delta_2})$ is an $r$-controlled  Mayer-Vietoris pair, we will view  $\widetilde{A_{\Delta_1}}$  the unitarization of  ${A_{\Delta_1}}$ as
${A_{\Delta_1}}+\C\cdot 1\subseteq A$  and similarly for ${A_{\Delta_2}}$ and  ${A_{\Delta_1}}\cap {A_{\Delta_2}}$.

\begin{example}\label{example-roe}
Let $(\Si,d)$ be a proper metric discrete space, let $X^{(1)}$ and $X^{(2)}$ be subsets in $\Si$ such that  $\Si=X^{(1)}\cup X^{(2)}$ and let $r$ be a positive number.
Assume that $X^{(1)}=\cup_{i\in\N} X^{(1)}_i$ and $X^{(2)}=\cup_{i\in\N} X^{(2)}_i$,
where  $(X^{(1)}_i)_{i\in\N}$ and $(X^{(2)}_i)_{i\in\N}$ are  families  of $R$-disjoint subsets of $\Si$ with  uniformly bounded diameter   for some positive number $R\gq 10r$. Let us consider as in Example \ref{example-Roe-algebras}
for $j=1,2$ the sets $\Delta_j\subseteq C^*(\Si)_r$  of locally compact operators on $\ell^2(\Si)\ts\H$ with support in
$$\big\{(x,y)\in \Si\times\Si;\, x\in X^{(j)},\,  d(x,y)\lq r\big\}$$  and  let us  consider then 
  the  subalgebra    $A_{\De_j}$ of  $C^*(\Si)$ of operators
with support in $\bigcup_{i\in \N}X^{(j)}_{i,s}\times X^{(j)}_{i,s}$
  for  some fixed positive number $s$ with $5r<s<R/2$.
{{Let $\chi_{X^{(2)}_{i,5r}}$ for $i$ integer be the characteristic function of
$$\{x\in \Si\text{ such that }d(x,X^{(2)}_i)\lq 5r\}.$$ Set 
$$\Psi: C^*(\Si)\lto C^*(\Si);\, x\mapsto \sum_{i\in\N}\chi_{X^{(2)}_{i,5r}}x\chi_{X^{(2)}_{i,5r}}.$$ Then $\Psi$ is norm decreasing. Since $\Psi(x_2)=x_2$ for every $x_2$ in $A_{\De_2}$, we obtain
$$\|\Psi(x_1)-x_2\|\lq \|x_1-x_2\|$$ for every $x_1$ in $M_n(A_{\De_1})$ and $x_2$ in $M_n(A_{\De_2})$. }}Since $\Psi(x_1)$ lies in
 $M_n(A_{\De_1}\cap A_{\De_2})$, we see that $(\De_1,\De_2,A_{\Delta_1},A_{\Delta_2})$ is an
 $r$-controlled  Mayer-Vietoris pair with coercitivity $1$.
\end{example}
Next lemma makes in the set-up of controlled  Mayer-Vietoris pairs the connection
between  the decompositions of   Proposition \ref{proposition-decom-unitaries} corresponding respectively to  an $\eps$-$r$-unitary  and its adjoint.

\begin{proposition}\label{proposition-decom-unitaries-adjoint} For every positive number $c$, 
there exists a control pair $(\alpha,l)$   such that the following holds.

\medskip
Let $A$ be any   unital  filtered $C^*$-algebra, let $r$ be any  positive number, let  $(\De_1,\De_2,A_{\De_1},A_{\De_2})$ be any $r$-controlled Mayer-Vietoris pair for $A$ at order $r$ with coercitivity $c$,
and  let $\eps$ and $r'$ be positive numbers with $\eps\in (0,\frac{1}{4\al})$  and $r\lq r'$. Assume that for some  $\eps$-$r$-unitary  $u$ in some
$M_n(A)$,   there exist
$v_1$ and $v_1'$ two  $\eps$-$r'$-unitaries in  $M_n(\widetilde{A_{\De_1}})$ 
and 
$v_2$ and $v_2'$ two $\eps$-$r'$-unitaries in  $M_n(\widetilde{A_{\De_2}})$ such that
$\|u-v_1v_2\|<\eps$ and   $\|u^*-v_1'v_2'\|<\eps$. Then  there exists an integer $k$ and $v_1''$ and $v_2''$ respectively $\al\eps$-$l_\eps r'$-unitaries in 
$M_{n+k}(\widetilde{A_{\De_1}})$  and $M_{n+k}(\widetilde{A_{\De_2}})$ such that
\begin{itemize}
\item $\|\diag(u^*,I_k)-v_1''v_2''\|<\al\eps$;
\item $v_i''$ is homotopic to $\diag(v^*_i,I_k)$ as an $\al\eps$-$l_\eps r'$-unitary in  $M_{n+k}(\widetilde{A_{\De_i}})$ for $i=1,2$.
\end{itemize}
Moreover, if $v_i-I_{n+k}$ and  $v'_i-I_{n+k}$ lie  in $M_{n+k}(A_{\De_i})$ for $i=1,2$ then $v''_1$ and $v''_2$ can be chosen such that 
  $v''_i-I_{n+k}$ lies in $M_{n+m}(A_{\De_i})$ for $i=1,2$.\end{proposition}
\begin{proof}
Let $(\al,l)$ be control pair as in Proposition \ref{proposition-decom-unitaries}. Since $\rho^{-1}_{A_{\De_j}}(v_j)$ and $\rho^{-1}_{A_{\De_j}}(v'_j)$ are for $j=1,2$ homotopic to $I_n$ as a $8\eps$-$s$-unitary of $M_n(\C)$ for every positive number $s$ \cite[Lemma 1.20]{oy2}, then up to replacing $\al$ by $90\al$, there exists an integer 
$k$  and  $w_1$ and $w_2$ be  two  $\al\eps$-$2l_\eps r$-unitaries respectively in $M_{2n+k}(A_{\De_1})$ 
and  $M_{2n+k}(A_{\De_2})$ such that  if we set $W_j=\diag(\rho^{-1}_{A_{\De_j}}(v_j),\rho^{-1}_{A_{\De_1}}(v'_j),I_k)$ for $j=1,2$, then
\begin{itemize}
\item $\|W_1 \diag(u,u^*,I_k)W_2 -w_1w_2\|<\al\eps$;
\item $w_j-I_{2n+k}$ is in  $M_{2n+k}(A_{\De_j})$ for $j=1,2$.
\item $w_j$ is homotopic to $I_{2n+k}$ as  an $\al\eps$-$2l_\eps r$-unitaries in $M_{2n+k}(A_{\De_j})$ for $j=1,2$.
\end{itemize}
Then
$$\|W_1diag(v_1v_2,v_1'v_2',I_k)W_2-w_1w_2\|<(\al+2)\eps$$ and hence we have

$$\|\diag(v_1^*,{v'}^*_1,I_k)W_1^*w_1-\diag(v_2,v_2',I_k)W_2w^*_2\|<5(\al+1)\eps.$$
Since  $\rho_{\A_{\De_j}}(\diag(v_j,v_j',I_k)W_j)=I_{2n+k}$ for $j=1,2$ and in view of   CIA property, there exist $v$ in $M_{2n+k}(A)$ with propagation less than $(2l_\eps+1)r'$ such that
\begin{itemize}
\item $v-I_{2n+k}$ is in $M_{2n+k}(A_{\De_1}\cap A_{\De_2})$;
\item $\|\diag(v_2,v'_2,I_k)W_2w^*_2-v\|<5c(\al+1)\eps$;
\item $\|\diag(v_1^*,{v'}^*_1,I_k)W_1^*w_1-v\|<5c(\al+1)\eps$.
\end{itemize}
In particular some control pair $(\al',l')$ depending only on $c$.
\begin{itemize}
\item  $v$ is an $\al'$-$l'_\eps$-$r'$-unitary in $M_{2n+k}(\widetilde{A_{\De_1}\cap A_{\De_2}})$;
\item $v$ is homotopic to $\diag(v^*_1,{v'_1}^*,I_k)$ as an $\al'$-$l'_\eps$-$r'$-unitary in $M_{2n+k}(\widetilde{A_{\De_1}})$;
\item $v$ is homotopic to $\diag(v_2,v'_2,I_k)$ as an $\al'$-$l'_\eps$-$r'$-unitary in $M_{2n+k}(\widetilde{A_{\De_2}})$.
\end{itemize}
Let us set  $v_1''=\diag(v'_1,I_{n+k})\cdot v$ and $v_2''=v^*\cdot \diag(v'_2,I_{n+k})$.  Then $v''_1$ and $v''_2$ satisfy the required properties for some suitable control pair depending only on $c$.
\end{proof}

\subsection{Controlled  Mayer-Vietoris pair associated to groupo\"\i ds}\label{subsection-groupoid} 

In this section, we discuss the example of a controlled  Mayer-Vietoris pair associated to groupo\"\i ds. Our method in this paper provides a different approach to the controlled Mayer-Vietoris sequence in the context of crossed product $C^\ast$-algebras in \cite{gwy}. Recall first the definition of a proper  symmetric length on an \'etale groupo\"\i d.
\begin{definition}
Let $\G$ be an \'etale groupo\"\i d, with compact base space $X$. A proper symmetric length on $\G$ is a continuous proper map
$\ell:\G\to\R^+$ such that
\begin{itemize}
\item $\ell(\ga)=0$ if and only $\ga$ is a unit  of $\G$;
\item $\ell(\ga)=\ell(\ga^{-1})$ for any $\ga$ in $\G$;
\item $\ell(\ga\cdot\ga')\lq \ell(\ga)+\ell(\ga')$ for any $\ga$ and $\ga'$ in $\G$ composable.
\end{itemize}
\end{definition}
Let $\G$ be an \'etale groupo\"\i d with compact base space $X$ and source map and range map $\textbf{r},\textbf{s}:\G\to X$  equipped with and symmetric proper  length $\ell$.
Then, if we set 
 $$\G_r=\{\ga\in\G;\, \text{such that }\ell(\ga)\lq r\},$$then 
  the reduced $C^*$-algebra $C^*_r(\G)$ of $\G$ is filtered by $(C^*_r(\G)_r)_{r>0}$ with 
  $$C^*_r(\G)_r=\{f\in C_c(\G)\text{ with support in }\G_r\}$$ for all positive number $r$.
  For every open subset $V$ of $X$ and every positive number $r$, set
$$V_r=\{\textbf{s}(\ga)\;, \ga\in\mathring{\G}_r\text{ and } \textbf{r}(\ga)\in V\}=\{\textbf{r}(\ga)\;, \ga\in\mathring{\G}_r\text{ and }\textbf{s}(\ga)\in V\}.$$ Then $V_r$ is an open subset of $X$ and $V\subseteq V_r$. If $Y$ and $Z$ are subsets of $ X$, then we set $\G_Y=\{\ga\in\G;\,\textbf{s}(\ga)\in Y\}\,,\G^Z=\{\ga\in\G;\,\textbf{r}(\ga)\in Z\}$ and
$\G_X^Y=\G_Y\cap\G^Z$. For every open subset $V$ of $X$ and every positive number $r$, let
$\G_V^{V,(r)}$ be the subgroupo\"\i d of $\G_V^{V}$ generated by $\G_{V,r}^{V}=\G_V^{V}\cap \mathring{\G}_r$. Then $\G_{V}^{V,(r)}$ is an open subgroupo\"\i d
 of $\G$. Let us set  $$\Delta_V=\{f\in C_0(\G_V)\text{ with support in }\G_r\}.$$Then $\Delta_V$ is a closed linear subspace of  $C^*_r(\G)_r$ and for every positive number $R$ and $R'$ with $r\leq R<R'$, then we have $C^*N_{\Delta_V}^{r,R}\subseteq  C^*_r\left(\G^{V_R',(R')}_{V_R'}\right)$. In particular, if $R>5r$, then $C^*_r\left(\G^{V_R,(R)}_{V_R}\right)$ is a $r$-controlled
 $\Delta_V$-neighborhood-$C^*$-algebra.

 \smallskip

Let $V^{(1)}$ and $V^{(2)}$ be two open subsets of $X$ such that $X=V^{(1)}\cup V^{(2)}$. Fix $R>5r$.
Set $\De_1=\De_{V^{(1)}}$ and $\De_2=\De_{V^{(2)}}$. Using partition of unity relatively to $V^{(1)}$ and $V^{(2)}$, we see that $(\De_1,\De_2)$ is a completely coercitive decomposition pair of order $r$ for $C^*_{red}(\G)$ with coercitivity $1$. Let us set also $A_{\De_1}=C^*_{red}\left(\G^{V^{(1)}_{R},(R)}_{V^{(1)}_{R}}\right)$ and $A_{\De_2}= C^*_{red}\left(\G^{V^{(2)}_{R},(R)}_{V^{(2)}_{R}}\right).$
Let $s$   be  positive number  with $s\lq r$, let $\eps$ be a positive number, let $x_1$ be an element
of $M_n\left(A_{\De_1, s}\right)$ and let $x_2$ be an element
of $M_n\left(A_{\De_2, s}\right)$ such that $\|x_1-x_2\|<\eps$. Let $x_1'$ and $x_2'$ be  respectively elements   in $M_n\left(C_c\left(\G^{V^{(1)}_{R},(R)}_{V^{(1)}_{R}}\right)\right)$ and
$M_n\left(C_c\left(\G^{V^{(2)}_{R},(R)}_{V^{(2)}_{R}}\right)\right)$ such that $\|x_1'-x_1\|<\eps$ and $\|x_2'-x_2\|<\eps$. Let  $K$ be a compact subset of $V^{(2)}_{R}$ such that all coefficients of $x'_2$ have support in $\G^K_K$ and let $h:X\to [0,1]$ be  a continuous function with support in $V^{(2)}_{R}$ and such
 that $h(z)=1$ for all $z$ in $K$. The Schur multiplication by $h\circ \textbf{s} \cdot h\circ \textbf{r}$
 $$C_c(\G)\to C_c(\G);\,f\mapsto f\cdot h\circ \textbf{s} \cdot h\circ \textbf{r} $$  extends to a completely positive  map
 $$\phi:C^*_{red}(\G)\to C^*_{red}(\G)$$ of complete norm less than $1$ and such that $\phi (x'_2) =x'_2$ and $\phi (x'_1)$ belongs to
 $C^*_{red}\left(\G^{V^{(1)}_{R},(R)}_{V^{(1)}_{R}}\right)_s\bigcap C^*_{red}\left(\G^{V^{(2)}_{R},(R)}_{V^{(2)}_{R}}\right)_s$. Moreover
 \begin{eqnarray*}
 \|\phi(x'_1) -x_2\|&\lq&\|\phi(x'_1) -x'_2\|+\|x_2'-x_2\|   \\
 &\lq&\|\phi(x'_1) -\phi(x'_2)\|+\eps\\
  &\lq&\|x'_1 -x'_2\|+\eps\\
 &\lq&4\eps.
 \end{eqnarray*} Hence,  $\left(\De_1,\De_2,C^*_{red}\left(\G^{V^{(1)}_{R},(R)}_{V^{(1)}_{R}}\right), C^*_{red}\left(\G^{V^{(2)}_{R},(R)}_{V^{(2)}_{R}}\right)\right)$     is for every $R>5r$  a $r$-controlled weak Mayer-Vietoris pair  for $C^*_{red}(\G))$ with coercitivity $5$. 
 Assume that there exists  a positive number $C$ such that for every compact subset $K$ of  $V^{(2)}$, there exists 
  a continuous function  $h:X\to [0,1]$ with support in $V^{(2)}$   that satisfies the following:
\begin{itemize}
\item    $h(z)=1$ for all $z$ in $K$;
\item the Schur multiplication by $h\circ \textbf{s} \cdot h\circ \textbf{r}$ extends to a completely bounded map
 $\Phi:C^*_{red}(\G)\to C^*_{red}(\G)$ with complete norm bounded by $C$ and such that for any $x$ in $C_c\left(\G^{V^{(2)}_{R},(R)}_{V^{(2)}_{R}}\right)$ with support in 
 $\G_K^K$, then $\Phi(x)=x$, 
 \end{itemize}
then 
 $\left(\De_1,\De_2,C^*_{red}\left(\G^{V^{(1)}_{R},(R)}_{V^{(1)}_{R}}\right), C^*_{red}\left(\G^{V^{(2)}_{R},(R)}_{V^{(2)}_{R}}\right)\right)$     is for every $R>5r$  a $r$-controlled  Mayer-Vietoris pair  for $C^*_{red}(\G))$ with coercitivity $c$ depending only on $C$. 
\section{Controlled Mayer-Vietoris six terms exact sequence  in quantitative $K$-theory}\label{section-six-terms}
In this section, we establish for a  control Mayer-Vietoris pair asssocied to a filtered $C^*$-algebra $A$ a controlled exact  sequence that allows to compute quantitative $K$-theory of $A$ up to a certain order.

\begin{notation}
Let $A$ be a unital $C^*$-algebra filtered by $(A_r)_{r>0}$, let $r$  be a  positive number and let $(\De_1,\De_2,A_{\Delta_1},A_{\Delta_2})$ be a  $r$-controlled  Mayer-Vietoris pair  for $A$. We denote by  $\jmath_{\Delta_1}: A_{\Delta_1}\to A\,,\jmath_{\Delta_2}: A_{\Delta_2}\to A\,,\jmath_{\Delta_1,\Delta_2}: A_{\Delta_1}\cap A_{\Delta_2} \to A_{\Delta_1}$ and $\jmath_{\Delta_2,\Delta_1}: A_{\Delta_1}\cap A_{\Delta_2} \to A_{\Delta_2}$ the obvious inclusion maps.
\end{notation}

\subsection{Controlled  half-exactness in the middle}

\begin{proposition}\label{prop-half-exact}For every positive number $c$, there exists a control pair $(\alpha,l)$ such that  for any filtered  $C^*$-algebra $A$, any   positive number $r$ and  any  $r$-controlled  weak Mayer-Vietoris pair $(\De_1,\De_2,A_{\Delta_1},A_{\Delta_2})$    for $A$
with coercitiviy $c$, then  the composition
$$\K_*(A_{\Delta_1}\cap A_{\Delta_2})
\stackrel{(\jmath_{\Delta_1,\Delta_2,*},\jmath_{\Delta_2,\Delta_1,*})}{-\!\!\!-\!\!\!-\!\!\!-\!\!\!-\!\!\!-\!\!\!-\!\!\!-\!\!\!-\!\!\!-\!\!\!-\!\!\!-\!\!\!\lto}
\K_*(A_{\Delta_1})\oplus \K_*(A_{\Delta_2})
\stackrel{(\jmath_{\Delta_1,*}-\jmath_{\Delta_2,*})}{-\!\!\!-\!\!\!-\!\!\!-\!\!\!-\!\!\!-\!\!\!-\!\!\!-\!\!\!\lto}\K_*(A)$$ is 
$(\alpha,l)$-exact at order $r$.
\end{proposition}
\begin{proof} Let us first assume that $A$ is unital.
In the even case let $y_1$ and $y_2$ be respectively element in $K_0^{\eps,s}(A_{\De_1})$ and $K_0^{\eps,s}(A_{\De_2})$
such that $\jmath^{\eps,s}_{\Delta_1,*}(y_1)=\jmath^{\eps,s}_{\Delta_2,*}(y_2)$ in $K_0^{\eps,s}(A)$. In view of Lemma \ref{lemma-almost-canonical-form}, we can up to rescaling $\eps$ assume without loss of generality there exist integer $m$ and $n$ with $m\lq n$ and two $\eps$-$r$-projections $q_1$ and $q_2$  in $M_n(A)$ such that
\begin{itemize}
\item $q_1-\diag(I_m,0)$ is an element  in the matrix algebra   $M_n(A_{\Delta_1})$;
\item $q_2-\diag(I_m,0)$  is an element  in the matrix algebra   $M_n(A_{\Delta_2})$;
\item $y_1=[q_1,m]_{\eps,s}$;
\item $y_2=[q_2,m]_{\eps,s}$.
\end{itemize}
Up to stabilization, we can also assume that $q_1$ and $q_2$ are homotopic as $\eps$-$s$-projections in $M_n(A)$.
Let $(\alpha,k)$ be the control pair of Proposition \ref{prop-conjugate} . Up to stabilisation there exists $u$ a $\alpha\eps$-$k_{\eps}s$-unitary in $M_n(A)$ such that
$\|u^*q_1u-q_2\|<\alpha\eps$. Up to replacing
$u$ by $\diag(u,u^*)$, $q_1$ by $\diag(q_1,0)$  and  $q_2$ by $\diag(q_2,0)$, we can assume in view of Lemma \ref{cor-example-homotopy} that $u$ is homotopic to  $I_n$ as a $3\alpha\eps$-$2k_{\eps}s$-unitary in $M_n(A)$.
According to Proposition \ref{proposition-decom-unitaries}, then for some control pair $(\lambda,l)$  depending only on $(\alpha,k)$ and $c$  with $ (\alpha,k)\lq(\lambda,l)$ and up to stabilization, there exist $w_1$ and $w_2$ some $\lambda\eps$-$k_\eps s$ unitaries in $M_n(A)$ such that
\begin{itemize}
\item $w_i-I_k$ is an element  in the matrix algebra   $M_n(C^*N^{(r,4r)}_{\Delta_i,l_\eps s})$ for $i=1,2$;
\item $\|w_1^*q_1w_1-w_2q_2w_2^*\|<\lambda\eps$
\end{itemize}
Notice that $w_1^*q_1w_1-\diag(I_m,0)$ is an element  in the matrix algebra   $M_n(A_{\De_1,(2l_\eps+1) s})$ and $w_2 q_2w_2^*-\diag(I_m,0)$  is an element  in the matrix algebra  $M_n(A_{\De_2,(2l_\eps+1) s})$. By Definition \ref{def-CIA} of the CIA property, there exists $y$ in $$M_n(A_{\De_1,(2l_\eps+1) s}\cap A_{\De_2,(2l_\eps+1) s})$$ such that
$$\|y-(w_1^*q_1w_1^*-\diag(I_n,0))\|<\lambda c\eps$$ and $$\|y-(w_2q_2w_2^*-\diag(I_n,0))\|<\lambda c\eps.$$ Let us set then   $$p=y+\diag(I_m,0).$$  Since
$\|p-w_1^*q_1w_1\|<\lambda c\eps$ and $\|p-w_2q_2w_2^*\|<\lambda c\eps$, up to stabilization , in view  of the proof of  \cite[Lemma 1.9]{oy2} and according to \cite[Lemma 1.7]{oy2}, we know that  for some control pair $(\alpha',l')$ depending only  on $(\alpha,k)$ and $c$ and such that $((c+1)\lambda,2l+1)\lq (\alpha',l')$, then for $j=1,2$
\begin{itemize}
\item   $w_1^*q_1w_1$  is an  $\alpha'$-$l'_\eps s$-projection in $M_n(\widetilde{A}_{\De_1})$;
\item   $w_1^*q_1w_1$ is  homotopic to $q_1$   as   an  $\alpha'$-$l'_\eps s$-projections  in $M_n(\widetilde{A}_{\De_1})$;
\item $p$ is connected to $w_1^*q_1w_1$   as   an  $\alpha'$-$l'_\eps s$-projections  in $M_n(\widetilde{A}_{\De_1})$;
\end{itemize}

and

\begin{itemize}
\item   $w_2q_2w_2^*$  is an  $\alpha'$-$l'_\eps s$-projection in $M_n(\widetilde{A}_{\De_2})$;
\item   $w_2q_2w_2$ is  homotopic to $q_2$   as   an  $\alpha'$-$l'_\eps s$-projections  in $M_n(\widetilde{A}_{\De_2})$;
\item $p$ is connected to $w_2^*q_2w_2$   as   an  $\alpha'$-$l'_\eps s$-projections  in $M_n(\widetilde{A}_{\De_2})$.
\end{itemize}

  Now if we set $x=[p,m]_{\alpha'\eps,l'_\eps s}$ in $K_0^{\alpha'\eps,l'_\eps s}(A_{\De_1}\cap A_{\De_2})$,  we have
that $$\jmath_{\Delta_1,\Delta_2}^{\alpha'\eps,l'_\eps s}(x)=\iota^{\eps,\alpha'\eps,s,l'_\eps s}(y_1)$$ in  $K_*^{\alpha'\eps,l'_\eps s}(A_{\Delta_1})$ and
$$\jmath_{\Delta_2,\Delta_1}^{\alpha'\eps,l'_\eps s}(x)=\iota^{\eps,\alpha'\eps,s,l'_\eps s}(y_2)$$ in  $K_*^{\alpha'\eps,l'_\eps s}(A_{\Delta_2})$.

A similar proof can be carried out in the odd case but we can also use the controlled Bott periodicity \cite[Lemma 4.6]{oy3}. The non unital case can be proved in a similar way using Lemma \ref{proposition-decom-unitaries-non-unital}, noticing that in view of the proof of Proposition \ref{prop-conjugate} and following the proof of the unital case above above, we can  assume that $u$ which is now a 
$\alpha_h\eps$-$k_{h,\eps}s$-unitary in $M_n(\tilde{A})$  is such that $u-I_n$ has coefficient in $A$ (see the proofs of \cite[Lemma 1.11 and  Corollary 1.31]{oy2}).

\end{proof}

\subsection{Quantitative  boundary maps for  controlled Mayer-Vietoris pair}\label{subs-section-boundary}

In this subsection, we introduce the quantitative boundary map in the controlled Mayer-Vietoris sequence for quantitative $K$-theory of filtered $C^\ast$-algebras. 
\begin{lemma}\label{lemma-technic}For every positive number $c$, there exists a control pair $(\lambda,k)$ such that the following holds:

\smallskip

Let $A$ be a unital $C^*$-algebra filtered by $(A_s)_{s>0}$, let $r$  be a  positive number and  let $(\De_1,\De_2,A_{\Delta_1},A_{\Delta_2})$ be a  $r$-controlled  weak Mayer-Vietoris pair  for $A$ with coercitivity $c$. Let   $\eps$ and $s$ be positive numbers  with  $\eps<\frac{1}{4\lambda}$ and  $s\lq r/2$, let  $m$ and $n$ be integers and let  $u$ in $U_n^{\eps,s}(A)$,  $v$ in $U_m^{\eps,s}(A)$ and $w_1,\,w_2$ be $\eps$-$s$-unitaries  in $M_{n+m}^{\eps,s}(A)$ such that
\begin{itemize}
\item $w_i-I_{n+m}$ is an element  in the matrix algebra   $M_{n+m}({A_{\Delta_i}})$ for $i=1,2$;
\item $\|\diag(u,v)-w_1w_2\|<\eps$.
\end{itemize}
Then,
\begin{enumerate} 
\item there exists  a $\lambda\eps$-$k_\eps s$-projection $q$ in $M_{n+m}(A)$ such that
\begin{itemize}
\item $q-\diag (I_n,0)$  is an element  in the matrix algebra   $M_{n+m}(A_{\Delta_{1}}\cap A_{\Delta_{2}})$;
\item $\|q-w_1^*\diag (I_n,0)w_1\|<\lambda\eps$;
\item $\|q-w_2\diag (I_n,0)w_2^*\|<\lambda\eps$.
\end{itemize}
\item if $q$ and $q'$ are two $\lambda\eps$-$k_\eps s$-projections   in $M_{n+m}(A)$  that satisfies the first  point, then $[q,n]_{\lambda\eps,k_\eps s}=[q',n]_{\lambda\eps,k_\eps s}$ in $K_0(A_{\De_1}\cap A_{\De_2})$.
\item Let  $(w_1,w_2)$ and $(w'_1,w'_2)$ be two pairs of   $\eps$-$s$-unitaries  in $M_{n+m}^{\eps,s}(A)$ satisfying the assumption of the lemma and let $q$ and $q'$ be  $\lambda\eps$-$k_\eps s$-projections  in $M_{n+m}(A)$  that satisfies the first point relatively to  respectively  $(w_1,w_2)$ and $(w'_1,w'_2)$, then $[q,n]_{\lambda\eps,k_\eps s}=[q',n]_{\lambda\eps,k_\eps s}$ in $K_0(A_{\De_1}\cap A_{\De_2})$. \end{enumerate}
\end{lemma}
\begin{proof}
Since  $\diag(u,v)$ is  an $\eps$-$s$-unitary, we have that
\begin{equation*}\begin{split}\|w_1^*\diag (I_n,0)w_1-w_1^*\diag(u,v)\diag (I_n,0)&\diag(u^*,v^*)w_1\|\\
&=\|w_1^*\diag (I_n-u^*u,0)w_1\| \\
&<2\eps.\end{split}\end{equation*} Since  $\|w_1^*\diag(u,v)-w_2\|<4\eps$, we deduce
that
$$ \|w_1^*\diag (I_n,0)w_1-w_2\diag (I_n,0)w_2^*\| <8\eps.$$
With notations  as in Definition
\ref{definition-coarse-MV-pair}, let $y$ be an element in $M_{n+m}(A_{\Delta_{1, 2s}}\cap A_{\Delta_{2,2s}} )$ such that
$$ \|w_1^*\diag (I_n,0)w_1-\diag(I_n,0)-y\| <8c\eps$$ and
$$ \|y-w_2\diag (I_n,0)w_2^*-\diag(I_n,0)\| <8c\eps$$
and set   $$q=y+\diag (I_n,0).$$
Then  $q$ is close to a $2\eps$-$2s$-projection and thus  we obtain in view of Lemma \ref{lemma-almost-closed}  that there exists a control pair $(\lambda,k)$, depending only on $c$ such that the conclusion of the first point is  satisfied.
 With notations as in Lemma \ref{lemma-technic} and in view of Lemma \ref{lemma-almost-closed},   if     $q$ and $q'$ are  $\lambda\eps$-$k_\eps s$-projections of $M_{n+m}(A)$ that satisfies the first point, then   $$[q,n]_{10\lambda\eps,k_\eps s}=[q',n]_{10\lambda\eps,k_\eps s}.$$
 
If  $(w_1,w_2)$ and $(w'_1,w'_2)$ are  two pairs of   $\eps$-$s$-unitaries  in $M_{n+m}^{\eps,s}(A)$ that satisfy the assumption of the lemma and let $q$ and $q'$ be  $\lambda\eps$-$k_\eps s$-projections  in $M_{n+m}(A)$  that satisfy the first point relatively to  $(w_1,w_2)$ and $(w'_1,w'_2)$.
Then $\|{w_1}w_2-w'_1w'_2\|<2\eps$ and hence $\|{w'_2}^*w_2-w'_1{w_1}^*\|<10\eps$. Hence using the CIA condition, we see that there exists
$v$ in $M_{n+m}(A_{2s})$ such that $v-I_{n+m}$ is in $M_{n+n}(A_{\De_1}\cap A_{\De_2})$,  $\|v-w'_1{w_1}^*\|<10c\eps$ 
and  $\|{w'_2}^*w_2-v\|<10c\eps$. Since we have then $\|w_1-v^*w'_1\|<30c\eps$ and   $\|w_2-w'_2v\|<30c\eps$  the last point is consequence of \cite[Lemma 1.9]{oy2} and of the second 
point applied to $45c\eps,\, (w_1,w_2)\,,q$ and $v^*qv$.

%
 \end{proof}

\begin{remark}\label{remark-unicite-lemme-technic}\
We have a similar statement in the non-unital case with  $u$ in $U_n^{\eps,s}(\tilde{A})$ and  $v$ in $U_m^{\eps,s}(\tilde{A)}$  such that $u-I_n$ and $v-I_m$ have coefficients in $A$
\end{remark}

We are now in position the define the boundary map associated to a controlled Mayer-Vietoris pair.
Let $A$ be a  filtered   $C^*$-algebra and  let $(\De_1,\De_2,A_{\Delta_1},A_{\Delta_2})$ be a  $r$-controlled  weak Mayer-Vietoris pair  for $A$ with coercitivity $c$. Assume first that $A$ is unital.

Let $(\alpha,l)$ be a control pair as is Proposition \ref{proposition-decom-unitaries}. For any positive numbers $\eps$ and $s$ with $\eps<\frac{1}{4\alpha}$ and $s\lq r/2$  and any $\eps$-$s$-unitary $u$ in $\M_n(A)$, let $v$ be an $\eps$-$s$-unitary in some $\M_m(A)$ such that
$\diag(u,v)$ is homotopic to $I_{n+m}$ as a $3\eps$-$2s$-unitary in  $\M_{n+m}(A)$, we can take for instance $v=u^*$ (see Lemma \ref{cor-example-homotopy}). Since  $C^*N^{(2s,8s)}_{\De_i}\subset A_{\De_i}$  as a filtered subalgebra for $i=1,2$, then according to Proposition \ref{proposition-decom-unitaries} and up to replacing $v$ by $\diag(v,I_k)$ for some integer $k$, there exists   $w_1$ and $w_2$ two $3\alpha\eps$-$2l_{3\eps} r$-unitaries  in $M_{n+m}(A)$  such that
\begin{itemize}
\item $w_i-I_{n+m}$ is an element  in the matrix algebra  $M_{n+m}(A_{\Delta_i,2l_{3\eps} s})$ for $i=1,2$;
\item for $i=1,2$, there exists  a homotopy $(w_{i,t})_{t\in[0,1]}$ of  $3\alpha\eps$-$2l_{3\eps} s$-unitaries between  $1$ and $w_i$ such that
$w_{i,t}-I_{n+m}$ is an element  in the matrix algebra   $M_{n+m}(A_{\De_i,l_{3\eps} s})$ for all $t$ in $[0,1]$.
\item $\|\diag(u,v)-w_1w_2\|<3\alpha\eps$.
\end{itemize}
Let $(\lambda,k)$ be the control pair of  Lemma \ref{lemma-technic} (recall that $(\la,k)$ on depends only  $c$). Then if $\eps$ is in $(0,\frac{1}{12\al\la})$,  there exists
   a $3\alpha\lambda\eps$-$2l_{3\eps} k_{3\alpha\eps} s$-projection $q$ in $M_{n+m}(A)$ such that
\begin{itemize}
\item $q-\diag (I_n,0)$ is an element  in the matrix algebra    $$M_{n+m}(A_{\Delta_{1,2l_{3\eps} k_{3\alpha\eps} s}}\cap A_{\Delta_{2, 2l_\eps k_{3\alpha\eps} s}});$$
\item $\|q-w_1^*\diag (I_n,0)w_1\|<3\alpha\lambda\eps$;
\item $\|q-w_2\diag (I_n,0)w_2^*\|<3\alpha\lambda\eps$.
\end{itemize}
In view of second point of Lemma   \ref{lemma-technic},  the class $[q,n]_{3\alpha\lambda\eps,2l_{3\eps} k_{3\alpha\eps} s}$ in
$$K_0^{3\alpha\lambda\eps,2l_{3\eps} k_{3\alpha\eps} s}(A_{\Delta_1}\cap A_{\Delta_2})$$ does not depend on the choice of $q$.
Set then $\alpha_c=3\alpha\lambda$ and $$k_c:\left(0,\frac{1}{4\alpha_c}\right)\lto (1,+\infty),\,\eps\mapsto 2l_{3\eps} k_{3\alpha\eps}$$ and define
$\partial^{\eps,s,1}_{\Delta_1,\Delta_2,*}([u]_{\eps,s})=[q,n]_{\alpha_c\eps,k_c s}$ and let us prove that we define in this way a morphism
$$\partial^{\eps,s,1}_{\Delta_1,\Delta_2,*}:K_1^{\eps,s}(A)\to K_0^{\alpha_c\eps,k_c s}(A_{\Delta_1}\cap A_{\Delta_2}).$$
It is straightforward to check that (compare with \cite[Chapter 8]{we}).
\begin{itemize}
\item two choices of elements satisfying the conclusion of   Lemma \ref{lemma-technic} relative to $\diag(u,v)$ give
  rise to homotopic elements in
  $\P_{n+j}^{\alpha_\D\eps,k_\D s}(A_{\Delta_1}\cap A_{\Delta_2})$ (this is a consequence
  of Lemma \ref{lemma-technic}).
\item Replacing $u$ by $\diag(u,I_m)$
and $v$ by $\diag(v,I_k)$ gives also rise to  the same element of $K_0^{\alpha_c\eps,k_c s}(A_{\Delta_1}\cap A_{\Delta_2})$.
\end{itemize}

Applying now  Proposition \ref{proposition-decom-unitaries}  to  the  $r$-controlled Mayer-Vietoris pair
$$(C([0,1],\Delta_1),C([0,1], \Delta_2),C([0,1],A_{\Delta_1}),C([0,1], A_{\Delta_2}))$$   for the  $C^*$-algebra
$C([0,1],A)$ filtered by $(C([0,1],A_s))_{s>0}$,
  we see that
$\partial^{\eps,s,1}_{\Delta_1,\Delta_2,*}([u]_{\eps,s})$
\begin{itemize}
\item only depends on the class of $u$ in $K_1^{\eps,s}(A)$;
\item does not depend on the choice of $v$ such that $\diag(u,v)$ is
  connected to $I_{n+j}$ in $U_{n+j}^{3\eps,2s}(A)$.
\end{itemize}
In the non unital case $\partial^{\eps,s,1}_{\Delta_1,\Delta_2,*}$ is defined similarly by using point (ii) of Remark \ref{remark-unicite-lemme-technic}, noticing that in view of Lemma \ref{lemma-almost-canonical-form-odd} and up to replacing $\eps$ by $3\eps$, every element $x$ in $K_1(A)$ is the class of a $\eps$-$r$-unitary $u$  in $M_n(\tilde{A})$ such that $u-I_n$ has coefficients in $A$.
It is straightforward to check that
$\partial^{\bullet,\bullet,1}_{\Delta_1,\Delta_2,*}$ is compatible with the structure isomorphisms.
Then if we set $$\DD^1_{\Delta_1,\Delta_2,*}=(\partial^{\eps,s,1}_{\Delta_1,\Delta_2,*})_{0<\eps<\frac{1}{4\al_c},0<s<\frac{r}{k_{c,\eps}}},$$ then
 $$\DD^1_{\Delta_1,\Delta_2,*}:\K_1(A)\to\K_0(A_{\De_1}\cap A_{\De_1})$$ is a  odd degree $(\al_c,k_c)$-controlled morphism of order $r$.

 Let us now define the boundary map in the even case using controlled Bott periodicity.
For $\De$ a closed subspace in an $C^*$-algebra, let us define its suspension as $S\De=C_0((0,1),A)$.
Let $[\partial]$ be the element of $KK_1(\C,C_0(0,1))$ that implements the extension
$$0\to C_0(0,1)\to C_0[0,1) \stackrel{ev_0}{\to} \C\to 0,$$ where $ev_0:C_0[0,1) {\to} \C$ is the evaluation at $0$.
Then $[\partial]$ is an invertible  element of $KK_1(\C,C_0(0,1))$ and according to \cite[Lemma 1.6]{oy2},
$$\T_B([\partial]):\K_*(B)\to \K_*(SB)$$ is a $(\al_\T,k_\T)$-controlled isomorphism of degree one with controlled inverse
$$\T_B([\partial]^{-1}):\K_*(SB)\to \K_*(B).$$
Let $A$  be a    $C^*$-algebra filtered by $(A_s)_{s>0}$, let $r$  be a  positive number and  
let $(\De_1,\De_2,A_{\Delta_1},A_{\Delta_2})$ be a  $r$-controlled weak  Mayer-Vietoris pair for $A$ with coercitivity $c$. Then  $(S\De_1,S\De_2,SA_{\Delta_1},SA_{\Delta_2})$ is a  $r$-controlled  weak Mayer-Vietoris pair for $SA$ (filtered by $(SA_r)_{r>0}$) with coercitivity $c$.
Set then $\lambda=\alpha_\T^2\alpha_c$ and $h_\eps=k_{\T,\al_\T\al_c\eps}k_{c,\al_c\eps}k_{c,\eps}$.
Let us define in the even case the quantitative boundary map for the $r$-controlled  Mayer-Vietoris pair  
$(\De_1,\De_2,A_{\Delta_1},A_{\Delta_2})$ as the $(\lambda,h)$-controlled morphism of order $r$
$$\DD^0_{\De_1,\De_2,*}\defi \T_{A_{\De_1}\cap A_{\De_2},*}([\partial]^{-1})\circ
\DD^1_{S\De_1,S\De_2,*}\circ\T_A([\partial]):\K_0(A)\lto \K_1(A_{\De_1}\cap A_{\De_2}).$$
For sake of simplicity, we will  rescale $(\al_c,k_c)$ to $(\lambda,h)$ and use the same control pair in the odd and in the even case for the definition of 
$$\DD_{\De_1,\De_2,*}\defi \DD^0_{\De_1,\De_2,*}\oplus \DD^1_{\De_1,\De_2,*}:\K_*(A)\lto \K_{*+1}(A_{\De_1}\cap A_{\De_2})$$ as a odd degree $(\al_c,k_\D)$-controlled morphism of order $r$.
Notice that the quantitative boundary map of a  $r$-controlled  weak Mayer-Vietoris pair is natural in the following sense: let $A$ and $B$ be filtered $C^*$-algebras, let  $(\De_1,\De_2,A_{\Delta_1},A_{\Delta_2})$  and 
$(\De'_1,\De'_2,B_{\Delta'_1},B_{\Delta'_2})$ be respectively 
  $r$-controlled  weak Mayer-Vietoris pairs  for $A$ and $B$ with coercitivity $c$    and let $f:A\to B$ be a filtered morphism such that $f(\De_1)\subseteq \De'_1,\,f(\De_2)\subseteq \De'_2,\,f(A_{\De_1})\subseteq B_{\De'_1}$ and$f(A_{\De_2})\subseteq B_{\De'_2}$. Then we have
  \begin{equation}\label{equ-boundary-natural}
f_{/A_{\De_1}\cap A_{\De_2},*} \circ \DD_{\De_1,\De_2,*}=\DD_{\De'_1,\De'_2,*} \circ f_*,\end{equation}
where $f_{/A_{\De_1}\cap A_{\De_2}}:A_{\De_1}\cap A_{\De_2}\to B_{\De'_1}\cap B_{\De'_2}$ is the restriction of $f$ to $A_{\De_1}\cap A_{\De_2}$.
\subsection{The controlled six-term exact sequence}
In this subsection, we  state the Mayer-Vietoris    controlled six term exact sequence associated to a $r$-controlled Mayer-Vietoris of order $r$.
\begin{lemma}\label{lemma-bound1}
There exists a control pair $(\lambda,l)$ such that 
\begin{itemize}
\item for any unital filtered $C^*$-algebra $A$ filtered by $(A_s)_{s>0}$ and any subalgebras  $A_1$ and $A_2$   of $A$ such that $A_1,\,A_2$  and $A_1\cap A_2$ are respectively filtered by     $(A_1\cap A_r)_{r>0},\,(A_2\cap A_r)_{r>0}$ and $(A_1\cap A_2\cap A_r)_{r>0}$;
\item for any positive number $\eps$  with  $\eps<\frac{1}{4\lambda}$   any integers $n$ and $m$ and any $\eps$-$r$-unitaries $u_1$ in $M_n(A)$ and $u_2$ in $M_m(A)$;
\item for any $\eps$-$r$-unitaries $v_1$ and $v_2$ respectively in $M_{n+k}(\widetilde{A_1})$ and 
$M_{n+k}(\widetilde{A_2})$ such that  
\begin{itemize} 
\item $\|\diag(u_1,u_2)-v_1v_2\|<\eps$;
\item there exists an $\eps$-$r$-projection  $q$ in $M_{n+m}(A)$ such that $q-\diag(I_n,0)$ is in $M_{n+m}(A_{1}\cap A_{2}),\,\|q-v_1^*\diag(I_n,0) v_1\|<\eps$ and $[q,n]_{\eps,r}=0$ in $K^{\eps,r}_0(A_{1}\cap A_{2})$.
\end{itemize}
\end{itemize}
Then there exists an integer $k$ and  $\lambda \eps$-$l_{\eps} r$-unitaries $w_1$ and $w_2$ respectively in  $M_{n+k}(\widetilde{A_{1}})$ and $M_{n+k}(\widetilde{A_{2}})$ such that
 $\|\diag(u_1,I_k)-w_1w_2\|<\la \eps$.
 Moreover, if $v_i-I_{n+k}$ lies in $M_{n+k}(A_i)$ for $i=1,2$ then $w_1$ and $w_2$ can be chosen such that
  $w_i-I_{n+k}$ lies in $M_{n+m}(A_i)$ for $i=1,2$\end{lemma}
\begin{proof}
  Up to replacing $u_2,\,v_1$ and $v_2$ respectively by $\diag(u_2,I_k),\,\diag(v_1,I_k)$ and $\diag(v_2,I_k)$ for  some integer $k$, we can assume that $q$ is homotopic to $\diag(I_n,0)$ as an
$\eps$-$r$-projection in $M_{n+m}(\widetilde{A_{1}\cap A_{2}})$. According to 
Lemma \ref{prop-conjugate}, there exist
\begin{itemize}
\item  a control pair $(\alpha,h)$;
\item  up to stabilization an $\alpha\eps$-$h_{\eps} r$-unitary
$v$ in $M_{n+m}(\widetilde{A_{\De_1}\cap A_{\De_2}})$ with  $v-I_{n+m}$  in $M_{n+m}({A_{\De_1}\cap A_{\De_2}})$\end{itemize}
 such that $$\|q-v\diag(I_n,0)v^*\|<\alpha\eps.$$ Up to take a larger 
control pair  $(\alpha,h)$, we can assume that  
$$\|v_1^*\diag(I_n,0)v_1-v\diag(I_n,0)v^*\|<\alpha\eps$$ and $$\|v_2\diag(I_n,0)v_2^*-v\diag(I_n,0)v^*\|<\alpha\eps$$
and hence even indeed that
$$\|v^*v_1^*\diag(I_n,0)v_1v-\diag(I_n,0)\|<\alpha\eps$$ and $$\|v^*v_2\diag(I_n,0)v_2^*v-\diag(I_n,0)\|<\alpha\eps.$$Hence, for some control pair $(\alpha',h')$ depending only on $(\alpha,h)$, there exist  $\alpha'\eps$-$h'_{\eps} s$-unitaries $v'_1$ in   $M_n(\widetilde{A_{\De_1})}$,  $v''_1$ in $M_m(\widetilde{A_{\De_1})}$, $v'_2$ in   $M_n(\widetilde{A_{\De_2})}$,  $v''_2$ in $M_m(\widetilde{A_{\De_2})}$ such that
$\|v_1v-\diag(v_1',v''_1)\|<\alpha'\eps$ and $\|v^*v_2-\diag(v_2',v''_2)\|<\alpha'\eps$. Thus,  for a control pair  $(\alpha'',h'')$ depending only on $(\alpha',h')$ we have,
$$\|\diag(u_1,u_2)-\diag(v_1'v_2',v''_1v''_2)\|<\alpha''\eps.$$Hence we deduce that   $\|u_1-v_1'v_2'\|<\alpha''\eps$.
\end{proof}
\begin{corollary}\label{cor-bound1}
For any positive number $c$, there exists a control pair $(\lambda,l)$ such that 
\begin{itemize}
\item for any  filtered $C^*$-algebra $A$;
\item for any positive number $r$ and any   $r$-controlled  weak Mayer-Vietoris pair $(\De_1,\De_2,A_{\Delta_1},A_{\Delta_2})$    for $A$ with coercitivity $c$;
\item for any positive numbers $\eps,\,\eps'$ and  $r'$ with   $0<\al_c\eps\lq \eps'<\frac{1}{4\lambda}$ and $r'\gq k_{c,\eps} r$
\end{itemize}
then for any $y$ in $K_1^{\eps,r}(A)$ such that $$\iota_*^{\al_c\eps,k_{c,\eps} r,\eps',r'}\circ\partial^{\eps,r} _{\De_1,\De_2,*}(y)=0$$ in 
$K_1^{\eps',r'}(A_{\De_1}\cap A_{\De_2})$,  there exist $x_1$ in $K_1^{\lambda\eps',l_{\eps' }r'}(A_{\De_1})$ and 
 $x_2$ in $K_1^{\lambda\eps',l_{\eps'}r'}(A_{\De_2})$ such that 
$$\iota_*^{\eps, \lambda\eps',r,l_{\eps'}r'}(y)=\jmath^{\lambda\eps',l_{\eps'}r'}_{\Delta_1,*}(x_1)-\jmath^{\lambda\eps',l_{\eps' }r'}_{\Delta_2,*}(x_2).$$\end{corollary}
\begin{proof}Let us assume for sake of simplicity that $A$ is unital, the non unital being similar (just extra notation are added).
Let $y$ be an element in $K^{\eps,r}_1(A)$  such that $\iota_*^{\al_c\eps,h_{\al_c,\eps} r,\eps',r'}\circ\partial^{\eps,r} _{\De_1,\De_2,*}(y)=0$ in 
$K_1^{\eps',r'}(A_{\De_1}\cap A_{\De_2})$. Let $(\la,l)$ be the controlled pair of Lemma \ref{lemma-bound1} and let $u$ be an $\eps$-$r$-unitary in some $M_n(A)$ such that $y=[u]_{\eps,r}$. Then according to the definition of  
$\partial^{\eps,r} _{\De_1,\De_2,*}(y)$, we see by using Lemma \ref{lemma-bound1} that  up  to replacing $u$ by $\diag(u,I_m)$ for some integer $m$, there exists 
two $\lambda \eps'$-$l_{\eps'} r'$-unitaries $w_1$ and $w_2$ respectively in  $M_{n}(\widetilde{A_{\De_1}})$ and
 $M_{n}(\widetilde{A_{\De_2}})$ such that   $\|u-w_1w_2\|<\la \eps'$.
 Then  $u$ is homotopic
to $w_1w_2$ as a $4\la\eps'$-$l_{\eps'}r'$-unitary in $M_{2n}(A)$. From this we deduce that
\begin{eqnarray*}
\iota_*^{\eps,4\la\eps' ,r, 2l_{\eps'} r'}(y)&=&[w_1w_2]_{4\la\eps',2l_{\eps'}r'}\\
&=&[w_1]_{4\la\eps',2l_{\eps'}r'}+[w_2]_{4\la\eps',2l_{\eps'}r'}\\
&=&\jmath_{\De_1,*}^{4\la\eps',2l_{\eps'}r'}(x_1)+\jmath_{\De_2,*}^{4\la\eps',2l_{\eps'}r'}(x_2)
\end{eqnarray*}
with $x_1=[w_1]_{4\la\eps',2l_{\eps'}r'}$ in $K_1^{4\la\eps,2l_{\eps}r'}(A_{\De_1})$ and
 $x_2=[w_2]_{4\la\eps,2l_{\eps'}r'}$ in $K_1^{4\la \eps,2l_{\eps'}r'}(A_{\De_2})$.
\end{proof}
As a consequence, we get following controlled exactness result.

\begin{proposition}\label{prop-bound1}For any positive number $c$, there exists a control pair $(\alpha,l)$ such that 
for any $C^*$-algebra $A$ filtered by  $(A_s)_{s>0}$, any positive number $r$ and any   $r$-controlled  weak Mayer-Vietoris pair $(\De_1,\De_2,A_{\Delta_1},A_{\Delta_2})$    for $A$ with coercitivity $c$ then the composition
$$
\K_1(A_{\Delta_1})\oplus \K_1( A_{\Delta_2})
\stackrel{\jmath_{\Delta_1,*}-\jmath_{\Delta_2,*}}{-\!\!\!-\!\!\!-\!\!\!-\!\!\!-\!\!\!-\!\!\!-\!\!\!-\!\!\!\lto}\K_1(A)
\stackrel{\DD_{\Delta_1,\Delta_2,*}}{-\!\!\!-\!\!\!-\!\!\!-\!\!\!-\!\!\!\lto}\K_0(A_{\Delta_1}\cap A_{\Delta_2})$$ is 
$(\alpha,l)$-exact at order $r$.
\end{proposition}
\begin{lemma}\label{lemma-technic2}There exists a control pair $(\lambda,h)$ such that the following holds:

\smallskip

\begin{itemize}
\item Let $A$ be a unital $C^*$-algebra filtered by $(A_r)_{r>0}$ and let $A_1$ and $A_2$ be subalgebras of $A$ such that
$A_1,\,A_2$ and $A_1\cap A_2$ are respectively filtered by     $(A_1\cap A_r)_{r>0},\,(A_2\cap A_r)_{r>0}$ and
$(A_1\cap A_2\cap A_r)_{r>0}$;
\item let $\eps$ and $s$ be positive numbers with $\eps<\frac{1}{4\la}$;
\item  let $n$ and $N$ be positive integers with $n\lq N$ and let $p$ an $\eps$-$s$ projection in $M_N(\widetilde{A_1\cap A_2})$ such that $\rho_{A_1\cap A_2}(p)=\diag(I_n,0)$.
\end{itemize}
Assume that 
that
 \begin{itemize}
 \item $p$ is homotopic to $\diag(I_n,0)$ as an $\eps$-$s$-projection in
 $M_N(\widetilde{A_{1}})$;
 \item $p$ is  homotopic to $\diag(I_n,0)$ as an $\eps$-$s$-projection in
 $M_N(\widetilde{A_{2}})$.
 \end{itemize}Then there exist  an integer $N'$ with $N'\gq N$, $w_1$ and $w_2$ in $U_{N'}^{\la\eps,h_\eps s}(A)$,  $u$ in $U_n^{\la\eps,h_\eps s}(A)$ and $v$ in $U_{N'-n}^{\la\eps,h_\eps s}(A)$  such that
\begin{itemize}
\item $w_i-I_{N'}$ is an element  in   $M_{N'}(A_{i})$ for $i=1,2$;
\item $$\|w_1^*\diag(I_n,0)w_1-diag(p,0)\|<\la\eps$$ and
$$\|w_2\diag(I_n,0)w_2^*-\diag(p,0)\|<\la\eps.$$
 \item for $i=1,2$, then $w_i$ is connected to $I_{N'}$ by a homotopy
of $\la\eps$-$h_\eps s$-unitaries  $(w_{i,t})_{t\in[0,1]}$ in $M_{N'}(A)$ such  that $w_{i,t}-I_{N'}$ is in $M_{N'}(A_{i})$ for all $t$ in $[0,1]$. 
\item $\|\diag(u,v)-w_1w_2\|<\la\eps$.
\end{itemize}
\end{lemma}
\begin{proof}
 Let $(\al,k)$ be the control pair of
Proposition  \ref{prop-conjugate}, then there exist up to stabilization
\begin{itemize}
\item $w_1$ an $\al\eps$-$k_{\eps}s$-unitary in $M_N(\widetilde{A_{1}})$;
\item $w_2$ an $\al\eps$-$k_{\eps}s$-unitary in $M_N(\widetilde{A_{2}})$,
\end{itemize}
 such  that
$$\|w_1^*\diag(I_n,0)w_1-p\|<\al_h\eps$$ and
$$\|w_2\diag(I_n,0)w_2^*-p\|<\al_h\eps.$$  Up to  replacing $w_1$and $w_2$ respectively by $\rho_{A_1}(w_1^{-1})w_1$ and  
$w_2\rho_{A_2}(w_2^{-1})$ and up to replacing $\al$ by $4\al$,  we can assume that
$w_1-I_N$ is an element  in the matrix algebra   $M_N(A_{1})$ and $w_2-I_N$ is an element  in the matrix algebra   $M_N(A_{2})$.  
Hence there
exists a control pair $(\al',k')$, depending only on $(\al,k)$ and that we can choose larger such that
\begin{equation}\label{equ-proof-half-exacness-1} \|w_1w_2\diag(I_n,0)w_2^*w_1^*-\diag(I_n,0)\|<\al'\eps\end{equation}
and
\begin{equation}\label{equ-proof-half-exacness-2} \|w_2^*w_1^*\diag(I_n,0)w_2w_1-\diag(I_n,0)\|<\al'\eps.\end{equation}

Up to replacing $w_1\,,w_2,\,p$ and $(\al',k')$  respectively by $\diag(w_1,w_1^*)$, $\diag(w_2,w_2^*)$, $\diag(p,0)$ and    $(3\al,2k)$, we can assume that $w_i$ for $i=1,2$ is connected to $I_N$ by a homotopy
of $\al\eps$-$k_\eps s$-unitaries  $(w_{i,t})_{t\in[0,1]}$ in $M_N(A)$ such that $w_{i,t}-I_N$ is in $M_N(A_{i})$ for all $t$ in $[0,1]$. 
Equations (\ref{equ-proof-half-exacness-1}) and  (\ref{equ-proof-half-exacness-2}) imply that
 for  a control pair $(\al'',k'')$, depending only on $(\al',k')$, there
exist  $u$ and $v$ some $\al''\eps$-$k''_{\eps}s$-unitaries respectively in
 $M_n(A)$ and  $M_{N-n}(A)$ such that
 $$\|\diag(u,v)-w_1w_2\|<\al''\eps.$$
 \end{proof}

\begin{proposition}\label{prop-bound2}For every positive number $c$, there exists a control pair $(\al,l)$ such that 
for any filtered  $C^*$-algebra $A$,  any positive number $r$   and
any $r$-controlled  weak Mayer-Vietoris pair $ (\De_1,\De_2,A_{\Delta_1},A_{\Delta_2})$  for $A$ of order $r$ with coercitivity $c$, then the 
 composition
$$
\K_1(A)
\stackrel{\DD_{\Delta_1,\Delta_2,*}}{-\!\!\!-\!\!\!-\!\!\!-\!\!\!-\!\!\!\lto}\K_0(A_{\Delta_1}\cap A_{\Delta_2})\stackrel{(\jmath_{\Delta_1,\De_2,*},\jmath_{\Delta2,\De_1,*})}{-\!\!\!-\!\!\!-\!\!\!-\!\!\!-\!\!\!-\!\!\!-\!\!\!-\!\!\!-\!\!\!-\!\!\!-\!\!\!-\!\!\!-\!\!\!-\!\!\!-\!\!\!\lto}\K_0(A_{\Delta_1})\oplus \K_0( A_{\Delta_2})$$ is 
$(\alpha,l)$-exact at order $r$.
\end{proposition}
\begin{proof} As in the previous proposition, let us assume that $A$ is unital. Let $y$ be an element in $K^{\eps,s}_0(A_{\De_1}\cap A_{\De_2})$  such that $\jmath^{\eps,s}_{\De_1,\De_2,*}(y)=0$ in
$K^{\eps,s}_0(A_{\De_1})$ and $\jmath^{\eps,s}_{\De_2,\De_1,*}(y)=0$ in $K^{\eps,s}_0(A_{\De_2})$.
Let $p$ be an $\eps$-$r$-projection in some $M_N(\widetilde{A_{\De_1}\cap A_{\De_2}})$ and $n$ be an integer
 such that  $y=[p,n]_{\eps,s}$. In view of Lemma \ref{lemma-almost-canonical-form} and up to replacing $\eps$ by $5\eps$, we can assume without loss of generality that $N\gq n$ and that $$\rho_{A_{\De_1}\cap A_{\De_2}}(p)=\diag(I_n,0).$$ Up to stabilization, we can also
 assume that
 \begin{itemize}
 \item $p$ is homotopic $\diag(I_n,0)$ as an $\eps$-$s$-projection in
 $M_N(\widetilde{A_{\De_1}})$;
 \item $p$ is homotopic $\diag(I_n,0)$ as an $\eps$-$s$-projection in
 $M_N(\widetilde{A_{\De_2}})$.
 \end{itemize}
 Let $(\al,k)$ be a control pair, $N'$ be an integer with $N'\gq N$, let $w_1$ and $w_2$ be in $U_{N'}^{\al\eps,k_\eps s}(A)$,  let $u$ be  in $U_n^{\al\eps,k_\eps s}(A)$ and let $v$ be  in $U_{N'-n}^{\al\eps,k_\eps s}(A)$  that satisfy all together the conclusion of Lemma \ref{lemma-technic2}.
 Since  $\|\diag(u,v)-w_1w_2\|<\al\eps$, we can  up to replacing $(\al,k)$ by $(4\al,2k)$ and in view of Lemma  \ref{lemma-almost-closed} moreover assume that $\diag(u,v)$ is homotopic
 to $I_N'$ as an $\al\eps$-$h_\eps r$-unitary of $M_{N'}(A)$.
 Since 
$$\|w_1^*\diag(I_n,0)w_1-\diag(p,0)\|<\al\eps$$ and
$$\|w_2\diag(I_n,0)w_2^*-\diag(p,0)\|<\al\eps$$ and in view of the definition of 
 the quantitative  boundary map of a control Mayer-Vietoris pair, there exists a control pair $(\al',k')$ depending only on 
 $(\al,k)$ and $c$ such that
$$\partial^{\al'\eps ,k'_\eps s}_{\De_1,\De_2,*}([u]_{\al''\eps ,k'_\eps s})=[p,l]_{\eps,\al'\alpha_c\eps ,s, k'_\eps k_{c,\lambda\eps} s}.$$Hence, if we set $x=[u]_{\al'\eps ,k'_\eps s}$, we get
$$\partial^{\al'\eps ,k'_\eps s}_{\De_1,\De_2,*}(x)=\iota_*^{\eps,\al'\alpha_c\eps ,s, k'_\eps k_{c,\lambda\eps} s}(y).$$
\end{proof}

Collecting together Propositions \ref{prop-half-exact}, \ref{prop-bound1} and \ref{prop-bound2} and using naturality of quantitative Bott isomorphism, we obtain the   controlled six terms exact sequence (at order $r$) associated to a weak $r$-controlled Mayer-Vietoris sequence.
\begin{theorem}\label{thm-six-term-exact-sequence}
For every positive number $c$, there exists a control pair $(\lambda,h)$ such that 
for any   $C^*$-algebra $A$ filtered by $(A_s)_{s>0}$,  any positive number $r$   and
any $r$-controlled  weak Mayer-Vietoris pair $ (\De_1,\De_2,A_{\Delta_1},A_{\Delta_2})$  for $A$ of order $r$, then 
we have a $(\lambda,h)$-exact six term exact sequence of order $r$:

$$\begin{CD}
\K_0(A_{\Delta_1}\cap A_{\Delta_2}) @>(\jmath_{\Delta_1,\De_2,*},\jmath_{\Delta2,\De_1,*})>> \K_0(A_{\Delta_1})\oplus \K_0( A_{\Delta_2}) @>\jmath_{\Delta_1,*}-\jmath_{\Delta_2,*}>>\K_0(A)\\
    @AA\DD_{\Delta_1,\Delta_2,*} A @.     @V\DD_{\Delta_1,\Delta_2,*} VV\\
\K_1(A) @<\jmath_{\Delta_1,*}-\jmath_{\Delta_2,*}<<\K_1(A_{\Delta_1})\oplus \K_1( A_{\Delta_2})@<(\jmath_{\Delta_1,\De_2,*},\jmath_{\Delta2,\De_1,*})<< \K_1(A_{\Delta_1}\cap A_{\Delta_2})
\end{CD}.
$$
\end{theorem}

\subsection{Quantitatively $K$-contractible $C^*$-algebra}In numerous case, the proof of the Baum-Connes conjecture and of its generalization amounts to show that the $K$-theory of some  obstruction
algebra vanishes \cite{y2}. In this subsection, we apply the controlled Mayer-Vietoris  six-term exact sequence to this situation.
\begin{definition}
Let $A$ be a filtered $C^*$-algebra.
$A$ is called  {\bf quantitatively $K$-contractible} if there exists a positive number $\lambda_0\gq1$ that satisfies the following holds:

\medskip

for any positive numbers $\eps$ and $r$ with $\eps<\frac{1}{4\lambda_0}$, there exists a positive number $r'$ with $r'\gq r$ such that $\iota_*^{\eps,\lambda_0\eps,r,r'}:K_*^{\eps,r}(A)\lto K_*^{\lambda_0\eps,r'}(A)$ is vanishing
(we say that $A$ is $K$-contractible with {\bf rescaling} $\lambda_0$).\end{definition}

\begin{example}\
\begin{enumerate}
\item Recall that a separable $C^*$-algebre $B$ is $K$-contractible if the  class of the identity map $Id_B:B\to B$ vanishes in $KK_*(B,B)$. According to Theorem \ref{thm-tensor}, if $B$  is $K$-contractible, then $A\ts B$ is quantitatively $K$-contractible for any filtered $C^*$-algebra $A$. Moreover, the rescaling does not depend  on $A$ or on  $B$;
\item Let $\Ga$ be a finitely generated group and let $A$ be a $C^*$-algebra provided with an action of $\Ga$ by automorphisms. Assume that
\begin{itemize}
\item the group $\Ga$ satisfies the Baum-Connes conjecture with coefficients.
\item for any finite subgroup $F$ of $\Ga$ the $K_*(A\rtimes F)=0$.
\end{itemize}\end{enumerate}
Then $A\rtr\Ga$ is  quantitatively $K$-contractible and the rescaling does not depend   on $\Ga$ or on $A$.
\end{example}
\begin{remark}\label{remark-universal-rescaling}
It can be proved that there exists a universal rescaling for quantitative $K$-contractibility, i.e there exists a positive number $\la_0$ with $\la_0\gq 1$ such that every quantitatively $K$-contractible $C^*$-algebra is indeed quantitatively $K$-contractible with rescaling $\la_0$.
\end{remark}
 \begin{theorem}\label{thm-quant-K-contractible}
 Let $A$  be a filtered $C^*$-algebra. Assume that  there exists   positive numbers $\lambda_0$  and $c$,  with $\lambda_0\gq 1$  such that for  every positive number $r$ there exists
 an  $r$-controlled  Mayer-Vietoris pair $(\De_1,\De_2,A_{\De_1},A_{\De_2})$ with coercitivity $c$   and with  $A_{\De_1},\,A_{\De_2}$ and $A_{\De_1}\cap A_{\De_2}$ quantitatively $K$-contractible with rescaling $\lambda_0$.
 Then there exists a positive number $\lambda_1$ depending only on $\lambda_0$ and on $c$ such that
 $A$  is quantitatively $K$-contractible with rescaling $\lambda_1$.
 \end{theorem}
\begin{proof}Let $(\la,l)$ be the controlled pair of Corollary  \ref{cor-bound1}  and  set $\lambda_1=\la\lambda_0^2\al_c$, let $\eps$ and $r$ be positive numbers
with $\eps<\frac{1}{4\lambda_1}$ and let  
 $y$ be an element in $K^{\eps,r}_*(A)$. Let $(\De_1,\De_2,A_{\De_1},A_{\De_2})$ be a 
 $r$-controlled  Mayer-Vietoris pair for $A$  with coercitivity $c$. Let $r'$ be a positive number with $r'\gq k_{c,\eps}r$ such that $$\iota^{\al_c\eps,\la_0\al_c\eps,k_{c,\eps}r,r'}(z)=0$$ in
 $K_*^{\la_0\al_c\eps,r'}(A_{\De_1}\cap A_{\De_2})$ for all $z$ in $K_*^{\al_c\eps,k_{c,\eps}r}(A_{\De_1}\cap A_{\De_2})$. Since 
 $$\iota^{\al_c\eps,\la_0\al_c\eps,k_{c,\eps}r,r'}\circ \partial_{\De_1,\De_2,*}^{\eps,r}(y)=0$$ in
  $K_*^{\al_c\eps,k_{c,\eps}r}(A_{\De_1}\cap A_{\De_2})$ and  according to Corollary  \ref{cor-bound1}, then if we set $\la'=\al_c\la_0$,   there exist   an element  $x_1$ in $K_1^{\la\la'\eps  ,l_{\la'\eps }r'}(A_{\De_1})$ and  an element 
 $x_2$ in $K_1^{\la\la'\eps,l_{\la'\eps}r'}(A_{\De_2})$ such that 
$$\iota_*^{\eps, \la\la'\eps,r,l_{\la'\eps}r'}(y)=\jmath^{\la\lambda'\eps,l_{\la'\eps}r'}_{\Delta_1,*}(x_1)-\jmath^{\lambda\la',l_{\la'\eps}r'}_{\Delta_2,*}(x_2).$$ 
 Let $r''$ be a positive number with $r''\gq l_{\la'\eps}r'$ such that for $i=1,2$, 
 $$\iota^{\la\lambda'\eps,\la_0\la\lambda'\eps,l_{\la'\eps}r',r''}(z)=0$$ in
 $K_1^{\la_0\la\la'\eps  ,r''}(A_{\De_i})$ for all $z$ in $K_1^{\la\la'\eps  ,l_{\la'\eps }r'}(A_{\De_i})$.
 Then we eventually obtain that 
 \begin{eqnarray*}
 \iota_*^{\eps,\la_1\eps,r,r''}(y)&=&\jmath^{\la_1\eps,r''}_{\Delta_1,*}\circ\iota^{\la\la'\eps,\la_1\eps,l_{\la'\eps }r',r''}_*(x_1)-\jmath^{\la_1\eps,r''}_{\Delta_2,*}\circ\iota^{\la\la'\eps,\la_1\eps,l_{\la'\eps }r',r''}_*(x_2)\\
 &=&0.
 \end{eqnarray*} Hence $A$  is quantitatively $K$-contractible with rescaling $\lambda_1$.
\end{proof}

\section{Quantitative    K\"unneth formula}\label{sec-kunneth}

In this section, we formulate a quantitative K\"unneth formula for filtered $C^*$-algebras and we  discuss its connection with controlled Mayer-Vietoris pairs. We also show that finitely generated groups  for which  the Baum-Connes conjecture with coefficient holds  provides numerous examples of filtered $C^*$-algebras that satisfy   the quantitative K\"unneth formula
\subsection{Statement on the   formula}
 Recall that if $A$ and $B$ are $C^*$-algebras, then there is a morphism
 $$\omega_{A,B,*}:K_*(A)\ts K_*(B)\lto K_*(A\ts B)$$ given by the external Kasparov product i.e., $\omega_{A,B,*}(x\ts y)=x\ts\tau_A(y)$ for all $x$ in $K_*(A)$ and $y$ in $K_*(B)$. Indeed, in the case of unital $C^*$-algebras, if $p$ and $q$ are respectively projections  in $M_n(A)$ and $M_k(B)$ and if  $u$ and $v$ are respectively unitary elements in $M_n(A)$ and $M_k(B)$, then
 \begin{eqnarray*}
 \omega_{A,B,*} ([p]\ts[q])&=&[p\ts q];\\
 \omega_{A,B,*} ([u]\ts[q])&=&[u\ts q+I_n\ts (I_k-q)];\\
  \omega_{A,B,*} ([p]\ts[v])&=&[p\ts v+(I_n-p)\ts I_k ].\\
 \end{eqnarray*}
 Let $A$ be a $C^*$-algebra  filtered by $(A_r)_{r>0}$ and let $B$ be a $C^*$-algebra (with a trivial filtration). Recall that $A\ts B$ is then filtered by $(A_r\ts B)_{r>0}$. Let us  consider then the  quantitative object $\K_*(A)\ts K_*(B)=(K_*^{\eps,r}(A)\ts K_*(B))_{0<\eps<1/4,r>0}$. With notations of Theorem \ref{thm-tensor}, define the $(\alpha_\T,k_\T)$-control morphism
   $$\Omega_{A,B,*}=(\omega_{A,B}^{\eps,r})_{0<\eps<\frac{1}{4\al_\T},r>0}: \K_*(A)\ts K_*(B)\to \K_*(A\ts B),$$ by
   $$\omega_{A,B,*}^{\eps,r}: K_*^{\eps,r}(A)\ts K_*(B)\to \K_*^{\eps,r}(A\ts B);\,x\ts y\mapsto \tau_A^{\al_\T\eps,h_{\T,\eps}r}(y)(x).$$ Then $\Omega_{A,B,*}$ induces $\omega_{A,B,*}$ in $K$-theory, i.e.,
  \begin{equation}\label{eq-induce-kunneth}\iota_*^{\eps,r}\circ \omega_{A,B,*}^{\eps,r}=\omega_{A,B,*}\circ \left( \iota_*^{\eps,r}\ts \Id_{K_*(B)}\right)\end{equation}  for every positive numbers $r$ and $\eps$ with $0<\eps<\frac{1}{4\al_\T}$.
  
  \begin{remark}\label{remark-kunneth}
  Let $A$  be a   unital filtered $C^*$-algebra  and let $B$ be a unital $C^*$algebra. Let $\eps$ and $r$ be positive numbers with $\eps<\frac{1}{4\al_\T}$.
  \begin{enumerate}
  \item for any $\eps$-$r$-projection $p$ in some $M_n(A)$, any integer $l$ and any projection $q$  in some  $M_k(B)$ then
  $\omega_{A,B,*}^{\eps,r}([p,l]_{\eps,r}\ts [q])=[p\ts q+I_l\ts (I_k-q),lk]_{\al_\T\eps,h_{\T,\eps}r}$ in $K_0^{\al_\T\eps,h_{\T,\eps}r}(A\ts B)$.
  \item for any $\eps$-$r$-unitary  $u$ in some $M_n(A)$  and any projection $q$  in some $M_k(B)$ then
  $\omega_{A,B,*}^{\eps,r}([u]_{\eps,r}\ts [q])=[u\ts q+I_n\ts (I_k-q)]_{\al_\T\eps,h_{\T,\eps}r}$ in $K_1^{\al_\T\eps,h_{\T,\eps}r}(A\ts B)$.
  \end{enumerate}
  \end{remark}
  The quantitative morphism $\Omega_{\bullet,\bullet,*}$ is compatible with the Kasparov  tensorisation (controlled) morphsim.
 \begin{lemma}\label{lemma-compatibility-kasparov}
 There exists a control pair $(\al,k)$ such that the following assertion holds:
 
 \medskip
 
 For any filtered $C^*$-algebra $A$, for any separable $C^*$-algebras $B_1,\,B_2,\,D_1$ and $D_2$, any $z$ in $KK_*(B_1,B_2)$ and any $z'$ in $KK_*(D_1,D_2)$, then the following diagram  is $(\al,k)$-commutative.
  $$\begin{CD}
\K_*(A\ts B_1)\ts K_*(D_1)@>\omega_{A\ts B_1,D_1,*}>> \K_*(A\ts B_1\ts D_1)\\
      @V\T_{A}(z)\ts (\bullet\ts z')VV  @V\T_{A}(\tau_{D_1}(z)\ts\tau_{B_2}(z') )VV \\    
      \K_*(A\ts B_2)\ts K_*(D_2)@>\omega_{A\ts B_2,D_2,*}>>  \K_*(A\ts B_2\ts D_2)   \end{CD},$$ 
  where  $\bullet\ts z':K_*(D_1)\to K_*(D_2)$ is right multiplication by $z'$.
  \end{lemma}
  \begin{proof}Let $y$ be an element of $K_*(D_1)$.
  According to point (vi) of Theorem \ref{thm-tensor} and to Theorem \ref{thm-product-tensor}, there  exists a control pair $(\la,h)$ such
  that 
  $$\T_{A\ts B_2}(y\ts z')\circ \T_A(z)\aeq \T_{A}(z\ts\tau_{B_2}(y)\ts\tau_{B_2}(z')).$$
  Since external Kasparov is commutative, we have
  $$z\ts\tau_{B_2}(y)=\tau_{B_1}(y)\ts \tau_{D_1}(z)$$ Using once again Theorem  \ref{thm-product-tensor} and up to rescaling  the control pair $(\la,h)$, we get that
  $$\T_{A\ts B_2}(y\ts z')\circ \T_A(z)\aeq   \T_A(\tau_{D_1}(z)\ts\tau_{B_2}(z'))\circ \T_{A\ts B_1}(y)$$ and hence the diagram is commutative.
   \end{proof}

 \begin{definition}
Let $A$  be a   filtered $C^*$-algebra and let $\la_0$ be a positive number with $\la_0\gq 1$. We say that $A$ satisfies the quantitative K\"unneth formula with rescaling $\la_0$ if
 $$\Omega_{A ,B,*}: \K_*(A)\ts K_*(B)\to \K_*(A\ts B)$$ is a quantitative isomorphism with rescaling $\la_0$ for every  $C^*$-algebra $B$ such that $K_*(B)$ is free abelian group.
 \end{definition}
 \begin{remark}
 If  a filtered  $C^*$-algebra $A$ satisfies the quantitative  K\"unneth formula, then according to equation (\ref{eq-induce-kunneth})
 $$\omega_{A,B,*}:K_*(A)\ts K_*(B)\lto K_*(A\ts B)$$ is an isomorphism  for every  $C^*$-algebra $B$ such that $K_*(B)$ is free. Using geometric resolutions, it was proved in \cite{sc} that in this case,   the $C^*$-algebra  $A$ satisfies the  K\"unneth formula, i.e., for any $C^*$-algebra $B$ we have a natural short exact sequence
 $$0\lto K_*(A)\ts K_*(B)\lto  K_*(A\ts B)\lto \operatorname{Tor}(K_*(A),K_*(B))\lto 0$$ (see also \cite{ceo2} for the straightforward generalization to the non nuclear case).
 \end{remark}
 
 Next theorem provides many examples of filtered $C^*$-algebras that satisfy the  quantitative  K\"unneth formula and will be proved in Section \ref{subsec-kunneth-BC}.
 \begin{theorem}\label{thm-BC-Qkunneth}
 Let $\Ga$ be a finitely generated group, let $A$ be a $\Ga$-$C^*$-algebra. Assume that
 \begin{itemize}
 \item $\Ga$ satisfies the Baum-Connes conjecture with coefficients.
 \item For each subgroup subgroup $K$ of $\Ga$, then $A\rtimes K$ satisfies the K\"unneth formula.
 \end{itemize}
 Then $A\rtimes_r\Ga$ satisfies the quantitative K\"unneth formula, i.e.,   for any $C^*$-algebra $B$ such that $K_*(B)$ is a free abelian group,  $$\Omega_{A\rtimes_r\Ga ,B,*}: \K_*(A\rtimes_r\Ga)\ts K_*(B)\to \K_*((A\rtimes_r\Ga )\ts B)$$ is a quantitative isomorphism with rescaling that does not depend   on $\Ga$ or on $A$.

 \smallskip Moreover, under above assumption, when   the $C^*$-algebra  $A$ runs  through family of   $\Ga$-$C^*$-algebras and $B$ runs through  $C^*$-algebras  such that $K_*(B)$ is a free abelian group, the family of quantitative isomorphisms  $(\Omega_{A\rtimes_r \Ga ,B})_{A,B}$ is uniform.

 \end{theorem}
The proof of this two results relies indeed on the quantitative statements of Theorem \ref{thm-quant-surj} which hold for groups that satisfy the Baum-Connes conjecture with coefficients. Similarly, using the geometric quantitative statements of Section \ref{subsection-quantitative-assembly-map} and Theorem \ref{thm-geo-statement}, we can prove the following result:
 \begin{theorem}
 Let  $\Si$ be a discrete proper metric space with bounded geometry that  coarsely embeds into a Hilbert space. If $A$ satisfies the K\"unneth formula, then $A\ts \Kp(\ell^2(\Si))$ 
  satisfies the quantitative K\"unneth formula  with rescaling that does not depend  on $\Si$ or on  $A$.
\end{theorem}

\begin{remark}
As for quantitative $K$-contractibility (see Remark \ref{remark-universal-rescaling}), it can be proved that there exists a universal rescaling for the quantitative  K\"unneth formula.\end{remark}
\subsection{Quantitative  K\"unneth formula and controlled  Mayer-Vietoris pairs}
In this subsection, we state  a permanence result  (that we shall prove in next subsection) for the  quantitative  K\"unneth formula with respect to controlled Mayer-Vietoris pairs which satisfy  a nuclear type condition.
\begin{definition}\label{definition-nuclear-coarse-MV-pair}
Let  $A$ be a $C^*$-algebra filtered by $(A_s)_{s>0}$   and let $r$  be a  positive number. An $r$-controlled nuclear  Mayer-Vietoris pair is a quadruple
 $(\De_1,\De_2,A_{\Delta_1},A_{\Delta_2})$, where $\De_1$ and $\De_2$ are closed linear subspaces of $A_r$ stable under involution and  $A_{\Delta_1}$ and $A_{\Delta_2}$ are respectively  $r$-controlled   $\Delta_1$ and $\De_2$-neighborhood-$C^*$-algebras 
 such that for some positive number $c$ and for any $C^*$-algebra $B$
\begin{enumerate}
\item  $(\De_1\ts B,\De_2\ts B)$ is a  coercive  decomposition pair for $A\ts B$ of order $r$ with coercitivity $c$.
\item the pair $ (A_{ \Delta_1, s}\ts B, A_{\Delta_2, s}\ts B)$ has the CIA property  with coercitivity $c$ for any positive number $s$.
\end{enumerate}
The positive number $c$ is called the {\bf coercity} of the $r$-controlled nuclear Mayer-Vietoris pair $(\De_1,\De_2,A_{\Delta_1},A_{\Delta_2})$.
\end{definition}

\begin{remark}\
\begin{enumerate}
\item Notice that $A_{\Delta_1}\ts B$ and $A_{\Delta_2}\ts B$ are respectively  $r$-controlled   $\Delta_1\ts B$ and $\De_2\ts B$-neighborhood-$C^*$-algebras.
\item Condition (ii) amounts to the following: for any positive numbers $\eps$ and $s$, any $x$ in $A_{\Delta_1,s}\ts B$ and any $y$ in $A_{\Delta_2,s}\ts B$ such that  $||x-y||<\eps $,
then there exists $z$ in  $(A_{\Delta_1,s}\cap A_{\Delta_2,s}) \ts B$ satisfying
$$||z-x||<c \eps, ~~~~||z-y||<c\eps.$$\end{enumerate}
\end{remark}
\begin{example}
Replacing $M_n(\C)$ with $B$, we see that   Example \ref{example-roe} and examples of Section \ref{subsection-groupoid} are indeed $r$-controlled nuclear Mayer-Vietoris pairs (with the same coercitivity).
\end{example}
Next lemma shows that the controlled boundary maps of a $r$-controlled nuclear Mayer-Vietoris pair are indeed compatible with Kasparov external product.
\begin{lemma}\label{lemma-mv-kunneth}
For any positive number $c$, there exists a control pair $(\al,h)$ such that the following is satisfied:
let $(\De_1,\De_2,A_{\Delta_1},A_{\Delta_2})$ be an $r$-controlled nuclear  Mayer-Vietoris pair with coercitivty $c$, let $B$ and $B'$
be two $C^*$-algebras and $z$ be an element in $KK_*(B,B')$, then
$$\TT_{(A_{\Delta_1}\cap A_{\Delta_2})\ts B,*}(z)\circ \DD_{\De_1\ts B,\De_2\ts B}\stackrel{(\al,h)}{\sim}  \DD_{\De_1\ts B',\De_2\ts B'}\circ \TT_{A,*}(z).$$
\end{lemma}
\begin{proof}
We first deal with the case $z$ even. According to
\cite[Lemma 1.6.9]{laff-inv}, there exists a $C^*$-algebra $D$
and homomorphisms $\theta:D \to B$ and $\eta:D \to B'$ such that
\begin{itemize}
\item the element $[\theta]$ of $KK_*(D,B)$ induced by $\theta$ is
   invertible.
\item $z=\eta_*([\theta]^{-1})$.
\end{itemize}
Let $\theta_A:A\ts B'\to A\ts D$ and $\theta_{A_{\Delta_1}\cap A_{\Delta_2}}:(A_{\Delta_1}\cap A_{\Delta_2})\ts B'\to (A_{\Delta_1}\cap A_{\Delta_2})\ts D$ be the homorphisms induced by $\theta$ on tensor products and define similarly $\eta_A:A\ts B'\to A\ts D$ and $\eta_{A_{\Delta_1}\cap A_{\Delta_2}}:(A_{\Delta_1}\cap A_{\Delta_2})\ts B'\to (A_{\Delta_1}\cap A_{\Delta_2})\ts D$ the homomorphisms induced by   $\eta$.

By naturality of  quantitative boundary morphism of $r$-controled Mayer-Vietoris pairs (see equation (\ref{equ-boundary-natural}) of Section \ref{subs-section-boundary}), we get that
$$\theta_{A_{\Delta_1}\cap A_{\Delta_2},*}\circ \DD_{\De_1\ts B',\De_2\ts B'}= \DD_{\De_1\ts D,\De_2\ts D}
\circ \theta_{A,*}.$$
According to Theorem \ref{thm-tensor}, the control morphisms  $\theta_{A,*}$ and $\theta_{A_{\Delta_1}\cap A_{\Delta_2},*}$ are
$(\al_\TT,k_\TT)$-invertible with $(\al_\TT,k_\TT)$-inverses respectively given by $\TT_{A,*}([\theta]^{-1})$ and
$\TT_{A_{\Delta_1}\cap A_{\Delta_2},*}([\theta]^{-1})$ and hence, there exists a control pair $(\al,k)$ depending only on  $c$ 
and $(\al_\TT,k_\TT)$ such that
$$\DD_{\De_1\ts B',\De_2\ts B'}\circ \TT_{A,*}([\theta]^{-1}) \stackrel{(\al,k)}{\sim} \TT_{A_{\Delta_1}\cap A_{\Delta_2},*}([\theta]^{-1})\circ \DD_{\De_1\ts D,\De_2\ts D}.$$
Then, using the bifunctoriality of $\TT_{A,*}$ (see Theorem  \ref{thm-tensor}), we obtain
\begin{eqnarray*}
\DD_{\De_1\ts B',\De_2\ts B'}\circ \TT_{A,*}(z)&=&\DD_{\De_1\ts B',\De_2\ts B'}\circ \TT_{A,*}([\theta]^{-1})\circ 
\eta_{A,*}\\
& \stackrel{(\al,k)}{\sim}&  \TT_{A_{\Delta_1}\cap A_{\Delta_2},*}([\theta]^{-1})\circ \DD_{\De_1\ts D,\De_2\ts D}\circ 
\eta_{A,*}\\
& \stackrel{(\al,k)}{\sim}&  \TT_{A_{\Delta_1}\cap A_{\Delta_2},*}([\theta]^{-1})\circ \eta_{A_{\Delta_1}\cap A_{\Delta_2},*}\circ \DD_{\De_1\ts B,\De_2\ts B}\\
& \stackrel{(\al,k)}{\sim}&  \TT_{A_{\Delta_1}\cap A_{\Delta_2},*}(z)\circ \DD_{\De_1\ts B,\De_2\ts B},
\end{eqnarray*}
where the third line is once again the consequence of  naturality of  quantitative boundary morphism of $r$-controled Mayer-Vietoris pairs (see equation (\ref{equ-boundary-natural}) of Section \ref{subs-section-boundary}).
If $z$ is an odd  element, recall that $[\partial]$ is the invertible element of $KK_1(\C,C(0,1))$ implementing the boundary 
morphism of the extension $$0\lto C(0,1) \lto C(0,1]\stackrel{ev_0}{\lto} \C\lto 0.$$ Then there exists a control pair $(\al,k)$ depending only 
on $c$ and on the control pairs of Theorems \ref{thm-tensor}  and \ref{thm-product-tensor} such that 
\begin{equation*}
\begin{split}
\TT_{A_{\Delta_1}\cap A_{\Delta_2},*}(z)&\circ \DD^0_{\De_1\ts B,\De_2\ts B}\\& \stackrel{(\al,k)}{\sim}\TT_{A_{\Delta_1}\cap A_{\Delta_2},*}(z)\circ \TT_{(A_{\Delta_1}\cap A_{\Delta_2})\ts B,*}([\partial]^{-1})\circ \DD^1_{S\De_1\ts B,S\De_2\ts B,*}\circ \TT_{A\ts B,*}( [\partial])\\
& \stackrel{(\al,k)}{\sim}\TT_{A_{\Delta_1}\cap A_{\Delta_2},*}(z)\circ \TT_{(A_{\Delta_1}\cap A_{\Delta_2}),*}(\tau_B([\partial]^{-1}))\circ \DD^1_{S\De_1\ts B,S\De_2\ts B,*}\circ \TT_{A\ts B,*}( [\partial])\\
& \stackrel{(\al,k)}{\sim}\TT_{A_{\Delta_1}\cap A_{\Delta_2},*}(\tau_B([\partial]^{-1})\ts_Bz)\circ  \DD^1_{S\De_1\ts B,S\De_2\ts B,*}\circ \TT_{A\ts B,*}( [\partial])\\
& \stackrel{(\al,k)}{\sim} \DD^1_{S\De_1\ts B',S\De_2\ts B',*}\circ \TT_{A,*}(\tau_B([\partial]^{-1})\ts_Bz)\circ \TT_{A,*}(\tau_{B,*}([\partial]))\\
& \stackrel{(\al,k)}{\sim} \DD^1_{S\De_1\ts B',S\De_2\ts B',*} \circ \TT_{A,*}(z),\end{split}
\end{equation*}
where
\begin{itemize}
\item the  first $\stackrel{(\al,k)}{\sim}$  holds by definition of  $\DD^0_{\De_1\ts B,\De_2\ts B}$;
\item the  $\stackrel{(\al,k)}{\sim}$ is a consequence of point (vi) of Theorem \ref{thm-tensor};
\item the third and fifth   $\stackrel{(\al,k)}{\sim}$ are consequences of Theorem \ref{thm-product-tensor};
\item the fourth  $\stackrel{(\al,k)}{\sim}$ is  a consequence of the even case;
\end{itemize}
Similarly, we can prove that there exists 
a control pair $(\al,k)$ depending only 
on $c$ and on the control pair of Theorems \ref{thm-tensor}  and \ref{thm-product-tensor} such that 
$$\TT_{A_{\Delta_1}\cap A_{\Delta_2},*}(z)\circ \DD^1_{\De_1\ts B,\De_2\ts B}  \stackrel{(\al,k)}{\sim}\DD^0_{S\De_1\ts B,S\De_2\ts B,*} \circ \TT_{A,*}(z)$$ and hence we obtain the result in the odd case.
\end{proof}

 \begin{theorem}\label{thm-mv-kunneth}
 Let $A$  be a filtered $C^*$-algebra. Assume there exists    positive numbers $c$  and $\lambda_0$ with $\la_0\gq 1$ that satisfies the following:   for  any positive number $r$, there exists
 an  $r$-controlled nuclear Mayer-Vietoris pair $(\De_1,\De_2,A_{\De_1},A_{\De_2})$ with coercitivity $c$ such that
 $A_{\De_1},\,A_{\De_2}$ and $A_{\De_1}\cap A_{\De_2}$ satisfies the 
  quantitative K\"unneth formula with rescaling $\la_0$. 
  
  \medskip
  
  Then $A$  satisfies the 
  quantitative K\"unneth formula with rescaling $\la_1$ for a positive number $\la_1$  with $\la_1\gq 1$ depending only on $c$ and $\la_0$.  \end{theorem}
The proof of this theorem will be given in next subsection.

\subsection{Proof of  Theorem \ref{thm-mv-kunneth}\label{subsubsection-preliminary}}
The idea of the proof of Theorem  \ref{thm-mv-kunneth} is to use the controlled exactness persistence properties stated in Corollary \ref{cor-bound1} and  Lemma \ref{lemma-technic2} to carry out  a controlled five lemma argument. This requires
 some preliminary work in order to formulate the quantitative K\"unneth formula in terms of  an (even degree) controlled morphism between quantitative $K$-theory groups.

  \subsubsection{Preliminaries}

We  denote by $SB=C_0((0,1),B)$ the suspension algebra of a $C^*$-algebra $B$. Let 
$A$ be a filtered $C^*$-algebra, let $B$ be a $C^*$-algebra and let $[\partial]$  be the invertible element  in $KK_1(\C,C(0,1))$ that implements the boundary of the Bott  extension,
$0\to C_0(0,1)\to C_0[0,1)\stackrel{ev_0}{\to}\C\to 0$. Applying Lemma \ref{lemma-compatibility-kasparov} to
$z=Id_\C$ and $z'=\tau_B([\partial])$, we see that we only have to consider the odd case
$$\Omega_{A,B,*}:\K_1(A)\ts K_0(B)\oplus \K_0(A)\ts K_1(B)\lto \K_1(A\ts B).$$

Let $\To_2=\{(z_1,z_2)\text{ such that }|z_1|=|z_2|=1\}$ be the two  torus. Let us view $(0,1)^2$ as an open subset of
$\To^2$ via the inclusion map $(0,1)^2\hookrightarrow \To^2;\, (s,t)\mapsto (e^{2\ip s},e^{2\ip t})$.
\begin{lemma}\label{lemma-eight} Possibly rescaling the control pair $(\al_\D,k_\D)$, then 
for any  filtered $C^*$-algebra $A$,  the filtered and semi-split extension
of filtered $C^*$-algebras \begin{equation} \label{equation-extension-eight} 0\lto S^2 A\stackrel{j_A}{\lto} C(\To_2,A)\stackrel{q_A}{\lto} C(\To_2\setminus (0,1)^2,A)\lto 0\end{equation}  has a vanishing controlled boundary map   $$\D_{S^2A,C(\To_2,A)}:\K_*(C(\To_2\setminus (0,1)^2,A))\lto   \K_{*+1}(S^2 A).$$ 
 \end{lemma}
\begin{proof}
By using  controlled Bott periodicity \cite[Lemma 4.6]{oy2} and in view of \cite[Proposition 3.19]{oy2}, we only need to check that the result holds for $$\D_{S^2A,C(\To_2,A)}:\K_1(C(\To_2\setminus (0,1)^2,A))\lto   \K_{0}(S^2 A).$$
But \begin{equation}\label{equation-eight}C(\To_2\setminus (0,1)^2,A)\cong \{(f_1,f_2)\in C(\To,A)\oplus \C(\To,A)\text{ such that }f_1(1)=f_2(1)\}.\end{equation} Let $(f_1,f_2)$ be an $\eps$-$r$-unitary  in $C(\To_2\setminus (0,1)^2,A)$ and define then $$u_{f_1,f_2}:\To^2\lto A;\, (z_1,z_2)\mapsto f_1(z_1)f^*_1(1)f_2(z_2)$$Then $u_{f_1,f_2}$ is a $9\eps$-$3r$-unitary in $C(\To_2,A)$. Moreover, under the identification of equation (\ref{equation-eight}), we have  $$q_A(u_{f_1,f_2})=(f_1f^*_1(1)f_1(1),f_1(1)f^*_1(1)f_2)$$ and hence $$\|q_A(u_{f_1,f_2})-(f_1,f_2)\|<3\eps.$$ Then the result holds by construction of the controlled boundary map in the odd case.
\end{proof}
In consequence,
by using the controlled six term exact sequence  (Theorem \ref{thm-six term}) associated to
the semi-split and filtered extension of equation (\ref{equation-extension-eight}), then if we set 
 $$ \ker q_{A,*}\defi (\ker q^{\eps,r}_{A,*}: K^{\eps,r}_*(C(\To_2, A))\lto  K^{\eps,r}_*(C(\To_2\setminus (0,1)^2, A))  )_{0<\eps<1/4,r>0},$$
we obtain the following corollary.
\begin{corollary}\label{corollary-eight}
There exists a controlled pair $(\la,h)$ such that for any filtered $C^*$-algebra $A$, then
$$j_{A,*}:\K_*(S^2 A)\lto \ker q_{A,*}$$ is a $(\la,h)$-controlled isomorphism. \end{corollary}
Notice that, by construction of the controlled boundary map associated to controlled Mayer-Vietoris pair 
  $(\De_1,\De_2,A_{\De_1},A_{\De_2})$,  we have that 
 
\begin{equation}\label{equation-boundary-eight} 
q_{A_{\De_1}\cap A_{\De_2}\ts B,*}
\circ \D_{C(\To^2,A_{\De_1}),C(\To^2,A_{\De_2}),*}=
\D_{C_0(\To^2\setminus ]0,1[^2,A_{\De_1}),C_0(\To^2\setminus ]0,1[^2,A_{\De_2}),*}\circ q_{A\ts B,*}.
\end{equation}
Let us set with notation of Corollary \ref{corollary-eight}
$$\O'_*(A,B)=\O'_0(A,B)\oplus \O'_1(A,B)\defi   \ker q_{A\ts B,*}.$$
Set $z=[\partial]\ts\tau_{C_0(0,1)}([\partial])$ and define then the $(\al^2_\T,h_\T*h_\T)$-controlled morphism
$$ j_{A\ts B,*}\circ \T_{A\ts B}(z)\circ \Omega_{A,B,*}:\K_*(A)\ts K_0(B)\oplus \K_{*+1}(A)\ts K_1(B)\lto\O'_*(A,B).$$
Since  $z$ is an invertible element in $KK_0(\C,\C_0((0,1)^2))$ and hence, according to Theorem \ref{thm-product-tensor} and  to Corollary  \ref{corollary-eight}, 
\begin{itemize}
\item if $\Omega_{A,B,*}$  is a quantitative isomorphism with rescaling $\la_0$ then there exists a positive number $\la_1$ depending only on $\la_0$ with $\la_1\gq 1$ such that $j_{A\ts B,*}\circ \T_{A\ts B}(z)\circ \Omega_{A,B,*}$ is a quantitative isomorphism with rescaling $\la_1$;
\item if   $j_{A\ts B,*}\circ \T_{A\ts B}(z)\circ \Omega_{A,B,*}$  is a quantitative isomorphism with rescaling $\la_1$,  then there exists a positive number $\la_0$ depending only on $\la_1$ with $\la_0\gq 1$ such that $\Omega_{A,B,*}$ is a quantitative isomorphism with rescaling $\la_0$.
\end{itemize}
 According to Lemma \ref{lemma-compatibility-kasparov}, we get that
  \begin{equation}\label{equation-bott-kunneth}
\T_A(z)\circ \Omega_{A,B,*} \aeq \Omega_{SA,SB,*}\circ ((\T_A([\partial])\ts \tau_B([\partial])\end{equation}for some control pair $(\la,k)$ that   does not depend  on $A$ or on  $B$. Let us set 
$$\O_*(A,B)\defi \K_*(A)\ts K_0(B)\oplus  \K_*(SA)\ts K_0(SB)$$  and denote for positive numbers $\eps,\,\eps',r$ and $r'$
with $\eps\lq\eps'<1/4$ and $r\lq r'$ respectively by $\iota^{\eps,\eps',r,r'}_{\O,*}:O^{\eps,r}_*(A,B)\lto O^{\eps',r'}_*(A,B)$ and
$\iota^{\eps,\eps',r,r'}_{\O',*}:{O_*'}^{\eps,r}(A,B)\lto {O_*'}^{\eps',r'}(A,B)$ the structure morphisms corresponding to $\O_*(A,B)$ and ${\O'_*}(A,B)$.  For $(\al_\F,k_\F)=(\la,h)$, let us define the $(\al_\F,k_\F)$-controlled morphism $$\F_{A,B,*}=(F_{A,B,*}^{\eps,r})_{0<\eps<\frac{1}{4\al_\F},r>0}:\O_*(A,B)\lto \O_*'(A,B)$$ by 
\begin{equation*}\begin{split}F_{A,B,*}^{\eps,r}(x\oplus x')=j_{A\ts B,*}^{\al_\F\eps,k_{\F,\eps}r}\circ\iota_{\O',*}^{\al^2_\T\eps,\al_\F\eps,k_{\T,\eps}k_{\T,\al_\T\eps}r,k_{\F,\eps}r}\circ  \tau^{\al_\T\eps,k_{\T,\eps}r}_A(z)&\circ \omega^{\eps,r}_{A,B,*}(x)\\
+\iota_{\O',*}^{\al_\T\eps,\al_\F\eps,k_{\T,\eps}r,k_{\F,\eps}r}&\circ j_{A\ts B,*}^{\al_\T\eps,k_{\T,\eps}r}\circ   \omega^{\eps,r}_{A,B,*}(x').\end{split}\end{equation*}Notice that $\F_{\bullet,B,*}$ is obviously natural with respect to filtered morphisms.
 Since $[\partial]$ is invertible in $KK_1(\C,C(0,1))$ and according to equation (\ref{equation-bott-kunneth}) and 
Theorem \ref{thm-product-tensor},  Theorem \ref{thm-mv-kunneth} is equivalent to the following statement:
\medskip

 Let $A$  be a filtered $C^*$-algebra. Assume there exists    positive numbers $c$  and $\lambda_0$ with $\la_0\gq 1$ that satisfies the following:   for  any positive number $r$, there exists
 an  $r$-controlled nuclear Mayer-Vietoris pair $(\De_1,\De_2,A_{\De_1},A_{\De_2})$ with coercitivity $c$ such that
for any $C^*$-algebra $B$ with  $K_*(B)$  free abelian, then $\F_{A_{\De_1},B,*},\,\F_{A_{\De_2},B,*}$ and $\F_{A_{\De_1}\cap A_{\De_2},B,*}$ are quantitative isomorphisms  with rescaling $\la_0$.  Then for any 
 for any $C^*$-algebra $B$ with  $K_*(B)$  free abelian,
 $\F_{A,B,*}$  is a quantitative isomorphism with coercity $\la_1$ for a positive number $\la_1$ with $\la_1\gq 1$ depending only on $\la_0$ and $c$.
 
\smallskip 

Notice that,  by controlled Bott periodicity, we only need to consider the odd case.
\subsubsection{Notations}   Let us introduce some notations that we will use throughout  the proof.
  Let $v$ be a  unitary in a unital $C^*$-algebra $B$.
  Define the unitary of  $M_2(B)$
  $$R_v:[0,1]\lto U_2(B);\,t\mapsto \left(\begin{smallmatrix}v&0\\ 0&1\end{smallmatrix}\right)\cdot \left(\begin{smallmatrix}\cos \frac{t\pi}{2}&\sin\frac{t\pi}{2}\\ -\sin \frac{t\pi}{2}&\cos\frac{t\pi}{2}\end{smallmatrix}\right)\cdot \left(\begin{smallmatrix}v^*&0\\ 0&1\end{smallmatrix}\right)\cdot \left(\begin{smallmatrix}\cos\frac{t\pi}{2}&-\sin \frac{t\pi}{2}\\ \sin\frac{t\pi}{2}&\cos \frac{t\pi}{2}\end{smallmatrix}\right),$$ and the projection  of  $M_2(B)$
  $$P_v:[0,1]\lto M_2(B) ;\, t\mapsto  R_v(t)\cdot \left(\begin{smallmatrix}1&0\\ 0 &0\end{smallmatrix}\right) R^*_{v}(t).$$
  Define also 
  $$P_{\text{Bott}}:[0,1]\times [0,1]\to M_2(\C);\, (s,t) \mapsto R_{e^{2\imath \pi s}}(t)\cdot \left(\begin{smallmatrix}1&0\\ 0 &0\end{smallmatrix}\right) R^*_{e^{2\imath \pi s}}(t).$$
  Then,  if we view the $2$-sphere $\S^2$  as the one point compactification of $(0,1)^2$, then $P_{\text{Bott}}$ is a rank one projection
  in $M_2(C(\S^2))$ and $[P_{\text{Bott}}]-[1]$ is a generator for $K_0(C_0((0,1)))\cong\Z$.
  
  \medskip

   For a   $r$-controlled nuclear Mayer-Vietoris pair
 $(\De_1,\De_2,A_{\De_1},A_{\De_2})$  with respect to a filtered $C^*$-algebra $A$, let 
  \begin{eqnarray*} \jmath_{\De_1,\O,*}&=&(\jmath^{\eps,r}_{\De_1,\O,*})_{0<\eps<1/4,r>0}:\O_*(A_{\De_1},B)\lto \O_*(A,B),\\ 
   \jmath_{\De_2,\O,*}&=&(\jmath^{\eps,r}_{\De_2,\O,*})_{0<\eps<1/4,r>0}:\O_*(A_{\De_2},B)\lto \O_*(A,B),\\
  \jmath_{\De_1,\De_2,\O,*}&=&(\jmath^{\eps,r}_{\De_1,\De_2,\O,*})_{0<\eps<1/4,r>0}:\O_*(A_{\De_1}\cap A_{\De_2},B)\lto \O_*(A_{\De_1},B)\end{eqnarray*} and  $$\jmath_{\De_2,\De_1,\O,*}=(\jmath^{\eps,r}_{\De_2,\De_1,\O,*})_{0<\eps<1/4,r>0}:\O_*(A_{\De_1}\cap A_{\De_2},B)\lto \O_*(A_{\De_2},B)$$ be the morphisms respectively induced by
the inclusions of $C^*$-algebras
${\jmath_{\De_1}: A_{\De_1}\hookrightarrow A}$, ${\jmath_{\De_2} :A_{\De_2}\hookrightarrow A}$, ${\jmath_{\De_1,\De_2}: A_{\De_1}\cap A_{\De_2}\hookrightarrow A_{\De_1}}$ and
$\jmath_{\De_2,\De_1}: A_{\De_1}\cap A_{\De_2}\hookrightarrow A_{\De_2}$.
In the same way, we define
\begin{eqnarray*} \jmath_{\De_1,\O',*}&=&(\jmath^{\eps,r}_{\De_1,\O',*})_{0<\eps<1/4,r>0}:\O'_*(A_{\De_1},B)\lto \O'_*(A,B),\\ 
   \jmath_{\De_2,\O',*}&=&(\jmath^{\eps,r}_{\De_2,\O',*})_{0<\eps<1/4,r>0}:\O'_*(A_{\De_2},B)\lto \O'_*(A,B),\\
  \jmath_{\De_1,\De_2,\O',*}&=&(\jmath^{\eps,r}_{\De_1,\De_2,\O',*})_{0<\eps<1/4,r>0}:\O'_*(A_{\De_1}\cap A_{\De_2},B)\lto \O'_*(A_{\De_1},B)\end{eqnarray*} and  $$\jmath_{\De_2,\De_1,\O',*}=(\jmath^{\eps,r}_{\De_2,\De_1,\O',*})_{0<\eps<1/4,r>0}:\O'_*(A_{\De_1}\cap A_{\De_2},B)\lto \O'_*(A_{\De_2},B)$$

  \subsubsection{Computation for    $F^{\eps,r}_{A,B,*}$}\label{subsubsec-computation-F} Let us now compute $F^{\eps,r}_{A,B}(x)$  for $A$, a unital filtered  $C^*$-algebra, $B$ a unital $C^*$-algebra and  $x$ an element  in $O_*^{\eps,r}(A,B)$.
  
  \smallskip
  
  Let us consider first the case $x=[u]_{\eps,r}\ts [p]$ where 
 $u$ is an $\eps$-$r$-unitary in some $M_n(A)$ and $p$ is  a projection in some $M_m(B)$. Let us set
  $v_{u,p}=u\ts p+I_n\ts (I_m-p)$. Then $v_{u,p}$ is an $\eps$-$r$-unitary in $M_{nm}(A\ts B)$.
   According to 
 Remark \ref{remark-kunneth}, then
 $$F^{\eps,r}_{A,B,*}([u]_{\eps,r}\ts [p])=j_{A,*}^{\al_\T\eps,h_{k_\T,\eps}r}\circ \tau^{\eps,r}_A(z)(v_{u,p}).$$
  It is well known that
  $z=[P_{\text{Bott}}]-[P_1]$ (with $P_1=\diag(1,0)$). Let us define  for $u$ and $u'$ some $\eps$-$r$-unitaries in $M_n(A)$ and  $p$ a projection in $M_m(B)$ 
  
  \begin{eqnarray*}
  W_{u,u',p}:\To_2&\lto &M_n(A)\ts M_m(B)\ts M_4(\C)\\
   (e^{2\ip s},e^{2\ip t})&\mapsto& \diag(v_{u,p\ts P_{\text{Bott}}(s,t)},v_{u',p\ts P_{1}}).\end{eqnarray*}Then
 $W_{u,u',p}$ is an   $\eps$-$r$-unitary in $C(\To_2,M_n(A)\ts M_m(B)\ts M_4(\C))$. Moreover, if $u'=u^*$ then we see that 
 $q_{A,B,*}[W_{u,u^*,p}]_{\eps,r}=0$ in $K^{\eps,r}_1(C(\To_2\setminus (0,1)^2,A\ts B))$ and hence 
 \begin{equation}\label{equ-F}
 F^{\eps,r}_{A,B,*}([u]_{\eps,r}\ts [p])=[W_{u,u^*,p}]_{\al_\F\eps,h_{\F,\eps}r}.
 \end{equation}
 
 Consider now the case  $x=[u]_{\eps,r}\ts ([p]-[p(0)])$, where $u$ is an $\eps$-$r$-unitary  in some $M_n(\widetilde{SA})$ and $p$ is a projection in some $M_m(\widetilde{SB})$. 
 
 
 For $\eps$-$r$-unitaries $u$ and $u'$ in $M_n(\widetilde{SA})$  and a  projection $p$ in $M_n(\widetilde{SA})$,
 set \begin{eqnarray*}W'_{u,u',p}:\To_2&\lto &M_n(A)\ts M_m(B)\ts M_2(\C)\\ (e^{2\ip s},e^{2\ip t})&\mapsto &\diag(v_{u(s),p(t)},v_{u'(s),p(1)}).\end{eqnarray*}
 Then $W'_{u,u',p}$ is an $\eps$-$r$-unitary in $C(\To_2,M_n(A)\ts M_m(B)\ts M_2(\C))$  and we can easily check that 
$[W'_{u,u^*,p}]_{\eps,r}$  is in $${\O'_1}^{\eps,r}(A,B)=\ker  q^{\eps,r}_{A,*}: K^{\eps,r}_1(C(\To_2, A))\lto  K^{\eps,r}_1(C(\To_2\setminus (0,1)^2, A))$$ for all positive numbers $\eps$ and $r$  with $\eps\in(0,1/4)$.
 Moreover, according to 
 Remark \ref{remark-kunneth}, we have
  \begin{equation}\label{equ-F-suspension}
 F^{\eps,r}_{A,B,*}([u]_{\eps,r}\ts ([p]-[I_k]))=[W'_{u,u^*,p}]_{\al_\F\eps,h_{\F,\eps}r}.
 \end{equation} 
 
 \smallskip
 Consider now the case  $x=[q,m]_{\eps,r}\ts [p]$, where $q$ is an $\eps$-$r$-projection in some $M_n(A)$, $m$ is an integer with $m\lq n$  and $p$ is a projection in some $M_k(B)$.
  Let us set
when $q$ is an $\eps$-$r$-projection in some $M_n(A)$,  $m$ is a positive number and $p_1$ and $p_2$ are projections in some $M_k(B)$  
$$E_{q,m,p_1,p_2}=\diag(q\ts p_1+P_m\ts (I_k-p_1),P_{n-m}\ts (I_k-p_2)+(I_n-q)\ts p_2)$$ with $P_m=\diag(I_m,0)$ in $M_n(\C)\subseteq M_n(A)$.
Then $E_{q,m,p_1,p_2}$ is an $\eps$-$r$-projection in $M_{2kn}(A\ts B)$.
According to Lemma \ref{lemma-compatibility-kasparov}, we have that $$F_{A,B,*}(x)=\Omega_{A,B,*}([q,m]_{\eps,r}\ts ([p]\ts z)).$$
Since $[p]\ts z=[p\ts P_{Bott}]-[p\ts P_{1}]$ in $K_0(S^2B)$, we see according to Remark
 \ref{remark-kunneth}   that
\begin{equation}\label{equ-F-even}F^{\eps,r}_{A,B,*}([q,m]_{\eps,r}\ts [p])=[E_{q,m,p\ts P_{Bott},p\ts P_1},2nk]_{\eps,r}.\end{equation}
 
 \smallskip
 Consider finally  the case of $x=[q',m']_{\eps,r}\ts ([p'_1]-[p'_2])$, where $q'$ is an $\eps$-$r$-projection in some $M_n(\widetilde{SA})$, $m'$  is an  integer and $p'_1$ and $p'_2$ are  projections in some $M_k(\widetilde{SB})$,
 with $p'_1(0)=p'_2(0)$. Then 
  \begin{equation}\label{equ-F-even-suspension}F^{\eps,r}_{A,B,*}([q',m']_{\eps,r}\ts ([p'_1]-[p'_2]))=[E_{q',m',p'_1,p'_2},nk]_{\eps,r}.\end{equation}

\subsubsection{$QI$-condition}          Let $(\De_1,\De_2,A_{\De_1},A_{\De_2})$  be 
 an  $r$-controlled nuclear Mayer-Vietoris pair  with coercitivity $c$ such that
for any $C^*$-algebra $B$ with  $K_*(B)$  free abelian, then $\F_{A_{\De_1},B,*},\,\F_{A_{\De_2},B,*}$ and $\F_{A_{\De_1}\cap A_{\De_2},B,*}$ are quantitative isomorphisms  with rescaling $\la_0$.
  Let us check that there exists a positive number $\la_1$ depending only on $\la_0$ and $c$, with $\la_1\gq 1$ such that for any positive numbers $\eps$ and $s$ with $\eps$ in $(0,\frac{1}{4\la_1})$ and $s\lq \frac{r}{\al_{\F,\eps}}$, then $F^{\eps,s}_{A,B,*}$ satisfies the $QI$-condition of Definition \ref{def-quantitative-morphism}. We can assume without loss of generality that $A$ and $B$ are unital.
 Moreover,  up to replacing  $B$ by $B\oplus \C$, we can assume that there exists a system of 
  generators of $K_0(B)$ given by classes of projections. 
  Let us fix such a system of generator for $K_0(B)$ and fix also a system of generator for $K_1(B)$.  
As discussed  in Section \ref{subsubsection-preliminary}.1, we  only need to consider the odd case, i.e show that the control morphism
 $\F_{A,B,*}:\O_1(A,B)\lto \O'_1(A,B)$ sastifies  the required conditions.

According to Lemma \ref{lemma-mv-kunneth}, there exists a control pair $(\al,h)$ depending only on $c$ with $(\al_c,k_c)\lq (\al,h)$ such that  \begin{equation}\label{equation-compatiblility-F-boundary}\F_{A_{\Delta_1}\cap A_{\Delta_2},B,*}\circ\left(\D_{\De_1,\De_2,*}\ts \Id_{K_0(B)} \oplus\D_{S\De_1,S\De_2,*}\ts \Id_{K_0(SB)}\right)\stackrel{(\al,h)}{\sim}\D_{\De_1\ts B,\De_2\ts B,*}\circ \F_{A,B,*}\end{equation} at order $r$.
For $\eps$ in $(0,\frac{1}{4\al\la_0})$, let $r_{\eps}^{\F}$ be a positive numbers with $r_{\eps}^{\F} \gq h_{\eps}r$ such that
 $$\iota^{\al\eps,\la_0\al\eps,h_{\eps}r,r_{\eps}^{\F}}_{\O,*} (x)=0$$ for all $x$ in
 $K_*^{\al\eps,h_{\eps}r}(A_{\De_1}\cap A_{\De_2})\ts K_*(B)$ such that  $F^{\al\eps,h_{\eps}r}_{A_{\Delta_1}\cap A_{\Delta_2},B,*}(x)=0$.
\begin{proposition}\label{proposition-technic}There exists a control pair  $(\la,h)$ depending only on $\la_0$ and $c$ such that    for any positive numbers $\eps,\,s$ and $r'$ with $\eps<\frac{1}{4\la},\,s\lq \frac{r}{k_{\F,\eps}}$  and $r'\gq r_{\eps}^{\F}$, for any  $x$ in $O^{\eps,s}(A,B)$ such that  $F^{\eps,s}_{A,B}(x)=0$ in ${O'}^{\al_\F,k_{\F,\eps}s}(A,B)$, there exist
\begin{itemize}
 \item an element $x^{(1)}$ in $O_1^{\la\eps,h_\eps r'}(A_{\De_1},B)$ and an element  $x^{(2)}$ in $O_1^{\la\eps,h_\eps r'}(A_{\De_2},B)$;
\item an integer $n$ and an  $(\eps,s)$-unitary $W_x$ in $M_n(C(\To_2,A\ts B))$; \item for $i=1,2$, a  $(\la\eps,h_{\eps}r')$-unitary $W^{(i)}_x$ in $M_n(C(\To_{2},A\ts B))$ with 
$W^{(i)}_x-I_n$ in $M_n(C(\To^{2},A_{\De_i}\ts B))$;\end{itemize}
such that
\begin{enumerate}
\item $\iota_{\O,1}^{\eps,\la\eps,s,h_\eps r'}(x)=\jmath_{\De_1,\O,*}^{\la\eps,h_\eps r'}(x^{(1)})+\jmath_{\De_2,\O,*}^{\la\eps,h_\eps r'}(x^{(2)})$;
\item  $[W_x]_{\al_\F\eps,k_{\F,\eps}r}=0$    in ${O'_1}^{\al_\F\eps,k_{\F,\eps}r}(A,B)$;
\item  $[W^{(i)}_x]_{\la\al_\F\eps,k_{\F,\la\eps}h_\eps r'}$ is in ${O'_1}^{\la\al_\F\eps,k_{\F,\la\eps}h_\eps r'}(A_{\De_i},B)$ and $F_{A_{\De_i},B,*}^{\la\eps,h_\eps r'}(x^{(i)})=[W^{(i)}_x]_{\la\al_\F\eps,k_{\F,\la\eps}h_\eps r'}$ for $i=1,2$; 
\item $\|W_x-W^{(1)}_x W^{(2)}_x\|<\la\eps$;
\end{enumerate}
\end{proposition}

  \begin{proof} The proof of the proposition is quite long so we split it in several steps.
  
  \smallskip
  
  {\bf Step 1:}
  Let $x$ be an element in $O_1^{\eps,s}(A,B)$.  Then there exists  integers $l$ and $l'$ and \begin{itemize}
  \item for $i=1,\ldots l$
  \begin{itemize}
  \item an $\eps$-$s$-unitary $u_i$ in some $M_{n_i}(A)$;
  \item a projection $p_i$ in some $M_{m_i}(B)$ with $[p_i]$ in the prescribed system of generators for $K_0(B)$.
  \end{itemize}
  \item  for $i=1,\ldots l'$
  \begin{itemize}
  \item an $\eps$-$s$-unitary $u'_i$ in some $M_{n'_i}(\widetilde{SA})$;
  \item a projection   $p'_i$ in some $M_{m'_i}(\widetilde{SB})$ with  $[p'_i]-[{p'_i(0)}]$ in the prescribed system of generators for $K_1(SB)$;
 \end{itemize} \end{itemize}
 
 such that
  $$x=\sum_{i=1}^l[u_i]_{\eps,s}\ts [p_i]+\sum_{i=1}^{l'}[u'_i]_{\eps,s}\ts ([p'_i]-[{p'_i(0)}]).$$
  
 Assume that
  $F^{\eps,s}_{A,B,*}(x)=0$ in ${O'_1}^{\al_{\F}\eps,k_{\F,\eps}s}(A,B)$.
  Using Morita equivalence, and up to replacing $A$ by $M_{n}(A)$ and $B$ by $M_{m}(B)$ with $n=n_1+\ldots n_l+n'_1+\ldots n'_{l'}    $ and
   $m=m_1+\ldots m_l+m'_1+\ldots m'_{l'}$, we can assume that $n_i=m_i=1$ for $i=1,\ldots,l$ and  $n'_i=m'_i=1$ for $i=1,\ldots,l'$. Using Lemma \ref{lemma-almost-canonical-form-odd}, we can moreover assume without loss of generality
that $u'_i(0)=u'_i(1)=I_{k_i}$ for $i=1,\ldots,l'$. Let us set 
\begin{equation}\label{equ-kun}W_x=\diag(W_{u_1,u_1^*,p_1},\ldots, W_{u_l,u_l^*,p_l},W'_{u'_1,{u'_1}^*,p'_1},\ldots W'_{u'_{l'},{u'_{l'}}^*,p'_{l'}} ).\end{equation}    Then $W_x$ is an $\eps$-$s$-unitary in  $M_{2(2l+l')}(C(\To_2,A\ts B))$ and
$[W_x]_{\eps,s}$ is in ${O'_1}^{\eps,s}(A,B)$.   In view of equations (\ref{equ-F}) and (\ref{equ-F-suspension}), we see that 
 $[W_x]_{\al_{\F}\eps,k_{\F,\eps}s}=F^{\eps,s}_{A,B,*}(x)=0$  in ${O_1'}^{\al_{\F}\eps,k_{\F,\eps}s}(A,B)\subseteq K_1^{\al_{\F}\eps,k_{\F,\eps}s}(C(\To_2,A\ts B))$. 
 
  \smallskip

 {\bf Step 2:} Let $(\al,h)$ be the control pair of equation (\ref{equation-compatiblility-F-boundary}), then we have
 \begin{equation}\label{equ-diag} F^{\al\eps,h_{\eps}s}_{A_{\Delta_1}\cap A_{\Delta_2},B,*}\circ\iota_{\O,*}^{\al_c\eps,\al \eps,k_{c,\eps}s, h_{\eps}s  }\circ\left(\partial^{\eps,s}_{\De_1,\De_2,*}\ts \Id_{K_0(B)})\oplus\partial^{\eps,s}_{S\De_1,S\De_2,*}\ts \Id_{K_0(SB)}\right)
 (x)=0.\end{equation}
 
 Since $[p_1],\ldots,[p_l]$ belong  to  the prescribed system of generator for $K_0(B)$ and 
 $[p'_1]-[p'_1(0)],\dots,[p'_{l'}]-[p'_{l'}(0)]$  belong  to  the prescribed system of generator for $K_1(B)$, we deduce from
 equation (\ref{equ-diag}) and from the choice of $r'$ that 
  $$\iota^{\al_c\eps,\la_0\al\eps,k_{c,\eps}s,r'}_* \circ \partial^{\eps,s}_{\De_1,\De_2,*}[u_i]_{\eps,s}=0$$for $i=1,\ldots,l$ and
   $$\iota^{\al_c\eps,\la_0\al\eps,k_{c,\eps}s,r'}_* \circ \partial^{\eps,s}_{S\De_1,S\De_2,*}[u'_i]_{\eps,s}=0$$for $i=1,\ldots,l'$.
According to Lemma \ref{lemma-bound1} and in view of the definition of $\partial^{\eps,s}_{\De_1,\De_2,*}$, then for a control pair $(\la,k)$ depending only on $c$, there exists for $i=1,\ldots,n$ and up to replacing $u_i$ by some $\diag(u_i,I_{n_i-1})$ for some integer $n_i$ two 
$\lambda\lambda_0\al \eps$-$k_{\la_0\al\eps} r'$-unitaries $v^{(1)}_i$ and $v^{(2)}_i$ respectively in  $M_{n_i}(\widetilde{A_{\De_1}})$ and $M_{n_i}(\widetilde{A_{\De_2}})$, with $v^{(1)}_i-I_{n_i}$ and  $v^{(2)}_i-I_{n_i}$ respectively in  $M_{n_i}(A_{\De_1})$ and $M_{n_i}({A_{\De_2}})$ and such that
 $$\|u_i-v^{(1)}_iv^{(2)}_i\|<\la\la_0\al_c \eps.$$ 
 Since we also have  $$\iota^{\al_c\eps,\la_0\al\eps,k_{c,\eps}s,r'}_* \circ \partial^{\eps,r}_{\De_1,\De_2}[u^*_i]_{\eps,s}=0,$$ according to Proposition \ref{proposition-decom-unitaries-adjoint} and up to rescaling $(\la,k)$, we can assume that  there exists
  two 
$\lambda\lambda_0\al \eps$-$k_{\la_0\al\eps} r'$-unitaries ${v'_i}^{(1)}$ and ${v'_i}^{(2)}$ respectively in  $M_{n_i}(\widetilde{A_{\De_1}})$ and $M_{n_i}(\widetilde{A_{\De_2}})$ with ${v'_i}^{(1)}-I_{n_i}$ and  ${v'_i}^{(2)}-I_{n_i}$ respectively in  $M_{n_i}(A_{\De_1})$ and $M_{n_i}({A_{\De_2}})$ and such that
 $$\|u^*_i-{v'_i}^{(1)}{v'_i}^{(2)}\|<\la\la_0\al \eps$$ and 
${v'_i}^{(j)}$ is homotopic to ${v^{(j)*}_i}$ as an $\lambda\lambda_0\al \eps$-$k_{\la_0\al\eps} r'$-unitary in  $M_{n_i}(\widetilde{A_{\De_i}})$ for $j=1,2$.

 In the same way, there exists 
  two 
$\lambda\lambda_0\al \eps$-$k_{\la_0\al\eps} r'$-unitaries ${w_i}^{(1)}$ and ${w'_i}^{(1)}$  in  $M_{n'_i}(\widetilde{SA_{\De_1}})$ and        two 
$\lambda\lambda_0\al \eps$-$k_{\la_0\al\eps} r'$-unitaries ${w_i}^{(2)}$ and ${w'_i}^{(2)}$           in            $M_{n'_i}(\widetilde{SA_{\De_2}})$ for some integer $n'_i$ such that

$$\|u'_i -{w_i}^{(1)}{w_i}^{(2)}\|<\la\la_0\al \eps,$$

 $$\|{u'_i}^* -{w'_i}^{(1)}{w'_i}^{(2)}\|<\la\la_0\al \eps,$$ 
and ${w'_i}^{(j)}$ is homotopic to ${w^{(j)*}_i}$
as an $\lambda\lambda_0\al \eps$-$k_{\la_0\al\eps} r'$-unitary in  $M_{n'_i}(\widetilde{SA_{\De_i}})$ for $j=1,2$ and
${w_i}^{(1)}(0)={w_i}^{(2)}(0)= {w'_i}^{(1)}(0)={w'_i}^{(2)}(0)=I_{{n'_i}}$      for $i=1,\ldots,l'$.
In particular, 
$W_{v^{(j)}_i,{v_i}^{(j)*}}$ and $W_{v^{(j)}_i,{v'_i}^{(j)}}$ for $j=1,2$ and $i=1,\ldots,l$ are homotopic     
 $\lambda\lambda_0\al \eps$-$k_{\la_0\al\eps} r'$-unitaries  in  $M_{4n_i}(\widetilde{A_{\De_j}})$     
 and  $W'_{w^{(j)}_i,{w_i}^{(j)*}}$ and $W_{w^{(j)}_i,{w'_i}^{(j)}}$   for $j=1,2$ and $i=1,\ldots,l'$ are homotopic     
 $\lambda\lambda_0\al \eps$-$k_{\la_0\al\eps} r'$-unitaries  in  $M_{2n'_i}(\widetilde{A_{\De_j}})$.

%
 \smallskip

 {\bf Step 3:} 
 According to Lemmas  \ref{lemma-almost-closed}  and \ref{cor-example-homotopy}  and up to replacing  $\la$ by $12\la$, we have that
 \begin{equation}\label{equ-mv-w}
 \begin{split}
 \iota^{\eps,\lambda\lambda_0\al \eps,s,2k_{\la_0\al\eps} r'}_{\O,*}(x)= &\jmath^{\lambda\lambda_0\al \eps,2k_{\la_0\al\eps} r'}_{\De_1,*}\ts Id_{K_0(B)}(\sum_{i=1}^l[v_i^{(1)}]_{\lambda\lambda_0\al \eps,2k_{\la_0\al\eps} r'}\ts [p_i])\\
 &+\jmath^{\lambda\lambda_0\al \eps,2k_{\la_0\al\eps} r'}_{\De_2,*}\ts Id_{K_0(B)}(\sum_{i=1}^l[v_i^{(2)}]_{\lambda\lambda_0\al \eps,2k_{\la_0\al\eps} r'}\ts [p_i])\\   
 &+\jmath^{\lambda\lambda_0\al\eps,2k_{\la_0\al\eps} r'}_{S\De_1,*}\ts Id_{K_0(SB)}(\sum_{i=1}^{l'}[w_i^{(1)}]_{\lambda\lambda_0\al \eps,2k_{\la_0\al\eps} r'}\ts ([p'_i]-[p'_i(0)]))\\
 &+\jmath^{\lambda\lambda_0\al \eps,2k_{\la_0\al\eps} r'}_{S\De_2,*}\ts Id_{K_0(SB)}(\sum_{i=1}^{l'}[w_i^{(2)}]_{\lambda\lambda_0\al \eps,2k_{\la_0\al\eps} r'}\ts ([p'_i]-[p'_i(0)]))
 \end{split} \end{equation}

Let us set for $j=1,2$
 \begin{equation*}
  \begin{split} x^{(j)}
 = \jmath^{\lambda\lambda_0\al \eps,2k_{\la_0\al\eps} r'}_{\De_1,*}\ts & Id_{K_0(B)}\left(\sum_{i=1}^l[v_i^{(j)}]_{\lambda\lambda_0\al \eps,2k_{\la_0\al\eps} r'}\ts [p_i]\right)\\
 &+\jmath^{\lambda\lambda_0\al\eps,2k_{\la_0\al\eps} r'}_{S\De_1,*}\ts Id_{K_0(SB)}\left(\sum_{i=1}^{l'}[w_i^{(j)}]_{\lambda\lambda_0\al \eps,2k_{\la_0\al\eps} r'}\ts ([p'_i]-[p'_i(0)]\right)
 \end{split} \end{equation*}
 and
 
  \begin{equation*}W^{(j)}_x=\diag(W_{v^{(j)}_1,{v'_1}^{(j)},p_1},\ldots, W_{v^{(j)}_l,{v'_l}^{(j)},p_l},W'_{w^{(j)}_1,{{w'_1}^{(j)}},p'_1},\ldots W'_{w^{(j)}_j,{{w'_{l'}}^{(j)}},p'_{l'}} ).\end{equation*} Then $x^{(1)},\,x^{(2)},\,W_x,\, W^{(1)}_x$ and $W^{(2)}_x$ satisfy the required condition for some suitable control pair. \end{proof}
%
{\it End of the proof of the $QI$-statement.} We divide the end of the proof of the $Q_I$-statement in the four  following  steps.

\bigskip

{\bf Step 1: } Let $x$ be an element in $O^{\eps,s}_1(A,B)$ such that $F_{A,B,*}^{\eps,s}(x)=0$. Let $r'$ be a positive number such that $r'\gq r^{\F}_{\eps}$. With notations of Proposition \ref{proposition-technic},
applying  Proposition \ref{proposition-decom-unitaries} to $W_x$,  up to rescaling $(\la,h)$ and  to replacing 
$W^{(1)}_x$ and $W^{(2)}_x$ respectively by  $\diag(W^{(1)}_x,I_j)$ and $\diag(W^{(2)}_x,I_j)$ for some integer $j$, there exist for any positive number $\eps$ in $(0,\frac{1}{4\la})$
   two $\lambda\eps$-$h_\eps r'$-unitaries  $W'_1$ and $W'_2$ in some $M_n(C(\To_2,A\ts B))$ such that
\begin{itemize}
\item $W'_i-I_n$ is an  element  in the matrix algebra   $M_n(C(\To_2,A_{\De_i}\ts B))$ for $i=1,2$;
\item for $i=1,2$, there exists  a homotopy $({W'_{i,t}})_{t\in[0,1]}$ of  $\la\eps$-$h_\eps r'$-unitaries between  $I_n$ and $W'_i$ such that
$W'_{i,t}-I_k\in M_n(C(\To_2,A_{\De_i}\ts B))$ for all $t$ in $[0,1]$.
\item $\|W^{(1)}_xW^{(2)}_x-W'_1W'_2\|<\lambda\eps$.
\end{itemize}
Up to replacing   $\la$ by $5\la$, we can assume indeed  that $\|{W'_1}^*W^{(1)}_x-W'_2{W^{(2)*}_x}\|<\lambda\eps$.
If we apply the $CIA$ property to   ${W'_1}^*W^{(1)}_x-I_n$  and $W'_2W_x^{(2)*}-I_n$, we get that there exists
$V'$ in $M_n(C(\To_2,A\ts B))_{h_\eps r'}$ such that
\begin{itemize}
\item $\|{W'_1}^*W^{(1)}_x-V'\|<c\lambda\eps$;
\item $\|W'_2W_x^{(2)*}-V'\|<c\lambda\eps$;
\item  $V'-I_n$ lies in $M_n(C(\To_2,(A_{\De_1}\cap A_{\De_2})\ts B))$;
\end{itemize}
In particular, in view of Lemma \ref{lemma-almost-closed},	 $V'$ is a $4(c+3)\lambda\eps$-$2h_\eps r'$-unitary in  
$M_n(\widetilde{C}(\To_2,(A_{\De_1}\cap A_{\De_2})\ts B)))$ homotopic to ${W'_1}^*W^{(1)}_x$ (resp. to $W'_2W_x^{(2)*}$) as a 
$4(c+3)\lambda\eps$-$2h_\eps r'$-unitary in $M_n(\widetilde{C}(\To_2,A_{\De_1}\ts B))$ (resp. in $M_n(\widetilde{C}(\To_2,A_{\De_2}\ts B))$). 

\bigskip

{\bf Step 2: }
Let us define \begin{eqnarray*}
V_x:\To_2:&\lto& M_{3n}(A\ts B)\\
(z_1,z_2)&\mapsto &\diag({V'}^*(1,1)V'(z_1,z_2),{V'}^*(z_1,1){V'}(1,1),{V'}^*(1,z_2){V'}(1,1)).
\end{eqnarray*}
If we set $\la'=12(c+3)\lambda$, then
\begin{itemize}
\item  $V_x$ is a   $\lambda'\eps$-$4h_\eps r'$-unitary in  
$M_n({C}(\To_2,A\ts B))$;
\item  $V_x-I_{3n}$  is in $M_{3n}(C(\To_2,(A_{\De_1}\cap A_{\De_2})\ts B))$;
\item $[V_x]_{\la'\eps, 4h_\eps r'}$  lies in ${O'_1}^{\la'\eps, 4h_\eps r'}(A_{\De_1}\cap A_{\De_2},B)$;
\end{itemize}

 Then we have
 \begin{equation}\label{equ-Vx1}\jmath^{\la'\eps,4h_\eps r'}_{\De_1,\De_2,\O',*}([V_x]_{\la'\eps, 4h_\eps r'})=[{W_1'}^*W^{(1)}_x]_{\la'\eps, 4h_\eps r'}=[W^{(1)}_x]_{\la'\eps, 4h_\eps r'},\end{equation} where the the second equality holds because 
 $W'_1$ is connected to $I_n$ as a $\la'\eps$-$4h_\eps r'$-unitary  of
$M_n(\widetilde{SA_{\De_1}})$.
In the same way, we have 
 \begin{equation}\label{equ-Vx2}\jmath^{\la'\eps,4h_\eps r}_{\De_2,\De_1,\O',*}([V_x]_{\la'\eps, 4h_\eps r'})=-[W^{(2)}_x]_{\la'\eps, 4h_\eps r'}.\end{equation}
 
 \bigskip

{\bf Step 3: }
Let $r''$ be a positive integer with 
$k_{\F,\la_0 \la'\eps}          r''\gq 4k_\eps r'$  such that for any $z$ in  ${O'_*}^{\la'\eps,4k_\eps r'}(A_{\De_1}\cap A_{\De_2},B)$, there exists $y$ in $O_*^{\la_0\la'\eps,r''}(A_{\De_1}\cap A_{\De_2},B)$ such that 
$$\iota_{\O',*}^{\la'\eps,\la_0\la'\al_\F\eps,4k_\eps r',k_{\F,\la_0 \la'\eps} r''}(z)= F^{\la_0\la'\eps,r''}_{A_{\De_1}\cap A_{\De_2},B,*}(y).$$       Then there exists
$x'$ in  $O_1^{\la_0\la',r''}(A_{\De_1}\cap A_{\De_2},B)$  such that
$$[V_x]_{\la_0\la'\al_\F\eps,k_{\F,\la_0 \la'\eps} r''}= F^{\la_0\la'\eps,r''}_{A_{\De_1}\cap A_{\De_2},B,*}(x').$$ 

Hence we have 
\begin{eqnarray}\label{equ1-end-QI}
\nonumber F^{\la_0\la'\eps,r''}_{A_{\De_1},B,*}\circ  \jmath^{\la_0\la'\eps, r''}_{\De_1,\De_2,\O,*}(x')&=&\jmath^{\la_0\la'\al_\F\eps, k_{\F,\la_0\la'\eps}r''}_{\De_1,\De_2,\O',*}([V_x]_{\la_0\la'\al_\F\eps,k_{\F,\la_0 \la'\eps} r''})\\
 \nonumber&=& [W^{(1)}_x]_{\la_0\la'\al_\F\eps,k_{\F,\la_0 \la'\eps} r''}\\
&=&F^{\la_0\la'\eps,r''}_{A_{\De_1},B,*}\circ \iota_{\O,*}^{\la_0\la'\eps,r''}(x^{(1)})\end{eqnarray}
where the first equality holds by naturality of   $\F_{\bullet,B,*}$, the second  equality holds in view of equation (\ref{equ-Vx1})  and the last equality is a consequence of Proposition  \ref{proposition-technic}.
In the same way, using equation (\ref{equ-Vx2}), we get that
\begin{equation}\label{equ2-end-QI}F^{\la_0\la'\eps,r''}_{A_{\De_2},B,*}\circ  \jmath^{\la_0\la'\eps, r''}_{\De_2,\De_1,\O,*}(x')= -F^{\la_0\la'\eps,r''}_{A_{\De_2},B,*}\circ \iota_{\O,*}^{\la_0\la'\eps,r''}(x^{(2)}).\end{equation}

\bigskip

{\bf Step 4:}
Let $R$ be a positive number, with $R \gq    r'' $ such that
 $$\iota^{\la_0\la'\eps,\la_0^2\la''\eps,r'',R}_{\O,*}(y)=0$$ for all $y$ in
 $O_*^{\la_0\la'\eps,r''}(A_{\De_j},B)$ such that  $F^{\la_0\la'\eps,r''}_{A_{\Delta_j},B,*}(y)=0$ and $j=1,2$.
 In particular, from equations (\ref{equ1-end-QI})  and  (\ref{equ2-end-QI}),   we deduce
 $$\iota^{\la_0\la'\eps,\la_0^2\la''\eps,r'',R}_{\O,*}\circ\jmath^{\la_0\la''\eps, r''}_{\De_1,\De_2,\O,*}(x') 
 =\iota^{\la\eps,\la_0^2\la'\eps,h_\eps r',R}_{\O,*}(x^{(1)})$$ and 
$$\iota^{\la_0\la'\eps,\la_0^2\la''\eps,r'',R}_{\O,*}\circ\jmath^{\la_0\la''\eps, r''}_{\De_2,\De_1,\O,*}(x') 
 =-\iota^{\la\eps,\la_0^2\la'\eps,h_\eps r',R}_{\O,*}(x^{(2)}).$$ 
 Since $\iota_{\O,*}^{\eps,\la\eps,s,h_\eps r'}(x)=x^{(1)}+x^{(2)}$, this establishes the $QI$-statement.
 
 \subsubsection{$QS$-condition}  Let $(\De_1,\De_2,A_{\De_1},A_{\De_2})$  be 
 an  $r$-controlled nuclear Mayer-Vietoris pair  with coercitivity $c$ for $A$ such that
for any $C^*$-algebra $B$ with  $K_*(B)$  free abelian, then $\F_{A_{\De_1},B,*},\,\F_{A_{\De_2},B,*}$ and $\F_{A_{\De_1}\cap A_{\De_2},B,*}$ are quantitative isomorphisms  with rescaling $\la_0$.
  Let us check that  $\F_{A,B,*}$ satisfies the  $QS$-condition of definition \ref{def-quantitative-morphism}. As for the $QI$-condition,  we can assume without loss of generality that $A$ and $B$ are unital  and that 
there exists a system of 
  generators of $K_0(B)$ given by classes of projections. 
  Let us fix such a system of generator for $K_0(B)$ and fix also a system of generator for $K_0(SB)$.  
We  also only need to consider the odd case, i.e show that the control morphism
 $\F_{A,B,*}:\O_1(A,B)\lto \O'_1(A,B)$ sastifies  the required conditions. 
 
 \smallskip
 
 Let $\eps$ be a positive number with $\eps$ in $(0,\frac{1}{4\la_0\al_c})$ and let $r^{\F,(1)}_\eps$ be a positive number with
 $k_{c,\eps}r\lq k_{\F,\la_0\al_c\eps} r^{\F,(1)}_\eps$ such that  
 for any $y$ in in ${O'_*}^{\al_c\eps,k_{c,\eps}r}(A_{\De_1}\cap A_{\De_2},B)$, there exists an element $x$ in
 ${O_*}^{\la_0\al_c\eps,r^{\F_\eps,(1)}}(A_{\De_1}\cap A_{\De_2},B)$ such that
 $$\iota_{\O',*}^{\al_c\eps,\la_0\al_c\al_\F\eps,k_{c,\eps}r,k_{\F,\la_0\al_c}r^{\F_\eps,(1)}}(y)=F^{\la_0\al_c\eps,r^{\F_\eps,(1)}}_{A_{\De_1}\cap A_{\De_2},B,*}(x).$$ 
 
 \smallskip
 
  Let $r^{\F,(2)}_\eps$ be a positive number with $r^{\F,(1)}_\eps\lq r^{\F,(2)}_\eps$ such that 
  $$\iota^{\la_0\al_c\eps,\la^2_0\al_c\eps,r_\eps^{\F,(1)},r_\eps^{\F,(2)}}_{\O,*} (x)=0$$ for all $x$ in
 $K_*^{\la_0\al_c\eps,r_\eps^{\F,(1)}}(A_{\De_1}\cap A_{\De_2})\ts K_*(B)$ such that  $F^{\la_0\al_c\eps,r_\eps^{\F,(1)}}_{A_{\Delta_1}\cap A_{\Delta_2},B,*}(x)=0$. 
 
\begin{proposition}\label{proposition-technic-QS} There exists a control $(\la,h)$ depending only on $c$ and $\la_0$ with
$(\al_c,k_c)\lq (\la,h)$ such that for any positive numbers  $\eps$ and  $r'$ with $\eps$ in $(0,\frac{1}{4\la})$ and 
$r^{\F,(2)}_\eps\lq r'$   the following is satisfied:

\medskip

 For any element  $y$ 
 in $O'^{\eps,r}_1(A,B)\subseteq  K_1^{\eps,r}(C(\To^2,A\ts B))$, there exists
 \begin{itemize}
 \item an element $z_y$ in $\O_1^{\la\eps,h_\eps r'}(A,B)$;
 \item a positive integer $n_y$ and two  $\la'\eps$-$h'_\eps r'$-unitaries $W_{y}$ and $W'_{y}$ in  $M_{n_y}(C(\To_2,A\ts B))$;
 \item for $j=1,2$ a $\la\eps$-$h_\eps r'$-unitary    $W^{(j)}_{y}$     in 
$M_{2n_y}(C(\To_2,A\ts B))$ with  $W^{(j)}_{y}-I_{2n_y}$  in $M_{2n_y}(C(\To_2,A_{\De_j}\ts B))$;
\item a $\la\eps$-$h_\eps r'$-projection $q_y$ in  $M_{2n_y}(C(\To_2,(A_{\De_1}\cap A_{\De_2})\ts B))$
\end{itemize}
such that 
\begin{itemize}
\item $[W_{y}]_{\la\al_\F\eps,k_{\F,\la\eps}h_\eps r'}$ is in ${O'_1}^{\la\al_\F\eps,k_{\F,\la\eps}r'}(A,B)$ and $F^{\la\eps,h_\eps r'}_{A,B,*}(z_y)=[W_{y}]_{\la\al_\F\eps,k_{\F,\la\eps}h_\eps r'}$;
\item $[{q}_y,n_y]_{\la\eps,h_\eps r'}=\iota^{\al_c\eps,\la\eps,k_{c,\eps}r,h'_\eps r'}_{\O',*}\circ\partial^{\eps,r}_{C(\To_2,\De_1\ts B),C(\To_2,\De_2\ts B)}(y)$.
\item $\|\diag(W_{y},W'_{y})-W^{(1)}_{y}W^{(2)}_{y}\|<\la\eps$;
\item $\|W^{(1)*}_{y}\diag(I_{n_y},0)W^{(1)}_{y}-{q}_y\|<\la\eps$
 \end{itemize}
 \end{proposition}
 \begin{proof}
 The proof of this proposition is also quite long so  is divided into four steps.
 
 \bigskip
 
 {\bf Step 1:} 
 According to equation (\ref{equation-boundary-eight}), we see that
 $y'=\partial^{\eps,r}_{C(\To_2,A_{\De_1}),C(\To_2,A_{\De_2}),*}(y)$ belongs to ${O'_0}^{\al_c\eps,k_{c,\eps}r}( A_{\De_1}\cap A_{\De_2},B)$.  
Set  $R=r^{\F,(1)}_\eps$. Then there exists an element $x$  in ${O_0}^{\la_0\al_c\eps,R}(A_{\De_1}\cap A_{\De_2},B)$ such that
 $$\iota_{\O',*}^{\al_c\eps,\la_0\al_c\al_\F\eps,k_{c,\eps}r,k_{\F,\la_0\al_c\eps}R}(y')=F^{\la_0\al_c\eps,R}_{A_{\De_1}\cap A_{\De_2},B,*}(x).$$

There exist two integers  $l$ and $l'$  and \begin{itemize}
  \item for $i=1,\ldots l$
  \begin{itemize}
  \item an $\la_0\al_c\eps$-$R$-projection $q_i$ in some $M_{n_i}(\widetilde{A_{\De_1}\cap A_{\De_2}} )$ and an integer $k_i$.
  \item a projection $p_i$ in some $M_{m_i}(B)$ with $[p_i]$ in the prescribed system of generators for $K_0(B)$;
  \end{itemize}
  \item  for $i=1,\ldots l'$
  \begin{itemize}
  \item an $\la_0\al_c\eps$-$R$-projection  $q'_i$ in some $M_{n'_i}(\widetilde{SA_{\De_1}\cap A_{\De_2} })$ and an integer $k'_i$;
  \item a projection   $p'_i$ in some $M_{m'_i}(\widetilde{SB})$ such that  $[p'_i]-[p'_i(0)]$ is in the prescribed system of generators for $K_1(SB)$;
 \end{itemize} \end{itemize}
 
 such that
  $$x=\sum_{i=1}^l[q_i,k_i]_{\la_0\al_c\eps,R}\ts [p_i]+\sum_{i=1}^{l'}[q'_i,k'_i]_{\la_0\al_c\eps,R}\ts ([p'_i]-[p'_i(0)]).$$
  By Morita equivalence, up to replacing   $B$ by $M_{m}(B)$ with  
   $m=m_1+\ldots m_l+m'_1+\ldots m'_{l'}$, we can assume that $m_i=1$ for $i=1,\ldots,l$ and  $m'_i=1$ for $i=1,\ldots,l'$. Using Lemma \ref{lemma-almost-canonical-form}, we can moreover assume without loss of generality
that $q_i-\diag(I_{k_i},0)$ is in $M_{n_i}({A_{\De_1}\cap A_{\De_2}} )$ for $i=1,\ldots,l$ and that $q'_i(0)=q'_i(1)=\diag(I_{k'_i},0)$ for $i=1,\ldots,l'$. 
Set $$q'_y=\diag(E_{q_1,n_1,p_1\ts P_{Bott},p_1\ts P_1},\ldots E_{q_l,n_l,p_l\ts P_{Bott},p_l\ts P_1},
E_{q'_1,n'_1,p'_1,p'_1(0)},\ldots E_{q'_l,n'_l,p'_l,p'_l(0)}).$$
In view of equations (\ref{equ-F-even}) and (\ref{equ-F-even-suspension}), 
 $$\iota_{\O',*}^{\al_c\eps,\la_0\al_c\al_\F\eps,k_{c,\eps}r,k_{\F,\la_0\al_c}R}(y')=[q'_y,n_y]_{\la_0\al_c\al_\F\eps,k_{\F,\la_0\al_c}R},$$ with $n_y=2(n_1+\ldots+n_l)+n'_1+\ldots+n'_l$.

 \bigskip
 {\bf   Step 2:}
By  naturality of $\F_{\bullet,B,*}$, we obtain
 $$F^{\la_0\al_c\eps, R}_{A_{\De_1},B,*}\circ \jmath^{\la_0\al_c\eps,R}_{\De_1,\De_2,\O,*}(x)=0$$ in 
 ${O'_0}^{\la_0\al_c\al_\F\eps,k_{\F,\la_0\al_c}R}(A_{\De_1},B)$  and 
   $$F^{\la_0\al_c\eps,r'}_{A_{\De_2},B,*}\circ \jmath^{\la_0\al_c\eps,R}_{\De_2,\De_1,\O,*}(x)=0$$ in 
 ${O'_0}^{\la_0\al_c\al_\F\eps,k_{\F,\la_0\al_c}R}(A_{\De_2},B)$. 
 
 Let $r'$ be  a positive number with $r^{\F,(2)}_\eps\lq r'$.   Since $[p_i]$ for $i=1,\ldots,l$ and   $[p'_i]-[p'_i(0)]$ for $i=1,\ldots,l'$ are respectively in the prescribed system of generator of $K_0(B)$ and $K_0(SB)$ and by definition of   $r^{\F,(2)}_\eps$,   we deduce that 
 \begin{itemize}
 \item $\jmath^{\la^2_0\al_c\eps,r'}_{\De_1,\De_2,*}([q_i,k_i]_{\la^2_0\al_c\eps,r'})=0$ in $K_0^{\la^2_0\al_c\eps,r'}(A_{\De_1})$ for $i=1,\dots,l$;
 \item  $\jmath^{\la^2_0\al_c\eps,r'}_{S\De_1,S\De_2,*}([q'_i,k'_i]_{\la^2_0\al_c\eps,r'})=0$ in $K_0^{\la^2_0\al_c\eps,r'}(SA_{\De_1})$ for $i=1,\dots,l'$;
 \item $\jmath^{\la^2_0\al_c\eps,r'}_{\De_2,\De_1,*}([q_i,k_i]_{\la^2_0\al_c\eps,r'})=0$ in $K_0^{\la^2_0\al_c\eps,r'}(A_{\De_2})$ for $i=1,\dots,l$;
  \item $\jmath^{\la^2_0\al_c\eps,r'}_{S\De_2,S\De_1,*}([q'_i,k'_i]_{\la^2_0\al_c\eps,r'})=0$ in $K_0^{\la^2_0\al_c\eps,r'}(SA_{\De_2})$ for $i=1,\dots,l'$.
  \end{itemize}
  Let $(\la,h)$ be the control pair of Lemma \ref{lemma-technic2}. Then for $i=1,\ldots l$ and up to stabilisation, we can assume that $n_i=2k_i$ and that 
  there exists
  $v^{(1)}_i$ and $v^{(2)}_i$ two $\la\la^2_0\al_c\eps$-$h_{\la^2_0\al_c\eps}r'$ unitaries in $M_{2n_i}(A)$ and  $u_i$ and $u'_i$  two $\la\la^2_0\al_c\eps$-$h_{\la^2_0\al_c\eps}r'$ unitaries in $M_{n_i}(A)$ such that
  
  \begin{itemize}
\item $v^{(j)}_i-I_{n_i}$ is an element  in   $M_{n_i}(A_{\De_j})$ for $j=1,2$;
\item $$\|v^{(1)*}_i\diag(I_{k_i},0)v^{(1)}_i-q_i\|<\la\la^2_0\al_c\eps$$ and $$\|v^{(2)}_i\diag(I_{k_i},0)v^{(2)*}_i-q_i\|<\la\la^2_0\al_c\eps.$$
 \item for $j=1,2$, then $v^{(j)}_i$ is connected to $I_{n_i}$ by a homotopy
of $\la\la^2_0\al_c\eps$-$h_{\la^2_0\al_c\eps}r''$-unitaries  $(v^{(j)}_{i,t})_{t\in[0,1]}$ in $M_{n_i}(A)$ such  that $v^{(j)}_{i,t}-I_{2n_i}$ is in $M_{n_i}(A_{\De_j})$ for all $t$ in $[0,1]$. 
\item $\|\diag(u_i,u'_i)-v^{(1)}_iv^{(2)}_i\|<\la\la^2_0\al_c\eps$.
\end{itemize}

In the same way, for $i=1,\ldots l'$ and up to stabilisation, we can assume that $n'_i=2k'_i$ and that 
  there exists
  $w^{(1)}_i$ and $w^{(2)}_i$ two $\la\la^2_0\al_c\eps$-$h_{\la^2_0\al_c\eps}r'$ unitaries in $M_{n'_i}(\widetilde{SA})$ and  $u''_i$ and $u'''_i$  two $\la\la^2_0\al_c\eps$-$h_{\la_0\al_c\eps}r'$ unitaries in $M_{n'_i}(\widetilde{SA})$ such that
  
  \begin{itemize}
\item $w^{(j)}_i-I_{n'_i}$ is an element  in   $M_{n'_i}(SA_{\De_j})$ for $j=1,2$;
\item $$\|w^{(1)*}_i\diag(I_{k_i},0)w^{(1)}_i-q'_i\|<\la\la^2_0\al_c\eps$$ and $$\|w^{(2)}_i \diag(I_{k'_i},0)w^{(2)*}_i-q'_i\|<\la\la^2_0\al_c\eps.$$
 \item for $j=1,2$, then $w^{(j)}_i$ is connected to $I_{n'_i}$ by a homotopy
of $\la\la^2_0\al_c\eps$-$h_{\la^2_0\al_c\eps}r'$-unitaries  $(w^{(j)}_{i,t})_{t\in[0,1]}$ in $M_{n'_i}(\widetilde{SA})$ such  that $w^{(j)}_{i,t}-I_{n'_i}$ is in $M_{n'_i}(SA_{\De_j})$ for all $t$ in $[0,1]$. 
\item $u''_i(0)=u''_i(1)=u'''_i(0)=u'''_i(1)=I_{2n'_i}$;
\item $\|\diag(u'_i,u''_i)-w^{(1)}_iw^{(2)}_i\|<\la\la^2_0\al_c\eps$.
\end{itemize}

 \bigskip
 
 {\bf   Step 3:} 
 Let us set  $\la'=\la\la^2_0\al_c$ and for any $\eps$ in $(0,\frac{1}{4\la'})$ set $h'_\eps=h_{\la\la^2_0\al_c\eps}$.
 Consider the element of $O^{\la'\eps,h'_\eps r'}_1(A,B)$ 
  $$z_y=[u_1]_{\la'\eps,h'_\eps r'}\ts [p_1]+\ldots+[u_l]_{\la'\eps,h'_\eps r'}\ts [p_l]+[u''_1]_{\la'\eps,h'_\eps r'}\ts ([p'_1]-[p'_1(0)]) +\ldots+[u''_{l'}]_{\la'\eps,h'_\eps r'}\ts ([p'_{l'}]- [p'_{l'}(0)]).$$ If we set 
 \begin{equation*}W_{y}=\diag(W_{u_1,u'_1,p_1},\ldots, W_{u_l,u'_l,p_l},W'_{u''_1,{u'''_1},p'_1},\ldots, W'_{u''_{l'},{u'''_{l'}},p'_{l'}} ),\end{equation*}  
then $W_{y}$ is a  $\la'\eps$-$h'_\eps r'$-unitary in  $M_{2n_y}(C(\To_2,A\ts B))$.  Using Lemmas \ref{lemma-almost-closed} and \ref{cor-example-homotopy} and up to replacing  $(\la',h')$ by  $(12\la',2h')$, then for $i=1,\ldots l$ (resp. for $i=1,\ldots,l'$)  $\diag(u'_i,I_{k_i)}$ is homotopic to $\diag(u^*_i,I_{k_i})$
(resp. $\diag(u'''_i,I_{k'_i})$ is homotopic to $\diag({u''_i}^*,I_{k'_i})$) as a $\la'\eps$-$h'_\eps r'$-unitary in $M_{n_i}(A)$
(resp. in $M_{n'_i}(\widetilde{SA})$). Hence we deduce that
$[W_{y}]_{\al_\F\la'\eps,k_{\F,\la'\eps}h'_\eps r'}$ belongs to ${O'_1}^{\al_\F\la'\eps,k_{\F,\la'\eps}h'_\eps r'}(A,B)$  and 
in view of equations (\ref{equ-F}) and (\ref{equ-F-suspension}), we see that 
 $$[W_{y}]_{\al_\F\la'\eps,k_{\F,\la'\eps}h'_\eps r'}=F^{\la'\eps,h'_\eps r'}_{A,B,*}(z_y)$$  in ${O'_1}^{\al_\F\la'\eps,k_{\F,\la'\eps}h'_\eps r'}(A,B)$. 
 In the same way, if we set 
  \begin{equation*}W'_{y}=\diag(W_{u'_1,u_1,p_1},\ldots, W_{u'_l,u_l,p_l},W'_{u'''_1,{u''_1},p'_1},\ldots, W'_{u'''_{l'},{u''_{l'}},p'_{l'}} ),\end{equation*}  
then $W'_{y}$ is also an $\la'\eps$-$h'_\eps r'$-unitary in  $M_{2n_y}(C(\To_2,A\ts B))$.

\bigskip
 
 {\bf   Step 4:}  For $i=1,\ldots,l$ and $j=1,2$, 
 let ${v'_i}^{(j)}$ be the matrix in $M_{n_i}(A)$ obtained from ${v_i}^{(j)}$ by flipping the $k_i$ first and the $k_i$-last coordinates, and define similarly ${w'_i}^{(j)}$ in  $M_{n'_i}(\widetilde{A})$ for $i=1,\ldots,l'$ and $j=1,2$.
 Up to replacing  $\la'$ by $2\la'$, we have  that \begin{itemize}
 \item $v'^{(1)}_i$ and $v'^{(2)}_i$ are  $\la'\eps$-$h'_{\eps}r'$ unitaries in $M_{n_i}(A)$;  
\item $$\|{v'_i}^{(1)*}\diag(I_{k_i},0){v'_i}^{(1)}-(I_{n_i}-\tilde{q_i})\|<\la'\eps$$ and $$\|{v'_i}^{(2)}\diag(I_{n_i},0){v'_i}^{(2)*}-(I_{n_i}-\tilde{q_i})\|<\la'\eps.$$
\item $\|\diag(u'_i,u_i)-{v'_i}^{(1)}{v'_i}^{(2)}\|<\la'\eps$.
\end{itemize}
where $\tilde{q_i}$ is obtained from $q_i$ by flipping the $k_i$ first and the $k_i$ last coordinates.
Similarly,  for $i=1,\ldots l'$, we have 
\begin{itemize}
\item   ${w'_i}^{(1)}$ and ${w'_i}^{(2)}$ are two  $\la'\eps$-$h'_{\eps}r'$ unitaries in $M_{n'_i}(\widetilde{SA})$;
\item ${w'_i}^{(j)}-I_{n'_i}$ is an element  in   $M_{2n'_i}(SA_{\De_j})$ for $j=1,2$;
\item $$\|{w'_i}^{(1)*}\diag(I_{k'_i},0){w'_i}^{(1)}-(I_{n'_i}-\tilde{q'_i})\|<2\la'\eps$$ and $$\|{w'_i}^{(2)}\diag(I_{n'_i},0){w'_i}^{(2)*}-(I_{n'_i}-\tilde{q'_i})\|<\la'\eps.$$
\item $\|\diag(u''_i,u'_i)-{w'_i}^{(1)}{w'_i}^{(2)}\|<\la'\eps$.
\end{itemize}where $\tilde{q'_i}$ is obtained from $q'_i$ by flipping  the $k'_i$ first and the $k'_i$ last coordinates.
Then we have  $$\|\diag(W_{u_i,u'_i,p_l},W_{u'_i,u_i,p_i})-W_{{v_i}^{(1)},{v'_i}^{(1)},p_i}W_{{v_i}^{(2)},{v'_i}^{(2)},p_i}\|<\la'\eps$$
and 
$$\|W^*_{{v_i}^{(1)},{v'_i}^{(1)},p_i}\cdot\diag(P_{k_i},P_{k_i})  \cdot W_{{v_i}^{(1)},{v'_i}^{(1)},p_i}-E'_{q_i,k_i,p_i\ts P_{Bott},p_i\ts P_1}\|<\la'\eps,$$

with $$E'_{{q_i,k_i,p_i\ts P_{Bott},p_i\ts P_1}}=\diag(q_i\ts p_i\ts P_{Bott}+P_{k_i}\ts (I_2-p_i\ts P_{Bott}),P_{k_i}\ts (I_2-p_i\ts P_{1})+(I_{n_i}-\tilde{q_i})\ts p_i\ts P_{1})$$
for $i=1\ldots l$. Similarly, we have 
 $$\|\diag(W'_{u''_i,u'''_i,p'_i},W_{u'''_i,u''_i,p'_i})-W_{{w_i}^{(1)},{w'_i}^{(1)},p'_i},W_{{w_i}^{(2)}{w'_i}^{(2)},p'_i})\|<\la'\eps,$$
 
 and 
 $$\|{W'}^*_{{w_i}^{(1)},{w'_i}^{(1)},p'_i}\cdot\diag(P_{k'_i},P_{k'_i})\cdot W'_{{w_i}^{(1)},{w'_i}^{(1)},p'_i}-E'_{q'_i,k'_i,p'_i,p'_i (0)}\|<\la'\eps,$$ for $i=1\ldots l'$ with 
  $$E'_{{q'_i,k'_i,p'_i,p'_i(0)\ts P_1}}=\diag(q'_i\ts p'_i+P'_{k'_i}\ts (I_2-p'_i),P_{k'_i}\ts (I_2-p'_i(0))+(I_{n'_i}-\tilde{q'_i})\ts p'_i(0))).$$ 
 
 From this we deduce that    there exist $W^{(1)}_{y}$ and  $W^{(2)}_{y}$ two $\la'$-$h'_\eps r'$-unitaries in 
$M_{2n_y}(C(\To_2,A\ts B))$ such that 
\begin{itemize}
\item  $W^{(i)}_{y}-I_{2n_y}$ is in $M_{2n_x}(C(\To_2,A_{\De_i}\ts B))$ for $i=1,\ldots n$;
\item $\|\diag(W_{y},W'_{y})-W^{(1)}_{y}W^{(2)}_{y}\|<\la'\eps$;
\item $\|W^{(1)*}_{y}\diag(I_{n_x},0)W^{(1)}_{y}-{q}_y\|<\la'\eps$,
 \end{itemize}
 where $${q}_y=\diag(E'_{q_1,k_1,p_1\ts P_{Bott},p_1\ts P_1},\ldots E'_{q_l,k_l,p_l\ts P_{Bott},p_l\ts P_1},
E'_{q'_1,k'_1,p'_1,p'_1(0)},\ldots E_{q'_{l'},k'_l,p'_{l'},p'_{l'}(0)}).$$ 
Clearly $q_y$ is $\la'\eps$-$h'_\eps r'$-projection in $M_{2n_y}(C(\To_2,A\ts B))$  such that
 $$[{q}_y,n_y]_{\la'\eps,h'_\eps r'}=[{q'_y},n_y]_{\la'\eps,h'_\eps r'}=\iota^{\eps,\la'\eps,r,h'_\eps r'}_{\O',*}\circ\partial^{\eps,r}_{C(\To_2,\De_1\ts B),C(\To_2,\De_2\ts B)}(y)$$ and hence 
 $z_y,\,{q}_y,\,n_y,W_{y},\,W'_{y},\,W^{(1)}_{y}$ and $W^{(2)}_{y}$ satisfy the required conditions for some suitable control pair.
\end{proof}

{\it End of the proof of the $QS$-statement.} Let $(\la,h)$ be a control pair as in Proposition \ref{proposition-technic-QS}, let $\eps$ be a positive number in $(0,\frac{1}{4\la})$, let $y$ be an element in ${O'_1}^{\eps,r}(A,B)$ and let $r',\, z_y,\,{q}_y,\,n_y,W_{y},\,W'_{y},\,W^{(1)}_{y}$ and $W^{(2)}_{y}$ as in the proposition.
Let $u$ be an $\eps$-$r$ unitary in some $M_n(C(\To_2,A\ts B))$, let $u_1$ and $u_2$ be  $\al_c\eps$-$k_{c,\eps} r$-unitaries in some $M_{2n}(C(\To_2,A\ts B))$ and let $q$ be an $\al_c\eps$-$k_{c,\eps}r$-projection in $M_{2n}(C(\To_2,A\ts B))$
such that
\begin{itemize}
\item $u_i-I_{2n}$ is in $M_{2n}(C(\To_2,A\ts B)$ for $i=1,2$;
\item $\|\diag(u,u^*)-u_1u_2\|<\al_c\eps$;
\item $q-\diag(I_n,0)$ is in $M_{2n}(C(\To_2,(A_{\De_1}\cap A_{\De_2})\ts B))$
\item $\|q-v^*_1\diag(I_n,0)v_1\|<\al_c\eps$;
\item $\|q-v_2\diag(I_n,0)v^*_2\|<\al_c\eps$;
\item $-y=[u]_{\eps,r}$;
\item $\partial^{\eps,r}_{C(\To_2,\De_1\ts B),C(\To_2,\De_2\ts B)}(-y)=[q,n]_{\al_c\eps,k_{c,\eps} r}$.
\end{itemize}
Then applying Lemma \ref{lemma-bound1} to $\diag(u,W_{y}),\,\diag(u^*,W'_{y})$ and to the   matrices respectively obtained from $\diag(u_1,W^{(1)}_{y}),\,\diag(u_2,W^{(2)}_{y})$ and $\diag(q,q_y)$ by swapping the order of coordinates $n+1,\ldots,2n$ and $2n+1,\ldots,2n+n_x$, we see that for a controlled pair $(\la',h')$ depending only on $\la_0$ and $c$, and if $\eps$ is in $(0,\frac{1}{4\la'})$,
there exist $U_1$ and $U_2$ some $\la'\eps$-$h'_\eps r'$-unitary in some $M_{n'}(C(\To_2,A\ts B))$ with
$U_1-I_{n'}$ in $M_{n'}(C(\To_2,A_{\De_1}\ts B))$ and $U_2-I_{n'}$ in $M_{n'}(C(\To_2,A_{\De_2}\ts B))$ such that
$$[U_1]_{\la'\eps,h'_\eps r'}+[U_2]_{\la'\eps,h'_\eps r'}=[W_{y}]_{\la'\eps,h'_\eps r'}-\iota_{\O',*}^{\eps,\la'\eps,r,h'_\eps r'}(y)$$     in $K_1^{\la',h'_\eps r'}(C(\To_2,A\ts B))$.
Up to replacing
$U_j$ for $j=1,2$ by $$\To_2\lto M_{3n'};(z_1,z_2)\mapsto \diag(U^*_j(1,1)U_j(z_1,z_2),U^*_j(z_1,1)U_j(1,1),U^*_j(1,z_2)U_j(1,1)),$$
and $(\la',h')$ by $(3\la',2h')$, we can assume that
$[U_j]_{\la'\eps,h'_\eps r'}$ belongs to ${O'_1}^{\la'\eps,h'_\eps r'}(A_{\De_j},B)$.

\medskip

 Let $r''$ be a positive number with
 $k_{\F,\la_0\la'\eps} r''\gq  h'_\eps r'$ such that  
 for $j=1,2$, any positive number $\eps$ in   $(0,\frac{1}{4\la_0\la'})$    and any $z$ in  ${O'_*}^{\la'\eps,h'_\eps r'}(A_{\De_j},B)$, there exists an element $x$ in
 ${O_*}^{\la_0\la'\eps,r''}(A_{\De_j},B)$ such that
 $$\iota_{\O',*}^{\la'\eps,\la_0\al_\F\la'\eps,h'_\eps r',k_{\F,\la_0\la'\eps}r''}(z)=F^{\la_0\la'\eps,r''}_{A_{\De_j},B,*}(x).$$Let then    $z_{y}^{(j)}$ be for $j=1,2$ an element in  ${O_1}^{\la_0\la'\eps,r''}(A_{\De_j},B)$ such that 
$$\iota_{\O',*}^{\la'\eps,\al_\F\la_0\la'\eps,h'_\eps r',k_{\F,\la_0\la'\eps}r''}([U_i]_{\la'\eps,h'_\eps r'})=F^{\la_0\la'\eps,r''}_{A_{\De_j},B,*}(z_{y}^{(j)}).$$ Let us set 
$$\tilde{z}_y=\iota_{\O,*}^{\la'\eps,\la_0\la'\eps,h'_\eps r'',h'_\eps R}(z_y)-\jmath^{\la_0\la'\eps,r''}_{\De_1,\O,*}(z_{y}^{(1)})-\jmath^{\la_0\la'\eps,r''}_{\De_2,\O,*}(z_{y}^{(2)})$$ in $\O_1^{\la_0\la'\eps,h'_\eps r''}(A,B)$.  By naturality of $\F_{\bullet,B,*}$, we see then that 
$$F^{\la_0\la'\eps,r''}_{A,B,*}(\tilde{z}_y)=\iota_{\O,*}^{\eps,\al_\F\la_0\la'\eps,r,k_{\F,\la_0\la'\eps}r''}(y)$$ and hence the $QS$-condition is satisfied.

 \subsection{Quantitative K\"unneth formula for crossed-product $C^*$-algebras}\label{subsec-kunneth-BC}
We shall next discuss the connection between the Baum-Connes conjecture and the quantitative K\"unneth formula. The connection between the usual Baum-Connes conjecture and the K\"unneth formula was studied 
in \cite{ceo2}.

  Before proving Theorem \ref{thm-BC-Qkunneth}, recall that article \cite{ceo2}  introduced an equivariant analogue of the map
$\omega_{\bullet,\bullet,*}$ for the topological $K$-theory of a locally compact group $G$ (i.e., the left-hand side of the Baum-Connes assembly map). Let $A$ be a $G$-$C^*$-algebra and let $B$ be a $C^*$-algebra. The $C^*$-algebra $B$ can be viewed as a $G$-$C^*$-algebra with  the  trivial action of $G$ and we equip $A\ts B$ with the diagonal action.
Then the elements  in $K_*(B)$ can be viewed as element of $K_*^G(B)$. If $X$ is a $G$-proper space, the map
$$\omega^{G,X}_{A,B,*}:KK^G_*(C_0(X),A)\ts K_*(B)\to KK^G_*(C_0(X),A\ts B);\, x\ts y\mapsto x\ts\tau_A(y),$$  is compatible with inclusion of $G$-proper cocompact spaces and hence gives rise to a morphism $$\omega^{G,top}_{A,B,*}:K^{top}_*(G,A)\ts \K_*(B)\to K^{top}_*(G,A\ts B).$$
 \begin{theorem}\label{thm-BC-kunneth}
 Let $\Ga$ be a discrete group  and  let $A$ be a $\Ga$-$C^*$-algebra. Assume that
 for every finite  subgroup $F$ of $\Ga$, the $C^*$-algebra  $A\rtimes F$ satisfies the K\"unneth formula,
 Then for any $C^*$-algebra $B$ such that $K_*(B)$ is a free abelian group and any positive number $d$
  $$\omega^{\Ga,P_d(\Ga)}_{A,B,*}:KK^{\Ga}_*(C_0(P_d(\Ga)),A)\ts  K_*(B)\to KK^{\Ga}_*(C_0(P_d(\Ga)),A\ts B)$$ is an isomorphism.
  \end{theorem}
  \begin{proof}(Compare with the proof of \cite[Lemma 1.7]{ceo2})
 The action of $\Ga$ on $P_{r}(\Ga)$ is simplicial and up
 take a barycentric subdivision of $P_{d}(\Ga)$, we can assume that
 $P_{d}(\Ga)$ is a locally finite and finite dimension typed
 simplicial complex, equipped with a simplicial and type preserving
 action of $\Ga$. Let $Z_0,\cdots,Z_n$ be the skeleton decomposition
 of $P_d(\Ga)$. Then $Z_j$ is a simplicial complex of dimension $j$,
 locally finite and equipped with a proper, cocompact  and type
 preserving simplicial action of $\Ga$. Let us prove by induction on
 $j$ that $$\omega_{A,B,*}^{\Ga,Z_j}:KK^G_*(C_0(Z_j),A)\ts K_*(B)\to KK^G_*(C_0(Z_j),A\ts B)$$
 is an isomorphism. The $0$-skeletton $Z_0$ is a
finite union of orbits and thus, for $j=0$, it is enought to prove that
 $$\omega^{\Ga,\Ga/F}_{A,B,*}:KK^\Ga_*(C_0(\Ga/F),A)\ts K_*(B)\to KK^G_*(C_0(Z_j),A\ts B)$$ is an isomorphism when $F$ is a finite subgroup of
$\Ga$. Let us recall from \cite{oyono}  that for every
$C^*$-algebra $B$ equipped with an action of $\Ga$, there is a natural
restriction isomorphism $$\res^B_{F,\Ga,*}:KK^\Ga_*(\Ga/F,B)\lto
KK^F_*(\C,B)\cong K_*(B\rtimes F).$$ Moreover, we have    the following
commutative diagram
$$
\begin{CD}
 KK^\Ga_*(C_0(\Ga/F),A)\ts K_*(B)  @>\omega^{\Ga/F,\Ga}_{A,B,*}  >>KK^\Ga_*(C_0(\Ga/F),A\ts B) \\
        @V\res^{A}_{F,\Ga,*}VV            @VV\res^{A\ts B}_{F,\Ga,*} V\\
 K_*(A\rtimes F)\ts  K_*(B)   @>\omega_{A\rtimes F,B,*}>> K_*(A\rtimes F\ts B)
\end{CD}.
$$
The bottom row being by assumption  an isomorphism, the top row is then also an isomorphism.
Let us assume that we have proved that  $\omega^{Z_{j-1},\Ga}_{A,B,*}$ is an
isomorphism. Then the short exact sequence
$$0\lto C_0(Z_j\setminus Z_{j-1}) \lto C_0(Z_j)\lto C_0(Z_{j-1})\lto 0  $$
gives rise to an natural long exact sequence
$$\lto KK^\Ga_*(C_0(Z_{j-1}),\bullet)\lto KK^\Ga_*(C_0(Z_{j}),\bullet)\lto
KK^\Ga_*(C_0(Z_{j}\setminus Z_{j-1}),\bullet)\lto
KK^\Ga_{*+1}(C_0(Z_{j-1}),\bullet)$$ and thus by naturality and since    $K_*(B)$ is a free abelian group, we get a
commutative diagram
{\tiny{$$
\begin{CD}
 KK^\Ga_*(C_0(Z_{j-1}),A)\ts\cdots @>>> KK^\Ga_*(C_0(Z_{j}),A) \ts\cdots@>>>
KK^\Ga_*(C_0(Z_{j}\setminus Z_{j-1}),A)\cdots @>>>
KK^\Ga_{*+1}(C_0(Z_{j-1}),A )\ts\cdots \\
        @V\omega^{Z_{j-1},\Ga}_{A,B,*}VV
        @V\omega^{Z_{j},\Ga}_{A,B,*}VV @V\omega^{Z_{j}\setminus Z_{j-1},\Ga}_{A,B,*}VV  @V\omega^{Z_{j-1},\Ga}_{A,B,*} VV    \\
   KK^\Ga_*(C_0(Z_{j-1}),\cdots @>>> KK^\Ga_*(C_0(Z_{j}),\cdots @>>>
KK^\Ga_*(C_0(Z_{j}\setminus Z_{j-1}),\cdots @>>>
KK^\Ga_{*+1}(C_0(Z_{j-1})\cdots \end{CD},
$$}}

Let $\intsi_j$  be the interior of the standard $j$-simplex. Since the
action of $\Ga$ is type preserving, then $Z_{j}\setminus Z_{j-1}$ is
equivariantly homeomorphic  to $\intsi_j\times\Sigma_j$, where $\Sigma_j$ is the
set of center of $j$-simplices  of $Z_{j}$, and where $\Ga$ acts
trivially on $\intsi_j$. This identification, together with Bott
periodicity, provides a commutative diagram

$$
\begin{CD}
KK^\Ga_*(C_0(Z_{j}\setminus Z_{j-1}),A)\ts K_*(B) @>>>KK^\Ga_{*+1}(C_0(\Sigma_j),A)\ts K_*(B)\\
  @V\omega^{Z_{j}\setminus Z_{j-1},\Ga}_{A,B,*} VV
        @V\omega^{\Si_j,\Ga,*}_{A,B,*}VV\\
KK^\Ga_*(C_0(Z_{j}\setminus Z_{j-1}),A\ts B) @>>>
KK^\Ga_{*+1}(C_0(\Sigma_j),A\ts B) \end{CD}.
$$

By the first step of induction, $\omega^{\Si_j,\Ga}_{A,B,*}$ is an isomorphism,
and hence  $\omega^{Z_{j}\setminus Z_{j-1},\Ga}_{A,B,*} $ is an
isomorphism. Using the induction hypothesis and the five lemma, we conclude
  that $\omega^{Z_{j},\Ga}_{A,B,*}$ is an isomorphism.\end{proof}

  \begin{lemma}\label{lemma-kunneth} There exits a positive number $\la_0$ and a
  function $$(0,+\infty)\times\left(0,\frac{1}{4\la_0}\right);\,(d,\eps)\mapsto r'_{d,\eps}$$ non decreasing in $d$ and non increasing in
  $\eps$ with $r_{d,\eps}\lq r'_{d,\eps}$ for all $\eps$ in $(0,\frac{1}{4\la_0})$ and $d>1$ such that the following holds:

  \medskip
   for any finitely generated group $\Ga$, any $\Gamma$-$C^*$-algebra $A$, any $C^*$-algebra $B$  and any positive numbers $\eps$, $r$ and
$d$ with $\eps<\frac{1}{4\la_0}$ and $r\gq r'_{d,\eps}$, then we have
$$\omega^{\eps,r}_{A\rt\Ga,B,*}\circ  (\mu_{\Gamma,A,*}^{\eps,r,d}\ts \Id_{K_*(B)})=
\mu_{\Gamma,A\ts B,*}^{\al_\T\eps,k_{\T,\eps}r,d}\circ \omega_{A,B,*}^{\Ga,P_d(\Ga)}.$$
\begin{proof}



Let $z$  be an element in $KK_*^\Ga(C_0(P_d(\Ga)),A)$ and let $y$ be an element in $K_*(B)$.  Then
\begin{equation*}\begin{split}
\omega^{\eps,r}_{A\rt\Ga,B,*}\left( J_\Gamma^{red,\frac{\eps}{\alpha_J},\frac{r}{k_{J,{\eps}/{\alpha_J}}}}(z)\left([p_{\Ga,d},0]_{\frac{\eps}{\alpha_J},\frac{r}{k_{J,{\eps}/{\alpha_J}}}}\right)\ts y\right)&\\
=\tau_{A\rtr,*}^{\eps,r}(y)\circ J_\Gamma^{red,\frac{\eps}{\alpha_J},\frac{r}{k_{J,{\eps}/{\alpha_J}}}}&(z)\left([p_{\Ga,d},0]_{\frac{\eps}{\alpha_J},\frac{r}{k_{J,{\eps}/{\alpha_J}}}}\right).\end{split}\end{equation*} The result is then a consequence of Remark \ref{remark-kasp-tens}.


\end{proof}
%
\end{lemma}
\medskip\ \\
{\it Proof of Theorem \ref{thm-BC-Qkunneth}.} Let $\lambda_0$ and  $(0,+\infty)\times\left(0,\frac{1}{4\la_0}\right);\,(d,\eps)\mapsto r'_{d,\eps}$ as in  Lemma \ref{lemma-kunneth}.
Let
 $\alpha_0$  be a positive number as in as in Theorem \ref{thm-quant-surj}, let $\eps$ and $r$ be positive numbers with $\eps<\frac{1}{4\lambda_0\alpha_0\al_\T}$.
 Let   $d$ and $R$ be  positive numbers with   $R\gq r'_{d,\eps}$ such that
 $QS_{\Ga,A}(d,\eps,\al_0\eps,r,R)$ is satisfied for every $\Ga$-$C^*$-algebra $A$. Let $d'$ be a positive number such that
 $QI_{\Ga,A}(d,d',\al_\T\al_0\eps,k_{\T,\al_0\eps}R)$ is satisfied for every $\Ga$-$C^*$-algebra $A$.
  \begin{itemize}
\item  Let $y$ be an element in $K_*^{\eps,r}((A\rtr\Ga)\ts B)$. Since $(A\rtr\Ga)\ts B\cong (A\ts B)\rtr\Ga$ and
 $QS_{\Ga,A\ts B }(d,\eps,\al_0\eps,r,R)$ is satisfied, there exists an  element $z_1$ in
 $KK_*^\Ga(C_0(P_d(\Ga)),A\ts B)$ such
 that $$\iota_*^{\eps,\al_0\eps,r,R}(y)=\mu^{\al_0\eps,R,d}_{\Ga,A\ts B,*}(z_1).$$ According to Theorem \ref{thm-BC-kunneth} there exists $z_0$ in $KK_*^\Ga(C_0(P_d(\Ga)),A)\ts K_*(B)$ such that
 $z_1=\omega_{A,B,*}^{\Ga,P_d(\Ga)}(z_0)$. Now if we set $\eps'=\al_0\lambda_0\eps$ and $r'=k_{\T,\al_0\eps}R$, then
 $x=\mu_{\Ga,A,*}^{\eps',r',d}\ts \Id_{K_*(B)}(z_0)$ is in $K^{\eps',r'}_{*}(A\rtr\Ga)\ts K_*(B)$ and from  Lemma
\ref{lemma-kunneth} we deduce that $$\iota_*^{\eps,\al_\T\eps',r,k_{\T,\eps'}r'}(y)=\omega^{\eps',r'}_{A\rt\Ga,B,*}(x).$$
\item Let $x$ be an element in $K^{\eps,r}_{*}(  A\rtimes_{red} \Gamma  )\ts K_*(B)$ such that $\omega^{\eps,r}_{A\rt\Ga,B,*}(x)=0$ in
$K^{\al_\T\eps,k_{\T,\eps}r}_{*}((A\rtr\Ga)\ts  B)$, let $z_0$ be an element in $KK_*^\Ga(C_0(P_d(\Ga)),A)\ts K_*(B)$ such that $$\iota_*^{\eps,\al_0\eps,r,R}\ts \Id_{K_*(B)}(x)=\mu_{\Ga,A,*}^{\al_0\eps,R,d}\ts \Id_{K_*(B)}(z_0)$$ and let us set
$$z_1=\omega_{A,B,*}^{\Ga,P_d(\Ga)}(z_0).$$ According to Lemma \ref{lemma-kunneth}, we  have that
$$\mu^{\al_\T\al_0\eps,k_{\T,\al_0\eps}R,d}_{\Ga,A\ts B,*}(z_1)=0$$ in $K_*^{\al_\T\al_0\eps,k_{\T,\al_0\eps}R}((A\rtr\Ga)\ts B)$  and hence since $QI_{\Ga,A\ts B}(d,d',\al_\T\al_0\eps,k_{\T,\al_0\eps}R)$ is satisfied, we have
$q_{d,d',*}(z_1)=0$ in $KK_*^\Ga(C_0(P_{d'}(\Ga)),A\ts B)$. According to Theorem \ref{thm-BC-kunneth} and since
$\omega_{A,B,*}^{\Ga,P_d(\Ga)}$ is compatible with inclusion $$P_d(\Ga)\hookrightarrow P_{d'}(\Ga),$$ we deduce that
$q_{d,d',*}(z_0)=0$ in  $KK_*^\Ga(C_0(P_{d'}(\Ga)),A)\ts K_*(B)$. Set $\eps'=\al_0\eps$ and pick any positive number $r'$ such that $r'\gq R$ and $r'\gq r_{d',\al_0\eps}$. Then we have
\begin{eqnarray*}
\iota_*^{\eps,\eps',r,r'}\ts \Id_{K_*(B)}(x)&=&
(\iota_*^{\al_0\eps,R,r'}\ts \Id_{K_*(B)})\circ \mu_{\Ga,A,*}^{\al_0\eps,R,d'}\ts \Id_{K_*(B)}(z_0)\\&=&0.\end{eqnarray*}
\end{itemize}
\qed

\section{$C^*$-algebras with finite  asymptotic nuclear  decomposition and quantitative K\"unneth formula}
In this section, we introduce the concept of finite asymptotic nuclear  decomposition for filtered $C^*$-algebras and and we prove for this class of $C^*$-algebras the quantitative  K\"unneth formula. We deduce from  this that uniform Roe algebras of  discrete metric spaces with bounded geometry and finite  asymptotic dimension satisfy the K\"unneth formula. \subsection{Locally  bootstrap $C^*$-algebras} Let us first recall the definition of the bootstrap category.
\begin{definition}
The bootstrap category $\mathcal{N}$  is  the smallest class of nuclear separable $C^\ast$-algebras  such that
\begin{enumerate}
\item $\mathcal{N}$ contains $\C$;
\item $\mathcal{N}$ is closed under countable inductive limits;
\item $\mathcal{N}$ is stable under extension, i.e for any extension of $C^*$-algebras $$0\to J\to A\to A/J\to 0,$$ if any two of the $C^*$-algebras are in $\mathcal{N}$ then so is the third;
\item $\mathcal{N}$ is closed under $KK$-equivalence.
\end{enumerate}
\end{definition}
Next we introduce the concept of locally bootstrap $C^*$-algebras.
\begin{definition} A filtered $C^*$-algebra $A$ with filtration  $(A_r)_{r>0}$ is called locally bootstrap if for all
positive number $s$ there exists a positive number $r$ with $r\gq s$ and a sub-$C^*$-algebra $A^{(s)}$ of $A$ such that
\begin{itemize}
\item $A^{(s)}$ belongs to the bootstrap class;
\item $A_s\subseteq A^{(s)}\subseteq A_r$.
\end{itemize}
\end{definition}
\begin{proposition}\label{prop-locally-bootstrap}There exists a positive number $\la_0$ with $\la_0\gq 1$ such that any 
 locally bootstrap $C^*$-algebra  satisfies the quantitative K\"unneth formula with rescaling $\la_0$.
\end{proposition}
\begin{proof}
Let $\lambda_0$ be as in the  second part  of Proposition \ref{proposition-approximation-K-th} and let $B$ be a separable $C^*$-algebra with $K_*(B)$ free abelian.
\begin{itemize}
\item Let us prove first the  $QI_{\Omega_{A,B,*}}$-statement. Let $\eps$ and $s$ be positive numbers with $\eps<\frac{4}{\lambda_0\al_\tau}$. Let then $r$ be a positive number with $r\gq k_{\tau,\eps}s$ and let $A^{(s)}$ be a $C^*$-algebra such that $A^{(s)}$ belongs to the bootstrap class
and  $A_{k_{\tau,\eps}s}\subseteq A^{(s)}\subseteq A_r$. Then $A^{(s)}$ is filtered by $(A^{(s)}\cap A_{s'})_{s'>0}$ and the filtration is indeed finite, i.e.  $A^{(s)}\cap A_{s'}=A^{(s)}$ for any positive number $s'$ with $s'\gq r$   . Let us consider  the commutative diagram
\begin{equation*}\begin{CD}
K_*^{\eps,s}(A)\ts K_*(B)@>>>K_*^{\eps,s}(A^{(s)})\ts K_*(B)@>\iota_*^{\eps,s,r}\ts Id_{K*(B)}>>K_*^{\eps,r}(A^{(s)})\ts K_*(B)\\
      @V\omega_{A,B,*}^{\eps,s} VV  @V\omega_{A^{(s)},B,*}^{\eps,s} VV         @VV\omega_{A^{(s)} ,B,*}^{\eps,r}V \\
      K_*^{\al_\tau\eps,k_{\tau,\eps}s}(A\ts B)@>>>K_*^{\al_\tau\eps,k_{\tau,\eps}s}(A^{(s)}\ts B)@>\iota_*^{\al_\tau\eps,k_{\tau,\eps}s,k_{\tau,\eps}r}>>K_*^{\al_\tau\eps,k_{\tau,\eps}r}(A^{(s)}\ts B),\end{CD}\end{equation*}where the  left  bottom and left top maps  are induced by  the
inclusion
$A_{s'}\subseteq A^{(s)}_{s'}$ for any  $s'\lq   k_{\tau,\eps}s$. Let $x$ be an element in $K_*^{\eps,s}(A)\ts K_*(B)$ such that 
$\omega_{A,B,*}^{\eps,s}(x)=0$ and let then $y$ in $K_*^{\eps,r}(A^{(s)})\ts K_*(B)$ be the image of $x$ under the compositions of the top row. Then $\omega_{A^{(s)},B,*}^{\eps,r}(y)=0$ and hence  
$$\omega_{A^{(s)},B,*}\circ (\iota^{\eps,r}_*\ts Id_{K_*(B)})(y)=\iota_*^{\al_\tau\eps,k_{\tau,\eps}r}\circ \omega_{A^{(s)},B,*}^{\eps,r}(y)=0.$$
Since $A^{(s)}$  is in the bootstrap class, then $$\omega_{A^{(s)},B,*}:K_*(A^{(s)})\ts K_*(B)\to K_*(A\ts B)$$ is an isomorphism and hence
$(\iota^{\eps,r}_*\ts Id_{K_*(B)})(y)=0$ in $K_*(A^{(s)})\ts K_*(B)$. Since $K_*(B)$ is free abelian and according to Proposition 
 \ref{proposition-approximation-K-th}, there exists a positive number $r'$, with $r'\geq r$ such that
 $$(\iota^{\eps,\lambda_0\eps,r,r'}_*\ts Id_{K_*(B)})(y)=0$$ in $K_*^{\lambda_0\eps,r'}(A^{(s)})\ts K_*(B)$. But since 
 $A^{(s)}$ has propagation less than $r$, then   $(\iota^{\eps,\lambda_0\eps,r}_*\ts Id_{K_*(B)})(y)=0$ in 
  $K_*^{\lambda_0\eps,r}(A^{(s)})\ts K_*(B)$.
  Hence composing with  the map  $$K_*^{\lambda_0\eps,r}(A^{(s)})\ts K_*(B)\lto K_*^{\lambda_0\eps,r}(A)\ts K_*(B)$$ induced by the inclusion 
 $A^{(s)}\hookrightarrow A$, we get then that 
  
 $$(\iota^{\eps,\lambda_0\eps,s,r}_*\ts Id_{K_*(B)})(x)=0.$$ 
 \item Let us prove now the $QS_{\Omega_{A,B,*}}$-statement. Let $s$ and $\eps$ be positive numbers with $\eps<\frac{1}{4\lambda_0\al_\tau}$  and let $A^{(s)}$ be a $C^*$-algebra such that $A^{(s)}$ belongs to the bootstrap class
and  $A_{s}\subseteq A^{(s)}\subseteq A_r$.
 Let $z$ be an element in some  $K_*^{\eps,s}(A\ts B)$ and let $z'$ in $K_*^{\eps,s}(A^{(s)}\ts B)$ be the image of $z$ under the map  $$K_*^{\eps,s}(A\ts B)\to K_*^{\eps,s}(A^{(s)}\ts B)$$   induced by the inclusion
 $A_s \subseteq A^{(s)}$.
  Since $A^{(s)}$ is in the bootstrap class, there exists $y$ in $K_*(A^{(s)})\ts K_*(B)$ such that 
  $\iota^{\eps,s}_*(z')=\omega_{A^{(s)},B,*}(y)$  in  $K_*(A^{(s)}\ts B)$. Since any element of $A^{(s)}$ has propagation less than $r$, there exists  an element $x$ in $K^{\eps,r}_*(A^{(s)}\ts B)$ such that 
  $(\iota^{\eps,r}_*\ts  Id_{K_*(B)})(x)=y$ in $K_*(A^{(s)})\ts K_*( B)$.
 Since   $$\iota^{\al_\tau\eps,k_{\tau,\eps}r}_*\circ \omega^{\eps,r}_{A^{(s)},B,*}= \omega_{A^{(s)},B,*}\circ (\iota^{\eps,r}_*\ts  Id_{K_*(B)}),$$we get that 
 $ \omega^{\eps,r}_{A^{(s)},B,*}(x)$ and $\iota^{\eps,\al_\tau\eps,s,k_{\tau,\eps}r}_*(z')$ have same image under the map
 $$\iota_*^{\al_\tau\eps,k_{\tau,\eps}r}:K^{\al_\tau\eps,k_{\tau,\eps}r}_*(A^{(s)}\ts B)\lto K_*(A^{(s)}\ts B).$$
 Hence, according to 
 Proposition 
 \ref{proposition-approximation-K-th}, there exists a positive number $r'$, with $r'\geq k_{\tau,\eps}r$, such that 
  $$\iota^{\al_\tau\eps,\lambda_0\al_\tau\eps,k_{\tau,\eps}r,r'}_* \omega^{\eps,r}_{A^{(s)},B,*}(x)=\iota^{\eps,\lambda_0\al_\tau\eps,s,r'}_*(z').$$ 
  But since $A^{(s)}_{r}=A^{(s)}_{r''}$ for all $r''\gq r$ we get that
  $$\iota^{\al_\tau\eps,\lambda_0\al_\tau\eps,k_{\tau,\eps}r, k_{\tau,\la_0\eps}r  }_* \omega^{\eps,r}_{A^{(s)},B,*}(x)=\iota^{\eps,\lambda_0\al_\tau\eps,s,k_{\tau,\la_0\eps}r}_*(z').$$ Composing with the map
   $$K_*^{\al_\tau\lambda_0\eps,k_{\tau,\eps}r}(A^{(s)}\ts B)\lto K_*^{\al_\tau\lambda_0\eps,k_{\tau,\eps}r}(A \ts B)$$ induced by the inclusion 
 $A^{(s)}\hookrightarrow A$, we get then that   $$\omega^{\la_0\eps,r}_{A,B,*}(x')=\iota^{\eps,\lambda_0\al_\tau\eps,s,k_{\tau,\la_0\eps}r}_*(z),$$ where $x'$ is the image of $\iota_{*}^{\eps,\la_0\eps,r}\ts Id_{K_*(B)}(x)$ under the composition
  $$K_*^{\lambda_0\eps,r}(A^{(s)}\ts B)\lto K_*^{\lambda_0\eps,r}(A \ts B)$$ induced by the inclusion 
 $A^{(s)}\hookrightarrow A$.
\end{itemize}
\end{proof}

We will need a uniform version of Proposition \ref{prop-locally-bootstrap}.
\begin{definition} A family of filtered $C^*$-algebras  $(A_i)_{i\in\N}$  is uniformly  locally bootstrap if for all integer $i$ 
and for all positive number $s$, there exist a positive number $r$ with $r\gq s$ and a sub-$C^*$-algebra $A_i^{(s)}$ of $A_i$ such that for all integer $i$,
\begin{itemize}
\item $A_i^{(s)}$ belongs to the bootstrap class;
\item $A_{i,s}\subseteq A_i^{(s)}\subseteq A_{i,r}$
\end{itemize}
($A_i$ being filtered by $(A_{i,r})_{r>0}$).
\end{definition}

Proposition \ref{prop-locally-bootstrap} can be extended to uniformly locally bootstrap families of $C^*$-algebras.
\begin{proposition}\label{prop-uniformly-locally-bootstrap}There exists a positive number $\la_0$ with $\la_0\gq 1$ such that any uniformly locally bootstrap family
 $(A_i)_{i\in\N}$ of  filtered  $C^*$-algebras and any $C^*$-algebra $B$ with $K_*(B)$-free abelian  then  
 $$(\Omega_{A_i,B,*}:\K_*(A_i)\ts K_*(B)\lto\K_*(A_i\ts B))_{i\in\N}$$ is  a uniform family of quantitative isomorphisms with rescalling $\la_0$ 
 \end{proposition}
 
 \subsection{Finite asymptotic nuclear  decomposition}
Let us define $\CC^{(0)}_{fand}$ as the class of uniformly locally bootstrap families of $C^*$-algebras.
Then we define by  induction $\CC^{(n)}_{fand}$ as the class of family $\A^{(1)}$ for which there exists a positive number $c$ such that for every positive number $r$, the following is satisfied: 

\smallskip
there exists a family $\A^{(2)}$ in $\CC^{(n-1)}_{fand}$ and 
 for any $C^*$-algebra $A$ in $\A^{(1)}$  
 an  $r$-controlled nuclear Mayer-Vietoris pair $(\De_1,\De_2,A_{\De_1},A_{\De_2})$ with coercitivity $c$  for $A$ with 
 $A_{\De_1},\,A_{\De_2}$ and $A_{\De_1}\cap A_{\De_2}$ in  $\A^{(2)}$.
 \smallskip
 
 Define then $\CC_{fand}$ as the class of families $\A$ such that $\A$ is in  $\CC^{(n)}_{fand}$ for some integer $n$.
Theorem \ref{thm-mv-kunneth} obviously admits a uniform version for families and hence, together with 
  Proposition \ref{prop-locally-bootstrap}, we obtain the following result.
  \begin{proposition}\label{prop-fnc} Let $\A$ be a family in $\CC_{fand}$. Then there    exists a positive number $\la_\A$ with $\la_\A\gq 1$  such that for any $C^*$-algebra $B$ with $K_*(B)$ free abelian, then  $$(\Omega_{A,B,*}:\K_*(A)\ts K_*(B)\lto\K_*(A\ts B))_{A\in \A}$$ is a uniform family of quantitative isomorphisms with rescalling  $\la_\A$ (indeed $\A$ only depends on $n$ such that $\A$ lies in $\CC^{(n)}_{fand}$. 
  \end{proposition}

  \begin{definition}
 A filtered $C ^*$-algebra $A$  has finite asymptotic nuclear decomposition  if the single family $\{A\}$ is in  $\CC_{fand}$. \end{definition}
 
 As a consequence of Proposition  \ref{prop-fnc}, we obtain

\begin{theorem}\label{theorem-kunneth}  If $A$ is a filtered $C^*$-algebra with finite asymptotic nuclear decomposition, then the quantitative  K\"unneth formula holds for $A$.
\end{theorem}

\begin{corollary}   If $A$ is a filtered $C^*$-algebra with  finite asymptotic nuclear decomposition, then $A$ satisfies  the K\"unneth formula in $K$-theory,  i.e. there exists a natural short exact sequence

$$0\rightarrow K_\ast(A) \otimes K_\ast(B) \rightarrow K_\ast(A\otimes B) \rightarrow Tor (K_\ast(A), K_\ast(B))\rightarrow 0$$ for any other $C^\ast$-algebra $B$.

\end{corollary}
%
%
%
%

Typical examples of family of filtered $C^*$-algebra in   $\CC_{fand}$ are provided by  spaces with  asymptotic dimension.
Recall that for a metric space $X$ and a positive number $r$, a cover $(U_i)_{i\in\N}$ has $r$-multiplicity $n$ if any ball of radius $r$ in $X$ intersects at most $n$ elements in $(U_i)_{i\in\N}$.

\begin{definition}
Let $\Sigma$ be a proper discrete metric space.  Then $\Sigma$ has {\bf finite asymptotic dimension} if there exists an integer $m$ such that for any positive number  $r$, there exists a  uniformly bounded cover $(U_i)_{i\in\N}$ with finite $r$-multiplicity $m+1$.  The smallest integer that satisfies this condition is called the {\bf asymptotic dimension} of $\Si$.
\end{definition}
Recall the following characterization of finite asymptotic dimension.
\begin{proposition}\label{proposition-finite-asymptotic-dim}
Let $\Si$ be a proper discrete metric space and let $m$ be an integer. Then the following assertions are equivalent:
\begin{enumerate}
\item $\Si$ has asymptotic dimension $m$;
\item For every positive number $r$ there exist $m+1$ subsets $X^{(1)},\ldots,X^{(m+1)}$   of $\Si$ such that
\begin{itemize}
\item $\Si=X^{(1)}\cup\ldots\cup  X^{(m+1)}$;
\item for $i=1,\ldots m+1$, then $X^{(i)}$ is the $r$-disjoint union of a family  $(X^{(i)}_k)_{k\in\N}$ of subsets of $X^{(i)}$ with uniformly bounded diameter, i.e  $X^{(i)}=\cup_{k\in\N}X^{(i)}_k,\, d_i(X^{(i)}_k,X^{(i)}_l)\gq r$ if $k\neq l$ and there exists a positive number $C$ such $diam\, X^{(i)}_k\lq C$ for all integer $k$.
\end{itemize}
\end{enumerate}
\end{proposition}
\begin{example}\label{example-tree}
If $T$ is a tree, then $T$ has asymptotic dimension equal to $1$.
\end{example}
Let $\Si$ be a proper metric space with  asymptotic dimension $m$, then there exists a sequence of positive numbers $(R_k)_{k\in\N}$  and for any  integer $k$ a cover $(U^{(k)}_i)_{i\in\N}$ of $\Si$ such that
\begin{itemize}
\item $R_{k+1}>4R_k$ for every  integer $k$;
\item  $U^{(k)}_i$ has diameter less than $R_k$ for every integers $i$ and $k$;
\item for any integer $k$, the $R_k$-multiplicity of $(U^{(k+1)}_i)_{i\in\N}$ is $m+1$.
\end{itemize}
The sequence $(R_k)_{k\in\N}$ is called {\bf the $\mathbf{m}$-growth} of $\Si$.
\begin{lemma}\label{lemma-afd}
Let  $m$ be an integer, let $(R_k)_{k\in\N}$ be a sequence of positive numbers such that  $R_{k+1}>4R _k$ for every   integer $k$.
Let $(\Si_i)_{i\in\N}$ be a family of proper metric spaces 
 with asymptotic dimension $m$ and $m$-growth  $(R_k)_{k\in\N}$. Then  for any family $(A_i)_{i\in I}$ in the bootstrap category, the family   $(A_i\ts \Kp(\ell^2(\Si_i)))_{i\in\N}$
belongs to $\CC_{fand}$. 
\end{lemma}
\begin{proof}
 Let us equip $\Si=\coprod_{i\in\N}\Si_i$ with a distance $d_\Si$ such that the inclusion $\Si_i\hookrightarrow\Si$ are isometric for all integer $i$ and $d_\Si(\Si_i,\Si_j)\lq i+j$ for all integers $i$ and $j$ with $i\neq j$.
 Then $\Si$ has asymptotic dimension $m$ and hence according to \cite{dz}, the metric space $\Si$ embeds uniformly in a product 
 of trees $\prod_{j=1}^n T_j$. Let $d$ be the metric  on $X=\prod_{j=1}^n T_j$  and $d_i$ the distance on $\Si_i$ when $i$ runs through integers. Then there exists two non-decreasing functions
 $\rho_\pm:[0,+\infty)\to [0,+\infty)$   and for every integer $i$ a map $f_i:\Si_i\to \prod_{j=1}^n T_j$ such that
  \begin{itemize}
 \item $\lim_{r\mapsto +\infty}\rho_{\pm}(r)=+\infty$;
 \item  $\rho_-(d_i(x,y))\lq d(f_i(x),f_i(y))\lq \rho_+(d_i(x,y))$ for all integer $i$ and all $x$ and $y$ in $\Si_i$.
 \end{itemize}
 If $n=1$, then $X$ is a tree and then the result holds in view  of  Examples \ref{example-Roe-algebras},    \ref{example-roe} and  \ref{example-tree}. A straightforward induction shows that if $\Si$ embeds uniformly in a product of $n$ trees, then
  $(A_i\ts \Kp(\ell^2(\Si_i)))_{i\in\N}$ is in $\CC_{fand}$. \end{proof}
In order study the structure of Roe algebras we need to add  some infinite product decompositions in the  quantitative decomposition process.

\begin{definition}
Let $(A^{(i)})_{i\in I}$ be a family of filtered $C^*$-algebras. Then the uniform products of $(A^{(i)})_{i\in I}$, denoted by $\prod^u_{i\in I} A^{(i)}$ is the closure
of $$\{(x_i)_{i\in I}\in\prod_{i\in I}A^{(i)}_r,\, r>0\}$$ in $\prod_{i\in I} A^{(i)}$ equipped with the supremium norm.
The uniform product  $\prod^u_{i\in I} A^{(i)}$ is then obviously a filtered $C^*$-algebra.
\end{definition}

It was proved in \cite[Lemma 1.11]{oy3} that the quantitative $K$-theory of a uniform product of a stable filtered $C^*$-algebra is computable in term of the quantitative $K$-theory of the algebras of the family.

\begin{definition}
 A $C^*$-algebra  is said to be   of finite asymptotic nuclear $\pi$-decomposition  if there exists a positive number $c$ and an integer $n$ such that for any positive number $r$, there exists 
  a $r$-controlled Mayer-Vietoris pair $(\De_1,\De_2,A_{\De_1},A_{\De_2})$  that satisfies the following.
  
  \medskip
  There exist  three families of filtered $C^*$-algebras $(B^{(1)}_k)_{k\in\N}\,,(B^{(2)}_k)_{k\in\N}$ and $(B^{(1,2)}_k)_{k\in\N}$ in $\CC_{fand}$ such that
  $A_{\De_1},\,A_{\De_2}$ and $A_{\De_1}\cap A_{\De_2}$ are respectively isomorphic  as filtered $C^*$-algebras to $\prod^u_{k\in\N} B^{(1)}_k,\,\prod^u_{k\in\N} B^{(2)}_k$ and 
  $\prod^u_{k\in\N} B^{(1,2)}_k$.
  
  If the $C^*$-algebras in the families  $(B^{(1)}_k)_{k\in\N}\,,(B^{(2)}_k)_{k\in\N}$ and $(B^{(1,2)}_k)_{k\in\N}$ are stable, then $A$ is said to be 
   of   stably asymptotic finite nuclear $\pi$-decomposition.
  \end{definition}

  \begin{proposition}
  If $\Si$ is a proper metric set of bounded geometry and with finite asymptotic dimension.
  Then the uniform Roe algebra $C^u_*(\Si)$ has asymptotic finite nuclear $\pi$-decomposition  and the Roe algebra $C^*(\Si)$ has stably asymptotic finite  nuclear $\pi$-decomposition.
  \end{proposition}
  \begin{proof}Let us prove de result for $C^u_*(\Si)$, the proof for  $C^*(\Si)$ being similar.
  Let us fix $x_0$ in $\Si$ and let $r$ be a positive number. Let us fix $s$ and $R$ two positive numbers such that $10r<2s<R$.
  Set for $k$ integer $$X^{(1)}_k=\{x\in \Si\text{ such that }2kR\lq d(x,x_0)\lq (2k+1)R\}$$ and
  $$X^{(2)}_k=\{x\in \Si\text{ such that }(2k+1)R\lq d(x,x_0)\lq (2k+2)R\}.$$
  Then  $\Si=X^{(1)}\cup X^{(2)}$  and $X^{(i)}$ is for $i=1,2$ the $R$-disjoint union of the  family $(X^{(i)}_k)_{k\in\N}$.   Let  $\De_i$ be  for $i=1,2$ the set of element in $C^u_*(\Si)$ with support in $$\{(x,y)\in\Si\times\Si \text{ such that }d(x,y)<r \text{ and } x\in X^{(i)}\}.$$  Since $\Si$ has bounded geometry, then with notations of Example \ref{example-Roe-algebras}, there exists a controlled Mayer-Vietoris $(\De_1,\De_2,A_{\De_1},A_{\De_2})$ of order $r$ and coercitivity $1$ such that 
 $$A_{\De_i}\cong \prod^u_{k\in\N}\Kp(\ell^2(X^{(i)}_{k,s}))$$ for $i=1,2$ and 
 $$A_{\De_1}\cap A_{\De_2} \cong \prod^u_{(k,l)\in\N^2}\Kp(\ell^2(X^{(i)}_{k,s}\cap X^{(i)}_{l,s})).$$ The result is now a consequence of Lemma \ref{lemma-afd}.
 \end{proof}
 Proceeding similarly we can prove the quantitative K\"unneth formula for uniform Roe algebras of spaces with finite asymptotic dimension.
 \begin{theorem}
  If  $\Si$ is a discrete proper metric set of bounded geometry and with finite asymptotic dimension.
  Then the uniform Roe algebra $C^u_*(\Si)$ satisfies the quantitative K\"unneth formula for some rescaling $\la$.
  \end{theorem}

\bibliographystyle{plain}

 \end{document}